\documentclass[10pt]{amsart}
      \usepackage[mathscr]{eucal}
      \usepackage{amsmath,amsfonts}
      \parskip\smallskipamount
          \newtheorem{theorem}{Theorem}[section]
      \newtheorem{definition}[theorem]{Definition}
      \newtheorem{proposition}[theorem]{Proposition}
      \newtheorem{corollary}[theorem]{Corollary}
      \newtheorem{lemma}[theorem]{Lemma}

      \newtheorem{remark}[theorem]{Remark}
      \makeatletter
      \@addtoreset{equation}{section}
      \makeatother

      \newcommand{\CC}{{\mathbb C}}
      \newcommand{\NN}{{\mathbb N}}
      
      \newcommand{\ZZ}{{\mathbb Z}}

      \newcommand{\FF}{{\mathbb F}}
      \newcommand{\TT}{{\mathbb T}}

      \newcommand{\cA}{{\mathcal A}}

      \newcommand{\cD}{{\mathcal D}}
      \newcommand{\cE}{{\mathcal E}}
      
      \newcommand{\cG}{{\mathcal G}}
      \newcommand{\cH}{{\mathcal H}}
      
      \newcommand{\cJ}{{\mathcal J}}
      \newcommand{\cK}{{\mathcal K}}
      \newcommand{\cL}{{\mathcal L}}
      \newcommand{\cM}{{\mathcal M}}
       \newcommand{\cO}{{\mathcal O}}
      
      \newcommand{\cN}{{\mathcal N}}
      
      \newcommand{\cR}{{\mathcal R}}
      \newcommand{\cS}{{\mathcal S}}

      \newcommand{\cV}{{\mathcal V}}

      \newcommand{\cW}{{\mathcal W}}

      \newcommand{\rank}{\hbox{\rm{rank}}\,}

      \newdimen\expt
      \def\boxit#1{\setbox0\hbox{$\displaystyle{#1}$}
            \hbox{\lower.4\expt
       \hbox{\lower3\expt\hbox{\lower\dp0
            \hbox{\vbox{\hrule height.4\expt
       \hbox{\vrule width.4\expt\hskip3\expt
            \vbox{\vskip3\expt\box0\vskip2\expt}%
       \hskip3\expt\vrule width.4\expt}\hrule height.4\expt}}}}}}
      
\begin{document}
       \pagestyle{myheadings}
      \markboth{ Gelu Popescu}{ Doubly $\Lambda$-commuting row isometries, universal models, and classification }

      \title [  Doubly $\Lambda$-commuting row isometries, universal models, and classification ]
      { Doubly $\Lambda$-commuting row isometries, universal models, and classification
      }
        \author{Gelu Popescu}
\date{October 14, 2019}
     \thanks{Research supported in part by  NSF grant DMS 1500922}
       \subjclass[2000]{Primary: 47A45;  46L06,    Secondary: 47A20; 47A15; 46L65.
   }
      \keywords{Multivariable operator theory,    Wold decomposition, $\Lambda$-commuting row isometries, Cuntz  algebra, twisted tensor algebra, von Neumann inequality, invariant subspace, dilation theory}
      \address{Department of Mathematics, The University of Texas
      at San Antonio \\ San Antonio, TX 78249, USA}
      \email{\tt gelu.popescu@utsa.edu}

\begin{abstract}
The goal of the paper is to  study the structure of  the $k$-tuples of doubly $\Lambda$-commuting row isometries and the $C^*$-algebras they generate from the point of view of   noncommutative multivariable operator theory.  We obtain  Wold decompositions, in this setting, and use them to  classify  the $k$-tuples of doubly $\Lambda$-commuting row isometries  up to a unitary equivalence. We prove that
there is a one-to-one correspondence  between the unitary equivalence classes of $k$-tuples of doubly $\Lambda$-commuting row isometries and the enumerations of  $2^k$ unitary equivalence classes of unital representations of the  twisted $\Lambda$-tensor algebras $\otimes_{i\in A^c}^{\Lambda}\cO_{n_i}$,  as $A$ is any subset of $\{1,\ldots, k\}$, where $\cO_{n_i}$ is the Cuntz algebra with $n_i$ generators. In addition, we obtain  a description and parametrization of the irreducible $k$-tuples of doubly $\Lambda$-commuting row isometries.

We introduce  the  standard $k$-tuple  $S:=(S_1,\ldots, S_k)$ of  doubly
$\Lambda$-commuting   pure row isometries  $S_i:=[S_{i,1}\cdots  S_{i,n_i}]$ acting on the Hilbert space $\ell^2(\FF_{n_1}^+\times \cdots \times \FF_{n_k}^+)$,  where $\FF_n^+$ is the unital free semigroup with $n$ generators,  and
 prove that  the universal  $C^*$-algebra generated by   a $k$-tuple of  doubly
$\Lambda$-commuting  row isometries is $*$-isomorphic to the $C^*$-algebra $C^*(\{S_{i,s}\})$.
We introduce the regular $\Lambda$-polyball $B_\Lambda(\cH)$ and show that
a $k$-tuple  $T:=(T_1,\ldots, T_k)$ of row operators $T_i:=[T_{i,1}\ldots  T_{i,n_i}]$, acting on   $\cH$, admits $S$ as universal model, i.e.
 there is a Hilbert space $\cD$  such that  $\cH$ is jointly co-invariant for  ${S}_{i,s}\otimes I_\cD$ and
 $$
 T_{i,s}^*=(S^*_{i,s}\otimes I_\cD)|_\cH,\quad i\in \{1,\ldots, k\} \ \text{ and } \  s\in \{1,\ldots, n_i\},
 $$
   if and only if  $T$ is a pure element of
$B_\Lambda(\cH)$. This leads to  von type Neumann inequalities  and  the introduction of the  noncommutative Berezin transform associated with the elements of the regular $\Lambda$-polyball, which plays an important role.
We use the Berezin kernel to  obtain a characterization of the Beurling type  jointly invariant subspaces of   the  standard $k$-tuple  $S=(S_1,\ldots, S_k)$ and  provide a  classification result for  the pure elements    in the regular $\Lambda$-polyball.

We show that any $k$-tuple  $T$ in the regular $\Lambda$-polyball admits  a minimal  dilation which is a $k$-tuple of doubly $\Lambda$-commuting row isometries and is  uniquely determined up to an isomorphism. Using the Wold decompositions obtained, we show that $T$ is a pure element in ${\bf B}_\Lambda(\cH)$ if and only if    its minimal dilation  is a pure element in
 ${\bf B}_\Lambda(\cK)$.
 In  the particular case when $n_1=\cdots =n_k=1$, we obtain  an  extension of  Brehmer's  result, showing that any $k$-tuple in the regular $\Lambda$-polyball  admits a unique minimal
  doubly $\Lambda$-commuting unitary
    dilation.

\end{abstract}

      \maketitle

\bigskip

\section*{Contents}
{\it

\quad Introduction

\begin{enumerate}
   \item[1.]   Doubly $\Lambda$-commuting row isometries and Wold decompositions
   \item[2.]  Standard  $k$-tuples of doubly $\Lambda$-commuting row isometries
 \item[3.]   Classification of doubly $\Lambda$-commuting row isometries
\item[4.]   Regular $\Lambda$-polyballs, noncommutative Berezin transforms, and von Neumann inequalities
\item[5.]   Invariant subspaces and classification of the pure elements in the regular  $\Lambda$-polyball
\item[6.]   Dilation theory on regular $\Lambda$-polyballs
   \end{enumerate}

\quad References

}

\bigskip

\section*{Introduction}

For each $i,j\in \{1,\ldots, k\}$ with $i\neq j$, let $\Lambda_{ij}:=[\lambda_{i,j}(s,t)]_{n_i\times n_j}$  be an $n_i\times n_j$-matrix  with  entries in the torus $\TT:=\{z\in \CC: \  |z|=1\}$ and assume that $\Lambda_{j,i}=\Lambda_{i,j}^*$.     Given row isometries  $V_i:=[V_{i,1}\cdots V_{i,n_i}]$,      i.e. $V_{i,s} V_{i,t}^*=\delta_{st}I$, we say that  $V:=(V_1,\ldots, V_k)$  is a  $k$-tuple of {\it doubly $\Lambda$-commuting row isometries}  if
\begin{equation*}
V_{i,s}^* V_{j,t}=\overline{\lambda_{i,j}(s,t)}V_{j,t}V_{i,s}^*
\end{equation*}
for any $i,j\in \{1,\ldots, k\}$ with $i\neq j$ and any $s\in \{1,\ldots, n_i\}$, $t\in \{1,\ldots, n_j\}$.
We note   that  in the particular case   when   $n_1=\cdots=n_k=1$ and $\Lambda_{i,j}=1$, the $k$-tuple  $(V_1,\ldots, V_k)$  consists of  doubly commuting isometries in the sense of Brehmer \cite{Br}.

If $C\subset \{1,\ldots, k\}$, we denote by $\otimes_{i\in C}^\Lambda\cO_{n_i}$ the universal $C^*$-algebra generated  by   isometries $V_{i,s}$, where $i\in C$, $s\in \{1,\ldots, n_i\}$, satisfying the Cuntz relations
$$
V_{i,1} V_{i,1}^*+\cdots +V_{i,n_i} V_{i,n_i}^*=I,\qquad i\in C,
$$
and  satisfying the $\Lambda$-commutation relation  above for any $i,j\in C$ with $i\neq j$ and any $s\in \{1,\ldots, n_i\}$, $t\in \{1,\ldots, n_j\}$. Note that if  $\lambda_{i,j}(s,t)=1$, then $\otimes_{i\in C}^\Lambda\cO_{n_i}$ coincides with the usual  tensor product of Cuntz algebras  $\otimes_{i\in C}\cO_{n_i}$ (see \cite{Cu}). In general, the algebra  $\otimes_{i\in C}^\Lambda\cO_{n_i}$ can be seen  as a twisted tensor product of Cuntz  algebras. We remark that, when $n_1=\cdots =n_k=1$, the corresponding algebras  are   higher-dimensional
noncommutative tori  which are  studied in  noncommutative differential geometry (see  \cite{R}, \cite{Con}, \cite{Dav},  and the  appropriate references there in). We should mention that  $C^*$-algebras generated by isometries with twisted commutation relations have been
studied in the literature  in various particular cases (see \cite{JPS},  \cite{Pr},      \cite{K}, and \cite{W}).

Our approach  to study the  $k$-tuples of  doubly $\Lambda$-commuting row isometries and the $C^*$-algebras (resp. non-self-adjoint algebras)  they generate  is from the noncommutative multivariable   operator theory point of view. Many of the techniques developed in \cite{Po-isometric}, \cite{Po-multi}, \cite{Po-von}, \cite{Po-poisson}, \cite{Po-Berezin1} are refined and used in the present paper.

Inspired by  the work of De Jeu and Pinto \cite{DJP} and J.~Sarkar \cite{S}, who considered the particular case when $n_1=\cdots=n_k=1$, we  obtain, in Section 1,  Wold decompositions for $k$-tuples of doubly $\Lambda$-commuting row isometries (see Theorem \ref{Wold2} and Theorem \ref{Wold3})  and use them to provide, in Section 3,   classification  and parametrization results     for  the $k$-tuples of doubly $\Lambda$-commuting row isometries.   More precisely,  we prove that
there is a one-to-one correspondence  between the unitary equivalence classes of $k$-tuples of doubly $\Lambda$-commuting row isometries and the enumerations of  $2^k$ unitary equivalence classes of unital representations of the  twisted $\Lambda$-tensor algebras $\otimes_{i\in A^c}^{\Lambda}\cO_{n_i}$,  as $A$ is any subset of $\{1,\ldots, k\}$, where $\cO_{n_i}$ is the Cuntz algebra with $n_i$ generators and $A^c$ is the complement of $A$ in $\{1,\ldots,k\}$. In addition, we obtain  a description and parametrization of the irreducible $k$-tuples of doubly $\Lambda$-commuting row isometries.

An important role in our investigation is played by the standard  $k$-tuples of doubly $\Lambda$-commuting row isometries which are discussed in Section 2. In particular,  we introduce the standard $k$-tuple $S:=(S_1,\ldots, S_k)$  of pure row isometries $S_i=[S_{i,1}\cdots S_{i.n_i}]$ acting on  the Hilbert space $\ell^2(\FF_{n_1}^+\times \cdots \times \FF_{n_k}^+)$,
  which plays a very  special  role.  Here $\FF_n^+$ denotes the unital free semigroup with $n$ generators. We prove that  the $C^*$-algebra $C^*(\{S_{i,s}\})$ is  $*$-isomorphic with the universal  $C^*$-algebra generated by   a $k$-tuple of  doubly
$\Lambda$-commuting  row isometries.  If   $V=(V_1,\ldots, V_k)$ is  a $k$-tuple of doubly $\Lambda$-commuting row isometries
and
$$\bigcap_{{i\in \{1,\ldots,k\}}\atop {s\in \{1,\ldots, n_i\}}} \ker V_{i,s}^*\neq \{0\},
$$
we show that  the $C^*$-algebra $C^*(\{V_{i,s}\})$ is isomorphic to $C^*(\{S_{i,s}\})$. This extends the  result obtained by  Coburn \cite{Co} in the single variable case. Moreover,
if $\cJ_\Lambda$ is the closed
two-sided ideal generated by the projections $I -\sum_{s=1}^{n_1} S_{1,s}S_{1,s}^*$, ..., $I-\sum_{s=1}^{n_k} S_{k,s} S_{k,s}^*$  in the $C^*$-algebra $C^*(\{S_{i,s}\})$, then we show that the sequence
of $C^*$-algebras
$$
0\to \cJ_\Lambda\to C^*(\{S_{i,s}\})\to \otimes_{i\in \{1,\ldots,k\}}^{\Lambda}\cO_{n_i}\to 0
$$
is exact.

In Section 2, we prove that the pure $k$-tuples of doubly  $\Lambda$-commuting  row isometries are  unitarily equivalent to  the standard $k$-tuples  ${S\otimes I_\cD}:=({ S_1\otimes I_\cD},\ldots, {S_k\otimes I_\cD})$ acting on $\ell^2(\FF_{n_1}^+\times\cdots \times  \FF_{n_k}^+)\otimes \cD$, where $\cD$ is a Hilbert space. A natural question that arises is the following. What are the $k$-tuples  $T=(T_1,\ldots, T_k)$ of row operators $T_i=[T_{i,1}\cdots  T_{i,n_i}]$, acting on a Hilbert space $\cH$, which admit ${S}$ as universal model, i.e.
 there is a Hilbert space $\cD$  such that  $\cH$ is jointly co-invariant for  ${S}_{i,s}\otimes I_\cD$ and
 $$
 T_{i,s}^*=(S^*_{i,s}\otimes I_\cD)|_\cH,
 $$
 for any $i\in \{1,\ldots, k\}$ and $s\in \{1,\ldots, n_i\}$.    We answer this question in Section 4.   Employing noncommutative Berezin transforms, we show that the relation above holds if and only if $T=(T_1,\ldots, T_k)$  is a pure   $k$-tuple   in  the  {\it regular $\Lambda$-polyball}  ${\bf B}_\Lambda(\cH)$, which is introduced as the set of all $k$-tuples of
   row contractions  $T_i=[T_{i,1}\ldots  T_{i,n_i}]$, i.e. $T_{i,1}T_{i,1}^*+\cdots +T_{i,n_i}T_{i,n_i}^*\leq I$,  such that
 \begin{equation*}
 T_{i,s} T_{j,t}=\lambda_{ij}(s,t)T_{j,t}T_{i,s}
\end{equation*}
for any $i,j\in \{1,\ldots, k\}$ with $i\neq j$ and any $s\in \{1,\ldots, n_i\}$, $t\in \{1,\ldots, n_j\}$,
and such that
$$
\Delta_{rT}(I):=
(id-\Phi_{rT_1})\circ\cdots \circ (id-\Phi_{rT_k})(I)\geq 0,\qquad r\in [0,1),
$$
where $\Phi_{rT_i}:B(\cH)\to B(\cH)$  is the completely positive linear map defined by $\Phi_{rT_i}(X):=\sum_{s=1}^{n_i} r^2T_{i,s}XT_{i,s}^*$.

 We prove  a noncommutative von Neumann inequality \cite{vN} in this setting. More precisely,
 if  $T=(T_1,\ldots, T_k)$ is a  $k$-tuple in the regular  $\Lambda$-polyball  ${\bf B}_\Lambda(\cH)$, then
$$
\|p(\{T_{i,s}\}, \{T_{i,s}^*\})\|\leq  \|p(\{S_{i,s}\}, \{S_{i,s}^*\})\|
$$
for any polynomial  $p(\{S_{i,s}\}, \{S_{i,s}^*\})$ of the form
\begin{equation*}
p(\{S_{i,s}\}, \{S_{i,s}^*\})= \sum a_{(\alpha_1,\ldots, \alpha_k,\beta_1,\ldots, \beta_k)}S_{1,\alpha_1}\cdots S_{k,\alpha_k}S_{1,\beta_1}^*\cdots S_{k,\beta_k}^*.
\end{equation*}
 For a more general result regarding the noncommutative Berezin transform associated with the elements of  ${\bf B}_\Lambda(\cH)$,     we refer the reader to  Theorem \ref{Berezin-transf} of Section 4.

 One of the goals of Section  5 is to classify the Beurling type \cite{Be} jointly invariant subspaces of   the universal  standard $k$-tuple  $S=(S_1,\ldots, S_k)$.
We prove, in this section,  that there is an isometric multi-analytic operator
$\Psi: \ell^2(\FF_{n_1}^+\times\cdots \times \FF_{n_k}^+)\otimes \cL\to
\ell^2(\FF_{n_1}^+\times\cdots \times \FF_{n_k}^+)\otimes \cK$, i.e.
$
\Psi(S_{i,s}\otimes I_\cL)=(S_{i,s}\otimes I_\cK)\Psi
$
for any $i\in\{1,\ldots, k\}$ and  $s\in \{1,\ldots, n_i\}$,
such that
$$
\cM=\Psi \left(\ell^2(\FF_{n_1}^+\times\cdots \times \FF_{n_k}^+)\otimes \cL\right)
$$
if and only if
$$(id-\Phi_{S_1\otimes I_\cK})\circ\cdots \circ (id-\Phi_{S_k\otimes I_\cK})(P_\cM)\geq 0$$
if and only if
the  $k$-tuple $\left((S_1\otimes I_\cK)|_\cM,\ldots, (S_k\otimes I_\cK) |_\cM\right)$  of row isometries is doubly $\Lambda$-commuting.

 The second goal of Section 5 is to provide a  classification result for  the pure elements  $T=(T_1,\ldots, T_k)$ in the regular $\Lambda$-polyball (see Theorem \ref{classif1}).
 In particular,
 we obtain the following  classification  of the pure elements  in  the regular $\Lambda$-polybal  with  defect of rank one, which extends a result by Douglas and Foias \cite{DF} (which corresponds to the particular case when $n_1=\cdots=n_k=1$ and $\lambda_{ij}=1$) regarding the uniqueness of multi-variate canonical models.

  Let $T=(T_1,\ldots, T_k)\in B(\cH)^{n_1}\times \cdots \times B(\cH)^{n_k}$.   Then $T$ is a pure element in the regular $\Lambda$-polyball such that $\rank \Delta_T(I)=1$  if and only if  there is  a jointly co-invariant subspace
$\cM\subset \ell^2(\FF_{n_1}^+\times\cdots \times \FF_{n_k}^+)$   under  the isometries $S_{i,s}$,  where $i\in\{1,\ldots, k\}$ and  $s\in \{1,\ldots, n_i\}$, such that $T$ is  jointly unitarily equivalent to  the compression
$P_\cM S|_\cM:=(P_\cM S_{1}|_\cM,\ldots, P_\cM S_{k}|_\cM)$, where
$$
P_\cM S_{i}|_\cM:=[P_\cM S_{i,1}|_\cM \cdots P_\cM S_{i,n_i}|_\cM]
$$
and   $S=(S_1,\ldots, S_k)$ is the universal  standard $k$-tuple.
If $\cM'$ is another jointly co-invariant subspace under $S_{i,s}$, then
$
P_\cM S|_\cM$
  and
$P_{\cM'} S|_{\cM'} $
are unitarily equivalent if and only if $\cM=\cM'$.

 In the last section of the paper,  we show that any $k$-tuple $T$ in the regular $\Lambda$-polyball admits  a minimal    isometric dilation which is a $k$-tuple    of doubly $\Lambda$-commuting row isometries, uniquely determined up to an isomorphism (see Theorem \ref{iso-dil}). Using our Wold decompositions, we show that $T$ is a pure element in ${\bf B}_\Lambda(\cH)$ if and only if    its minimal isometric dilation  is a pure element in ${\bf B}_\Lambda(\cK)$.

In addition, we  prove that  the sequence
of $C^*$-algebras
$$
0\to \boldsymbol\cK\to C^*(\{S_{i,s}\})\to C_\Delta^*(\{V_{i,s}\})\to 0
$$
is exact, where  $\boldsymbol\cK$ is  the ideal of all  compact operators in $B(\ell^2(\FF_{n_1}^+\times\cdots \times \FF_{n_k}^+))$ and
 $C_\Delta^*(\{V_{i,s}\})$ is  the universal algebra  generated by a $k$-tuple $V=(V_1,\ldots, V_k)$ of doubly $\Lambda$-commuting row isometries such that
$\Delta_V(I)=0$.

Finally, in  the particular case when $n_1=\cdots =n_k=1$ and $\Lambda_{ij}=\lambda_{ij}\in \TT$ with
 $\lambda_{ji}=\bar \lambda_{ij}$ we obtain the following  extension of  Brehmer's  result \cite{Br} (which corresponds to  case when $\lambda_{ij}=1$).
  If $T=(T_1,\ldots, T_k)$ is a $k$-tuple in the $\Lambda$-polyball ${\bf B}_\Lambda(\cH)$, then there is a
Hilbert space $\widetilde\cK\supset \cH$ and a $k$-tuple  $U=(U_1,\ldots, U_k)$  of  doubly
 $\Lambda$-commuting  unitary operators on $\widetilde\cK$  such that \begin{equation*}
P_\cH\left(  U_{1}^{m_1^-}\cdots U_{k}^{m_k^-} {U_{1}^*}^{m_1^+}\cdots {U_{k}^*}^{m_k^+}\right)|_\cH\\
=  T_{1}^{m_1^-}\cdots T_{k}^{m_k^-} {T_{1}^*}^{m_1^+}\cdots {T_{k}^*}^{m_k^+}, \qquad (m_1,\ldots, m_k)\in \ZZ^k,
\end{equation*}
and the dilation $U=(U_1,\ldots, U_k)$ is minimal. Moreover,
 the  minimal  $\Lambda$-commuting  unitary dilation   is unique up to an isomorphism.

In a forthcoming  paper \cite{Po-twisted}, we use the results of the present paper to develop a multivariable functional calculus  for $k$-tuples of $\Lambda$-commuting row contractions   on noncommutative Hardy spaces associated with regular $\Lambda$-polyballs.  We  also study  the   characteristic functions   and  the associated multi-analytic models  for the elements of $B_\Lambda(\cH)$.

\bigskip

\section{Doubly $\Lambda$-commuting row isometries and Wold decompositions}

In this section, we obtain Wold decompositions for $k$-tuples of doubly $\Lambda$-commuting row isometries.
For each $i,j\in \{1,\ldots, k\}$ with $i\neq j$, let $\Lambda_{ij}:=[\lambda_{i,j}(s,t)]$, where $s\in \{1,\ldots, n_i\}$ and $t\in \{1,\ldots, n_j\}$, be an $n_i\times n_j$-matrix  with the entries in the torus $\TT:=\{z\in \CC: \  |z|=1\}$, and assume that $\Lambda_{j,i}=\Lambda_{i,j}^*$. Given   row isometries $V_i:=[V_{i,1}\cdots V_{i,n_i}]$, $i\in \{1,\ldots, k\}$, we say that  $V:=(V_1,\ldots, V_k)$  is a the $k$-tuple of {\it doubly $\Lambda$-commuting row isometries}  if
\begin{equation}\label{lac*} V_{i,s}^* V_{j,t}=\overline{\lambda_{i,j}(s,t)}V_{j,t}V_{i,s}^*
\end{equation}
for any $i,j\in \{1,\ldots, k\}$ with $i\neq j$ and any $s\in \{1,\ldots, n_i\}$, $t\in \{1,\ldots, n_j\}$.
We note  that the relation above also  implies  relation
\begin{equation}\label{lac} V_{i,s} V_{j,t}=\lambda_{i,j}(s,t)V_{j,t}V_{i,s}.
\end{equation}
 Indeed, as in \cite{JPS}, one can easily see that,  since $V_{i,s}$ are isometries,  $\lambda_{i,j}(s,t)\in \TT$,  and  $\lambda_{j,i}(t,s)=\overline{\lambda_{i,j}(s,t)}$,
 we have
  $$
  \left(V_{i,s} V_{j,t}-\lambda_{i,j}(s,t)V_{j,t}V_{i,s}\right)^*(V_{i,s} V_{j,t}-\lambda_{i,j}(s,t)V_{j,t}V_{i,s})=0.
  $$

We remark that, in general,  relation \eqref{lac} does not imply relation \eqref{lac*}. However, if relation \eqref{lac} holds and
$$
V_{i,1}V_{i,1}^*+\cdots +V_{i,n_i} V_{i, n_i}^*=I,\qquad i\in \{1,\ldots, k\},
$$
then $V=(V_1,\ldots, V_k)$  is a the $k$-tuple of {\it doubly $\Lambda$-commuting row isometries}. Indeed, note  that, if $i\neq j$, then
$$
V_{j,t}V_{i,1}V_{i,1}^*+\cdots +V_{j,t}V_{i,n_i} V_{i, n_i}^*=V_{j,t}
$$
which, due to relation  \eqref{lac},  implies
$$
\lambda_{j,i}(t,1)V_{i,1} V_{j,t}V_{i,1}^*+\cdots +\lambda_{j,i}(t,n_i) V_{i, n_i}V_{j,t}V_{i,n_i}^*=V_{j,t}.
$$
Multiplying this equation by $V_{i,s}^*$ to the left and taking into account that $V_{i,s}^*V_{i,p}=\delta_{sp}I$, we obtain
$\lambda_{j,i}(t,s)V_{j,t}V_{i,s}^*=V_{i,s}^*  V_{j,t}$. Hence, using the fact that $\lambda_{i,j}(s,t)\in \TT$,  and  $\lambda_{j,i}(t,s)=\overline{\lambda_{i,j}(s,t)}$, we deduce that  relation \eqref{lac*} holds.

For each $i\in \{1,\ldots, k\}$, let $\FF_{n_i}^+$ be the unital free semigroup with generators $g_1^i,\ldots, g_{n_i}^i$ and neutral element $g_0^i$. The length of  $\alpha\in \FF_{n_i}^+$ is defined by $|\alpha|=0$ if $\alpha=g_0^i$ and $|\alpha|=m$ if $\alpha=g_{p_1}^i\cdots g_{p_m}^i\in \FF_{n_i}^+$, where $p_1,\ldots, p_m\in \{1,\ldots, n_i\}$.
If $T_i:=[T_{i,1}\cdots T_{i,n_i}]$,  we  use the notation
$T_{i,\alpha}:=T_{i,p_1}\cdots T_{i,p_m}$ and  $T_{i,g_0^i}:=I$.
Let  $V_i:=[V_{i,1}\cdots V_{i,n_i}] $ be  a row isometry with $V_{i,s}\in B(\cK)$. We say that $V_i$ is a {\it Cuntz row isometry} if
$$
V_{i,1}V_{i,1}^*+\cdots +V_{i,n_i} V_{i,n_i}^*=I_\cK.
$$
We call the row isometry $V_i$ {\it pure}  if there is a wandering subspace $\cL\subset \cK$ , i.e. $\cL\perp V_{i,\alpha} \cL$ for any $\alpha\in \FF_{n_i}, |\alpha|\geq 1$, such that
$\cK=\oplus_{\alpha\in \FF_{n_i}^+} V_{i,\alpha}\cL$.
Let us recall the Wold type decomposition for row isometries obtained in \cite{Po-isometric}.

\begin{theorem} \label{Wold1} Let  $V_i:=[V_{i,1}\cdots V_{i,n_i}] $ be  a row isometry with $V_{i,m}\in B(\cK)$. Then the Hilbert space $\cK$ admits an orthogonal decomposition $\cK=\cK_i^{(s)}\oplus \cK_i^{(c)}$ with the following properties:
\begin{enumerate}
\item[(i)]  $\cK_i^{(s)}$ and $\cK_i^{(c)}$ are reducing subspaces for each isometry $V_{i,m}$, $m\in \{1,\ldots, n_i\}$;

\item[(ii)] $[V_{i,1}|_{\cK_i^{(s)}}\cdots V_{i,n_i}|_{\cK_i^{(s)}}] $ is a pure row isometry on $\cK_i^{(s)}$;

\item[(iii)] $[V_{i,1}|_{\cK_i^{(c)}}\cdots V_{i,n_i}|_{\cK_i^{(c)}}] $ is a Cuntz row isometry on $\cK_i^{(c)}$.
\end{enumerate}
Moreover, the decomposition is uniquely determined, namely
$$
\cK_i^{(s)}=\bigoplus_{\alpha\in \FF_{n_i}^+} V_{i,\alpha}\cL, \qquad  \text{ where } \ \cL=\cK \ominus \left(\bigoplus_{m=1}^{n_i}V_{i,m}\cK\right),
$$
and
$$
\cK_i^{(c)}=\bigcap_{p=0}^\infty \left(\bigoplus_{\alpha\in \FF_{n_i}^+, |\alpha|=p} V_{i,\alpha}\cK\right).
$$
\end{theorem}
An alternative description of the subspaces $\cK_i^{(s)}$ and $\cK_i^{(c)}$ is the following:
\begin{equation} \label{alt}
\begin{split}
\cK_i^{(s)}&=\left\{ h\in \cK: \  \lim_{q\to\infty} \sum_{\alpha\in \FF_{n_i}^+, |\alpha|=q} \|V_{i,\alpha}^*h\|^2=0\right\},  \\
\cK_i^{(c)}&=\left\{ h\in \cK: \    \sum_{\alpha\in \FF_{n_i}^+, |\alpha|=q} \|V_{i,\alpha}^*h\|^2=\|h\|^2\ \text{ for every } \  q\in \NN\right\}.
\end{split}
\end{equation}

Under the notations of Theorem \ref{Wold1}, let $P_i^{(s)}$ and  $P_i^{(c)}$  be the orthogonal projections of $\cK$ onto  the subspaces $\cK_i^{(s)}$ and $\cK_i^{(c)}$, respectively. According to  \cite{Po-multi} , Theorem \ref{Wold1}  implies the following relations:
\begin{enumerate}
\item[(a)]
$P_i^{(s)}=\text{\rm SOT-}\lim_{p\to\infty}\sum_{q=0}^p \sum_{\alpha\in \FF_{n_i}^+, |\alpha|=q} V_{i,\alpha}(I-\sum_{m=1}^{n_i} V_{i,m} V_{i,m}^*)V_{i,\alpha}^*$;
\item[(b)]
$P_i^{(c)}=\text{\rm SOT-}\lim_{q\to\infty}  \sum_{\alpha\in \FF_{n_i}^+, |\alpha|=q} V_{i,\alpha} V_{i,\alpha}^*
$;
\item[(c)] $I_\cK=P_i^{(s)} +P_i^{(c)}$ and  $P_i^{(s)} P_i^{(c)}=0$;

\item[(d)]  $P_i^{(s)}$ and  $P_i^{(c)}$ commute with $V_{i,m}$ for any $m\in \{1,\ldots, k\}$.

\end{enumerate}

A simple consequence   of Theorem \ref{Wold1} and relation \eqref{alt} is the following.

\begin{proposition} \label{charact}
 Let $\cM\subset \cK$ be  a reducing  subspace under each isometry  $V_{i,m}$, $m\in \{1,\ldots, n_i\}$ and let $P_\cM$ be the orthogonal projection of $\cK$ onto $\cM$. Then the following statements  hold.
\begin{enumerate}
\item[(i)]  $[V_{i,1}|_\cM \cdots V_{i,n_i}|_\cM]$ is a Cuntz row isometry on $\cM$ if and only if
 $P_\cM P_i^{(c)}=P_\cM$.

 \item[(ii)]   $[V_{i,1}|_\cM \cdots V_{i,n_i}|_\cM]$ is a pure row isometry on $\cM$ if and only if
 $P_\cM P_i^{(s)}=P_\cM$.
\end{enumerate}

\end{proposition}

Now, using Theorem \ref{Wold1} and Proposition \ref{charact}, one can easily deduce the following  characterization of pure row isometries and Cuntz  row isometries, respectively.
\begin{corollary}
Let  $[V_{i,1}\cdots V_{i,n_i}] $ be  a row isometry with $V_{i,m}\in B(\cK)$. Then  the following statements hold.
\begin{enumerate}
\item[(i)]
$[V_{i,1}\cdots V_{i,n_i}] $ is a pure row isometry if and only if  the only reducing subspace $\cM\subset \cK$ for all  $V_{i,m}$, $m\in \{1,\ldots, n_i\}$, such that $[V_{i,1}|_\cM \cdots V_{i,n_i}|_\cM]$ is a Cuntz row isometry on $\cM$ is $\cM=\{0\}$.
\item[(ii)]
$[V_{i,1}\cdots V_{i,n_i}] $ is a Cuntz row isometry if and only if  the only reducing subspace $\cM\subset \cK$ for all  $V_{i,m}$, $m\in \{1,\ldots, n_i\}$, such that $[V_{i,1}|_\cM \cdots V_{i,n_i}|_\cM]$ is a pure row isometry on $\cM$ is $\cM=\{0\}$.
\end{enumerate}
\end{corollary}

Another consequence of Theorem \ref{Wold1} is the following.
\begin{proposition} Let  $V_i:=[V_{i,1}\cdots V_{i,n_i}] $ be  a row isometry with $V_{i,m}\in B(\cK)$. Then the  following statements hold.
\begin{enumerate}
\item[(i)] If $\cM, \cN$ are reducing  subspaces under each isometry  $V_{i,m}$, $m\in \{1,\ldots, n_i\}$,  such that
$[V_{i,1}|_\cM \cdots V_{i,n_i}|_\cM]$ is a pure row isometry on $\cM$ and $[V_{i,1}|_\cN \cdots V_{i,n_i}|_\cN]$ is a Cuntz row isometry on $\cN$,
then  $\cM\perp \cN$.

\item[(ii)] If $\cK=\cM\oplus \cN$ (algebraically) and $\cM, \cN$ are as in part (i), then
$$
\cM=\cK_i^{(s)}\quad \text{ and }\quad \cN=\cK_i^{(c)}.
$$

\end{enumerate}

\end{proposition}
\begin{proof} Under the hypothesis of item (i) and applying Proposition \ref{charact}, we deduce that
$ \cM\subset \cK_i^{(s)}$ and $ \cN\subset \cK_i^{(c)}$. Since $\cK_i^{(s)}\perp \cK_i^{(c)}$, the result follows. Note that part (ii), is due to part (i) and Theorem \ref{Wold1}. The proof is complete.
\end{proof}

\begin{proposition} \label{rel} Let $V_i:=[V_{i,1}\cdots V_{i,n_i}]$, $i\in \{1,\ldots, k\}$,  be    doubly $\Lambda$-commuting row isometries  and let $i,j\in \{1,\ldots, k\}$ with $i\neq j$. Then the following statements hold.
\begin{enumerate}
\item[(i)] $V_{i,s}$ commutes with $V_{j,\alpha}V_{j,\alpha}^*$ for any $\alpha\in \FF_{n_j}^+$ and $s\in \{1,\ldots, n_j\}$.

\item[(ii)] $V_{i,s}$  and $V_{i,s}^*$ commute with
$$
V_{j,\alpha}\left(I-\sum_{t=1}^{n_j} V_{j,t}V_{j,t}^*\right)V_{j,\alpha}^*,\qquad \alpha\in \FF_{n_j}^+.
$$
 \end{enumerate}
\end{proposition}

\begin{proof} Let $\alpha=g_{p_1}^j\cdots g_{p_m}^j\in \FF_{n_j}^+$, where $p_1,\ldots, p_m\in \{1,\ldots, n_j\}$. Using the relations  \eqref{lac*}, \eqref{lac},  and the fact that  $\lambda_{j,i}(t,s)=\overline{\lambda_{i,j}(s,t)}$ and $\lambda_{i,j}(s,t)\in \TT$, we have
\begin{equation*}
\begin{split}
V_{i,s}\left(V_{j,\alpha}V_{j,\alpha}^*\right)&= V_{i,s}V_{j,p_1}\cdots V_{j,p_m}V_{j,p_m}^*\cdots V_{j,p_1}^*\\
&=\lambda_{i,j}(s,p_1)\cdots \lambda_{i,j}(s,p_m)V_{j,p_1}\cdots V_{j,p_m} V_{i,s}
V_{j,p_m}^*\cdots V_{j,p_1}^*\\
&=
\lambda_{i,j}(s,p_1)\cdots \lambda_{i,j}(s,p_m) \lambda_{j,i}(p_m,s)\cdots \lambda_{j,i}(p_1,s)V_{j,p_1}\cdots V_{j,p_m}
V_{j,p_m}^*\cdots V_{j,p_1}^*V_{i,s}\\
&=|\lambda_{i,j}(s,p_1)|^2\cdots |\lambda_{i,j}(s,p_m)|^2V_{j,p_1}\cdots V_{j,p_m}
V_{j,p_m}^*\cdots V_{j,p_1}^*V_{i,s}\\
&=V_{j,p_1}\cdots V_{j,p_m}
V_{j,p_m}^*\cdots V_{j,p_1}^*V_{i,s},
\end{split}
\end{equation*}
which proves item (i). Consequently, since $V_{i,s}$ commutes with $V_{j,\alpha}V_{j,\alpha}^*$ for any $\alpha\in \FF_{n_j}^+$ and $s\in \{1,\ldots, n_j\}$, so is $V_{i,s}^*$. Now,  item (ii) is clear.
\end{proof}

\begin{proposition}  \label{commutativity}
 Let $V_i:=[V_{i,1}\cdots V_{i,n_i}]$, $i\in \{1,\ldots, k\}$,  be    doubly $\Lambda$-commuting row isometries with $V_{i,m}\in B(\cK)$. If $\cK=\cK_i^{(s)}\oplus \cK_i^{(c)}$ is  the Wold  decomposition corresponding to the row isometry $V_i$, and   $P_i^{(s)}$, $P_i^{(c)}$  are  the orthogonal projections of $\cK$ onto  the subspaces $\cK_i^{(s)}$ and $\cK_i^{(c)}$, respectively, then the orthogonal projections
$$
P_1^{(s)},\ldots, P_k^{(s)}, P_1^{(c)},\ldots, P_k^{(c)}
$$
are pairwise commuting  and  they are also commuting with all the isometries $V_{i,m}$, where $i\in \{1,\ldots, k\}$ and $m\in \{1,\ldots, n_i\}$.

\end{proposition}
\begin{proof} Let  $i,j\in \{1,\ldots, k\}$ with $i\neq j$ and let $p_i, p_j\in \NN$.  Consider the projections
$$
E_i(p_i):=\sum_{\alpha\in \FF_{n_i}^+, |\alpha|\leq p_i} V_{i,\alpha}\left(I-\sum_{m=1}^{n_i} V_{i,m}V_{i,m}^*\right)V_{i,\alpha}^*\quad \text{ and }\quad  F_i(p_i):=\sum_{\alpha\in \FF_{n_i}^+, |\alpha|=p_i} V_{i,\alpha}V_{i,\alpha}^*.
$$
According to  Proposition \ref{rel}, we have
\begin{equation}\label{EF}
\begin{split}
E_i(p_i)E_j(p_j)&=E_j(p_j)E_i(p_i)\\
F_i(p_i)F_j(p_j)&=F_j(p_j)F_i(p_i)\\
 E_i(p_i)F_j(p_j)&=F_j(p_j)E_i(p_i)
\end{split}
\end{equation}
for any $p_i, p_j\in \NN$.
 Let $\alpha=g_{i_1}^i\cdots g_{i_p}^i\in \FF_{n_i}^+$ and $\beta=g_{j_1}^j\cdots g_{j_q}^j\in \FF_{n_j}^+$. Using relation \eqref{lac},  we obtain
\begin{equation*}
V_{i,\alpha} V_{j,\beta}=\boldsymbol{\lambda}_{i,j}(\alpha,\beta)V_{j,\beta}
V_{i,\alpha},\qquad i\neq j,
\end{equation*}
where $$\boldsymbol{\lambda}_{i,j}(\alpha, \beta):=\prod_{a=1}^p\prod_{b=1}^q\lambda_{i,j}(i_a,j_b)\in \TT.$$
On the other hand,  Proposition \ref{rel},  implies
\begin{equation*}
\begin{split}
&V_{i,\alpha}\left(I-\sum_{m=1}^{n_i} V_{i,m}V_{i,m}^*\right)V_{i,\alpha}^* V_{j,\beta}\left(I-\sum_{t=1}^{n_j} V_{j,t}V_{j,t}^*\right)V_{j,\beta}^*\\
&\qquad \qquad=
V_{j,\beta} V_{i,\alpha}\left(I-\sum_{m=1}^{n_i} V_{i,m}V_{i,m}^*\right)V_{i,\alpha}^* \left(I-\sum_{t=1}^{n_j} V_{j,t}V_{j,t}^*\right)V_{j,\beta}^*\\
&\qquad \qquad=
V_{j,\beta} V_{i,\alpha}\left(I-\sum_{m=1}^{n_i} V_{i,m}V_{i,m}^*\right) \left(I-\sum_{t=1}^{n_j} V_{j,t}V_{j,t}^*\right)V_{i,\alpha}^*V_{j,\beta}^*\\
&\qquad\qquad =
\overline{\boldsymbol{\lambda}_{i,j}(\alpha,\beta)}\boldsymbol{\lambda}_{i,j}(\alpha,\beta)V_{i,\alpha} V_{j,\beta}
\left(I-\sum_{m=1}^{n_i} V_{i,m}V_{i,m}^*\right) \left(I-\sum_{t=1}^{n_j} V_{j,t}V_{j,t}^*\right)V_{j,\beta}^*V_{i,\alpha}^*\\
&\qquad\qquad =
 V_{i,\alpha} V_{j,\beta}
\left(I-\sum_{m=1}^{n_i} V_{i,m}V_{i,m}^*\right) \left(I-\sum_{t=1}^{n_j} V_{j,t}V_{j,t}^*\right)V_{j,\beta}^*V_{i,\alpha}^*.
\end{split}
\end{equation*}
Now, taking the appropriate sums, we deduce that
\begin{equation*}
\begin{split}
E_i(p_i)E_j(p_j)=
\sum_{\alpha\in \FF_{n_i}^+, |\alpha|\leq p_i}\sum_{\beta\in \FF_{n_j}^+, |\beta|\leq p_j}V_{i,\alpha}V_{j,\beta}
\left(I-\sum_{m=1}^{n_i} V_{i,m}V_{i,m}^*\right) \left(I-\sum_{t=1}^{n_j} V_{j,t}V_{j,t}^*\right)V_{j,\beta}^*V_{i,\alpha}^*.
\end{split}
\end{equation*}
Using again Proposition \ref{rel}, part (i), we have
$$V_{i,\alpha}V_{i,\alpha}^*V_{j,\beta}V_{j,\beta}^*
=V_{i,\alpha}V_{j,\beta}V_{j,\beta}^*V_{i,\alpha}^*,\qquad \alpha\in \FF_{n_i}^+, \beta\in \FF_{n_j}^+,
 $$
which implies relation

\begin{equation*}
F_i(p_i)F_j(p_j)
=\sum_{\alpha\in \FF_{n_i}^+, |\alpha|=p_i}\sum_{\beta\in \FF_{n_j}^+, |\beta|=p_j}V_{i,\alpha}V_{j,\beta}V_{j,\beta}^*V_{i,\alpha}^*.
\end{equation*}
Similarly, one can see that
$$
V_{i,\alpha}\left(I-\sum_{m=1}^{n_i} V_{i,m}V_{i,m}^*\right)V_{i,\alpha}^* V_{j,\beta} V_{j,\beta}^*=
V_{j,\beta}V_{i,\alpha}\left(I-\sum_{m=1}^{n_i} V_{i,m}V_{i,m}^*\right)V_{i,\alpha}^*  V_{j,\beta}^*
$$
for any $\alpha\in \FF_{n_i}^+, \beta\in \FF_{n_j}^+$,
which implies
$$
E_i(p_i)F_j(p_j)=\sum_{\beta\in \FF_{n_j}^+, |\beta|= p_j}\sum_{\alpha\in \FF_{n_i}^+, |\alpha|\leq p_i} V_{j,\beta}V_{i,\alpha}\left(I-\sum_{m=1}^{n_i}V_{i,m} V_{i,m}^*\right)
V_{i,\alpha}^* V_{j,\beta}^*.
$$
Recall that
$$P_i^{(s)}=\text{\rm SOT-}\lim_{p_i\to\infty} E_i(p_i),\quad P_i^{(c)}=\text{\rm SOT-}\lim_{p_i\to\infty} F_i(p_i),$$
 and $P_i^{(s)} P_i^{(c)}=0$ for any $i\in \{1,\ldots, k\}$. Since
 $\{E_i(p_i)E_j(p_j)\} $,
$\{F_i(p_i)F_j(p_j)\}$, and $\{ E_i(p_i)F_j(p_j)\}$ are increasing sequences of projections as $p_i\to \infty$ and $p_j\to \infty$, we can pass to the limit in relation  \eqref{EF} and obtain
$P_i^{(s)}P_j^{(s)}=P_j^{(s)}P_i^{(s)}$, $P_i^{(c)}P_j^{(c)}=P_j^{(c)}P_i^{(c)}$, and  $P_i^{(s)}P_j^{(c)}=P_j^{(c)}P_i^{(s)}$.

On the other hand,   according to  Proposition \ref{rel}, for  any $i,j\in \{1,\ldots, k\}$, $i\neq j$ the projections $E_i(p_i)$ and $F_i(p_i)$ are commuting with all the isometries $V_{j, t}$, where  $t\in \{1,\ldots, n_j\}$.  Taking the SOT-limits  as $p_i\to \infty$, we deduce that
$P_i^{(s)}$ and $P_i^{(c)}$ are commuting with all the isometries $V_{j, t}$.  Moreover, due to Theorem \ref{Wold1},  the projections $P_i^{(s)}$ and $P_i^{(c)}$ are commuting with all the isometries $V_{i, m}$, where $m\in \{1,\ldots, m\}$.
 The proof is complete.
\end{proof}

Now, we can present our first version of Wold decomposition for $k$-tuples of   doubly $\Lambda$-commuting row isometries.

\begin{theorem} \label{Wold2}    Let $V_i:=[V_{i,1}\cdots V_{i,n_i}]$, $i\in \{1,\ldots, k\}$,  be    doubly $\Lambda$-commuting row isometries with $V_{i,m}\in B(\cK)$.  Then $\cK$ admits a  unique orthogonal decomposition
$$
\cK=\bigoplus_{A\subset \{1,\ldots, k\}} \cK_A
$$
with the following properties:
\begin{enumerate}
\item[(i)]  for each  subset $A\subset \{1,\ldots, k\}$,  the subspace $\cK_A$ is  reducing  for all the isometries
$V_{i,m}$, where $i\in \{1,\ldots, k\}$ and $m\in \{1,\ldots, n_i\}$;

\item[(ii)]  if $i\in A$, then $V_i|_{\cK_A}:=[V_{i,1}|_{\cK_A}\cdots V_{i,n_i}|_{\cK_A}]$ is a pure row isometry;

\item[(iii)] if $i\in A^c$, then $V_i|_{\cK_A}:=[V_{i,1}|_{\cK_A}\cdots V_{i,n_i}|_{\cK_A}]$  is a Cuntz row isometry.

\end{enumerate}

Moreover, we have
$$
\cK_A=\left(\bigcap_{i\in A}\cK_i^{(s)}\right)\cap \left(\bigcap_{i\in A^c}\cK_i^{(c)}\right),
$$
where
$$
\cK_i^{(s)}:=\bigoplus_{\alpha\in \FF_{n_i}^+}V_{i,\alpha}\left(\cap_{m=1}^{n_i}\ker V_{i,m}^*\right)\quad \text{ and }\quad \cK_i^{(c)}:=\bigcap_{p=0}^\infty\left(\bigoplus_{\alpha\in \FF_{n_i}^+, |\alpha|=p}V_{i,\alpha} \cK\right).
$$
\end{theorem}
\begin{proof}
Let  $\cK=\cK_i^{(s)}\oplus \cK_i^{(c)}$ be   the Wold  decomposition for  the row isometry $V_i$, and   let $P_i^{(s)}$, $P_i^{(c)}$  be   the orthogonal projections of $\cK$ onto  the subspaces
 $\cK_i^{(s)}$ and $\cK_i^{(c)}$, respectively. Since $I_\cK=P_i^{(s)}+P_i^{(c)}$ for any $i\in \{1,\ldots, k\}$.    Proposition \ref{commutativity} implies
 $$
 I_\cK=\prod_{i=1}^k \left(P_i^{(s)}+P_i^{(c)} \right)=\sum_{A\subset\{1,\ldots, k\}}
 \left(\prod_{i\in A}P_i^{(s)}\right)\left(\prod_{i\in A^c}P_i^{(c)}\right).
 $$
Consider the orthogonal projection  $P_A:=\left(\prod_{i\in A}P_i^{(s)}\right)\left(\prod_{i\in A^c}P_i^{(c)}\right)$ and set $\cK_A:=P_A \cK$.   Now, one can easily see
that $\cK_A=\left(\bigcap_{i\in A}\cK_i^{(s)}\right)\cap \left(\bigcap_{i\in A^c}\cK_i^{(c)}\right)$, where
$\cK_i^{(s)}$  and $\cK_i^{(c)}$ are  given by
 Theorem \ref{Wold1}.

Note that  if $A,B$ are  distinct subsets of $\{1,\ldots, k\}$  the $P_AP_B$ contains a factor $P_a^{(s)}P_a^{(c)}$ for some $a\in \{1,\ldots, k\}$.  Since $P_a^{(s)}P_a^{(c)}=0$, we also have
$P_AP_B=0$, which is equivalent to $\cK_A\perp \cK_B$.  Now, it is clear that
$\cK=\bigoplus_{A\subset \{1,\ldots, k\}} \cK_A$.

On the other, due to Proposition \ref{commutativity},  the projection $P_A$  commutes with all the isometries $V_{i,m}$, where $i\in \{1,\ldots, k\}$ and $m\in \{1,\ldots, n_i\}$. Therefore, $\cK_A$ is reducing for all these isometries, which proves part (i) of the theorem.
Now, note that if $i\in A$, then $P_AP_i^{(s)}=P_A$, while $P_AP_i^{(c)}=P_A$ if $i\in A^c$. Applying Proposition \ref{charact}, we deduce items (ii) and (iii).

To prove uniqueness of the Wold decomposition, assume that
$\cK=\bigvee_{A\subset \{1,\ldots, k\}} \cK_A'$, where the subspaces $\cK_A'$  are such that
$\cK_A' \cap\cK_B'=\{0\}$ if $A$ and $B$ are distinct  subsets of $\{1,\ldots, k\}$, and such that the conditions (i), (ii), and (iii) hold when $\cK_A$ is replaced by $\cK_A'$.
We will prove that, under these conditions, $\cK_A'=\cK_A$ for any $A\subset \{1,\ldots, k\}$.
Due to the first part of the theorem, it is enough to show that $\cK_A'\subset\cK_A$.
Actually we show a little bit more, namely, that
 $V_i|_{\cK_A'}:=[V_{i,1}|_{\cK_A'}\cdots V_{i,n_i}|_{\cK_A'}]$ is a pure row isometry if  $i\in A$  and Cuntz row isometry if $i\in A^c$ if and only if $\cK_A'\subset\cK_A$.
 To prove the direct implication, we can use  Proposition \ref{charact} to deduce that
 $P_{\cK_A'} P_i^{(s)}=P_{\cK_A'}$ if   $i\in A$ and $P_{\cK_A'} P_i^{(c)}=P_{\cK_A'}$ if  $i\in A^c$.
Hence,  we have
$$
P_{\cK_A'}\prod_{i\in A} P_i^{(s)}=P_{\cK_A'} \quad \text{  and } \quad  P_{\cK_A'} \prod_{i\in A^c}P_i^{(c)}=P_{\cK_A'},
$$
which implies $P_{\cK_A'}P_A=P_{\cK_A'}$. Therefore, $\cK_A'\subset \cK_A$.

Conversely, assume that $\cK_A'\subset \cK_A$ and let $i\in A$. Then
$P_{\cK_A'}P_A=P_{\cK_A'}$ and $P_A P_i^{(s)}=P_A$.
Due to the commutativity, we have
$$
P_{\cK_A'}P_i^{(s)}=P_{\cK_A'}P_A P_i^{(s)}=P_{\cK_A'}P_A=P_{\cK_A'}.
$$
Similarly, if $i\in A^c$, we deduce that $P_{\cK_A'}P_i^{(c)}=P_{\cK_A'}$.
Using again Proposition \ref{charact}, we conclude that $V_i|_{\cK_A'}:=[V_{i,1}|_{\cK_A'}\cdots V_{i,n_i}|_{\cK_A'}]$ is a pure row isometry if  $i\in A$,  and Cuntz row isometry if $i\in A^c$.
The proof is complete.
\end{proof}

We record the following result that was proved in the proof of Theorem \ref{Wold2}.

\begin{proposition} Let $V_i:=[V_{i,1}\cdots V_{i,n_i}]$, $i\in \{1,\ldots, k\}$,  be    doubly $\Lambda$-commuting row isometries with $V_{i,m}\in B(\cK)$. If $\cM\subset\cK$  is a reducing subspace for all the isometries
$V_{i,m}$, where $i\in \{1,\ldots, k\}$ and $m\in \{1,\ldots, n_i\}$, and  $A\subset \{1,\ldots, k\}$, then the following statements are equivalent.
\begin{enumerate}
\item[(i)]  $V_i|_{\cM}:=[V_{i,1}|_{\cM}\cdots V_{i,n_i}|_{\cM}]$ is a pure row isometry if $i\in A$ and it is a Cuntz row isometry if $i\in A^c$.

\item[(ii)] $P_\cM P_A=P_\cM$, where $P_A:=\left(\prod_{i\in A}P_i^{(s)}\right)\left( \prod_{i\in A^c} P_i^{(c)}\right)$
\end{enumerate}
\end{proposition}

In what follows, we present a more precise description  of the subspaces defined in Theorem \ref{Wold2}.

\begin{theorem}\label{Wold3} Let $V_i:=[V_{i,1}\cdots V_{i,n_i}]$, $i\in \{1,\ldots, k\}$,  be    doubly $\Lambda$-commuting row isometries with $V_{i,m}\in B(\cK)$ and let $\cK_A$ be the subspace defined in Theorem \ref{Wold2}, where  $A=\{i_1,\ldots,i_p\}\subset \{1,\ldots, k\}$.  If  $A^c=\{j_1,\ldots, j_{k-p}\}$, then there is a unique  subspace $\cL_A\subset \cK_A$ that is invariant under the isometries $V_{j_1,t_1},\ldots, V_{j_{k-p},t_{k-p}}$, where $t_1\in \{1,\ldots, n_{j_1}\},\ldots, t_{k-p}\in \{1,\ldots, n_{j_{k-p}}\}$, and such that
$$
\cK_A=\bigoplus_{\alpha_{i_1}\in \FF_{n_{i_1}}^+,\ldots, \alpha_{i_p}\in \FF_{n_{i_p}}^+} V_{i_1,\alpha_{i_1}}\cdots V_{i_p,\alpha_{i_p}}\left(\cL_A\right).
$$
Moreover, we have
$$
\cL_A=\bigcap_{m_{j_1},\ldots, m_{j_{k-p}}=0}^\infty
\left(\bigoplus_{{\alpha_{j_1}\in \FF^+_{n_{j_1}}}\atop{ |\alpha_{j_1}|=m_{j_1}}}\cdots
\bigoplus_{{\alpha_{j_{k-p}}\in \FF^+_{n_{j_{k-p}}}}\atop { |\alpha_{j_{k-p}}|=m_{j_{k-p}}}}
V_{j_1,\alpha_{j_1}}\cdots V_{j_{k-p}, \alpha_{j_{k-p}}}\left(
\bigcap_{i\in A, s\in \{1,\ldots, n_i\}}\ker V_{i,s}^*\right)\right).
$$
In this case, the following statements hold.
\begin{enumerate}
\item[(i)] If $i\in A^c$, then $\cL_A$ is reducing for the isometries $V_{i,m}$, where $m\in \{1,\ldots, n_i\}$, and
$[V_{i,1}|_{\cL_A}\cdots V_{i,n_i}|_{\cL_A}]$  is a Cuntz row isometry.

\item[(ii)] If $i,j\in A^c$, $i\neq j$, then
$$ \left(V_{i,s}|_{\cL_A}\right)^* \left(V_{j,t}|_{\cL_A}\right)=\overline{\lambda_{ij}(s,t)}\left(V_{j,t}|_{\cL_A}\right)\left(V_{i,s}|_{\cL_A}\right)^*
$$
for any   $s\in \{1,\ldots, n_i\}$, $t\in \{1,\ldots, n_j\}$.
\item[(iii)]
If $i\in A$, then
$$\cL_A\subset \bigcap_{s\in \{1,\ldots, n_i\}}\ker V_{i,s}^*.
$$
\item[(iv)]
If $r\in \{1,\ldots, p\}$, $\alpha_{i_1}\in \FF_{n_{i_1}}^+,\ldots, \alpha_{i_p}\in \FF_{n_{i_p}}^+$ and $\beta\in \FF_{n_{i_r}}^+$, then
$$
V_{i_r, \beta}\left( V_{i_1,\alpha_{i_1}}\cdots V_{i_r,\alpha_{i_r}}\cdots V_{i_p,\alpha_{i_p}}(\cL_A)\right)
=
V_{i_1,\alpha_{i_1}}\cdots V_{i_r,\beta\alpha_{i_r}}\cdots V_{i_p,\alpha_{i_p}}(\cL_A).
$$
\end{enumerate}
\end{theorem}

\begin{proof}
Assume that $1\leq p\leq k-1$. The first step of the proof is to  show that
\begin{equation}
\label{PA}
\begin{split}
P_A&:=\left(\prod_{i\in A}P_i^{(s)}\right)\left( \prod_{j\in A^c} P_j^{(c)}\right)\\
&=
\sum_{\alpha_{i_1}\in \FF_{n_{i_1}}^+}\cdots \sum_{\alpha_{i_p}\in \FF_{n_{i_p}}^+}
\left\{ V_{i_1,\alpha_{i_1}}\cdots V_{i_p,\alpha_{i_p}}\left[\left( \prod_{j\in A^c} P_j^{(c)}\right)\Delta_A\right]V_{i_p,\alpha_{i_p}}^*\cdots V_{i_1,\alpha_{i_1}}^*\right\},
\end{split}
\end{equation}
where  $\Delta_A:=\prod_{i\in A}\left(I_\cK-\sum_{m=1}^{n_i} V_{i,m} V_{i,m}^*\right)$ and the convergence of the series is in the strong operator topology.
Recall that Theorem \ref{Wold1} implies
$$P_i^{(s)}=\text{\rm SOT-}\lim_{p\to\infty}\sum_{q=0}^p \sum_{\alpha\in \FF_{n_i}^+, |\alpha|=q} V_{i,\alpha}(I-\sum_{m=1}^{n_i} V_{i,m} V_{i,m}^*)V_{i,\alpha}^*,\qquad i\in A,
$$
and
$$
P_j^{(c)}=\text{\rm SOT-}\lim_{q\to\infty}  \sum_{\alpha\in \FF_{n_j}^+, |\alpha|=q} V_{j,\alpha} V_{j,\alpha}^*, \qquad j\in A^c.
$$
Consequently, as in the proof of Proposition \ref{commutativity}, we can prove inductively that
\begin{equation}\label{Ps}
\prod_{i\in A}P_i^{(s)}=
\sum_{\alpha_{i_1}\in \FF_{n_{i_1}}^+}\cdots \sum_{\alpha_{i_p}\in \FF_{n_{i_p}}^+}
 V_{i_1,\alpha_{i_1}}\cdots V_{i_p,\alpha_{i_p}} \Delta_A V_{i_p,\alpha_{i_p}}^*\cdots V_{i_1,\alpha_{i_1}}^*
\end{equation}
and
\begin{equation}\label{Pc}
\begin{split}
\prod_{j\in A^c} P_j^{(c)}
&=\text{\rm SOT-}\lim_{{m_{j_1}\to\infty}\atop{{\ldots}\atop{ m_{j_{k-p}\to \infty}}}}
\prod_{j\in A^c}\left(\sum_{\alpha\in\FF_{n_j}^+, |\alpha_j|=m_j} V_{j,\alpha_j} V_{j,\alpha_j}^*\right)\\
&=
\text{\rm SOT-}\lim_{{m_{j_1}\to\infty}\atop{{\ldots}\atop{ m_{j_{k-p}\to \infty}}}}
\sum_{{\alpha_{j_1}\in \FF^+_{n_{j_1}}}\atop{ |\alpha_{j_1}|=m_{j_1}}}\cdots \sum_{{\alpha_{j_1}\in \FF^+_{n_{j_1}}}\atop{ |\alpha_{j_1}|=m_{j_1}}}
V_{j_1,\alpha_{j_1}}\cdots V_{j_{k-p}, \alpha_{j_{k-p}}}V_{j_{k-p}, \alpha_{j_{k-p}}}^*\cdots V_{j_1,\alpha_{j_1}}^*.
\end{split}
\end{equation}
Using Proposition \ref{rel}, we deduce that  the projection $\prod_{j\in A^c} P_j^{(c)}$ commutes with  $V_{i,\alpha_i}$ for any $i\in A$ and $\alpha_i\in \{1,\ldots, n_i\}$, and the projection
$\Delta_A$  commutes with  $V_{j,\alpha_j}$ for any $j\in A^c$ and $\alpha_j\in \{1,\ldots, n_j\}$.
Consequently, relation \eqref{Ps} implies relation \eqref{PA}, while relation \eqref{Pc} implies
\begin{equation*}
\left(\prod_{j\in A^c} P_j^{(c)}\right)\Delta_A
=
\text{\rm SOT-}\lim_{{m_{j_1}\to\infty}\atop{{\ldots}\atop{ m_{j_{k-p}\to \infty}}}}
\sum_{{\alpha_{j_1}\in \FF^+_{n_{j_1}}}\atop{ |\alpha_{j_1}|=m_{j_1}}}\cdots \sum_{{\alpha_{j_1}\in \FF^+_{n_{j_1}}}\atop{ |\alpha_{j_1}|=m_{j_1}}}
V_{j_1,\alpha_{j_1}}\cdots V_{j_{k-p}, \alpha_{j_{k-p}}}\Delta_A V_{j_{k-p}, \alpha_{j_{k-p}}}^*\cdots V_{j_1,\alpha_{j_1}}^*.
\end{equation*}
Now, note that $\Delta_A$ is the orthogonal projection of $\cK$  onto the subspace
 $$\bigcap_{{i\in A}\atop{m\in \{1,\ldots,n_i\}}} \ker V_{i,m}^*
 $$
and
\begin{equation} \label{seq}
\sum_{{\alpha_{j_1}\in \FF^+_{n_{j_1}}}\atop{ |\alpha_{j_1}|=m_{j_1}}}\cdots \sum_{{\alpha_{j_1}\in \FF^+_{n_{j_1}}}\atop{ |\alpha_{j_1}|=m_{j_1}}}
V_{j_1,\alpha_{j_1}}\cdots V_{j_{k-p}, \alpha_{j_{k-p}}}\Delta_A V_{j_{k-p}, \alpha_{j_{k-p}}}^*\cdots V_{j_1,\alpha_{j_1}}^*
\end{equation}
is the orthogonal projection onto the subspace
$$
\bigoplus_{{\alpha_{j_1}\in \FF^+_{n_{j_1}}}\atop{ |\alpha_{j_1}|=m_{j_1}}}\cdots
\bigoplus_{{\alpha_{j_{k-p}}\in \FF^+_{n_{j_{k-p}}}}\atop { |\alpha_{j_{k-p}}|=m_{j_{k-p}}}}
V_{j_1,\alpha_{j_1}}\cdots V_{j_{k-p}, \alpha_{j_{k-p}}}\left(
\bigcap_{{i\in A}\atop{m\in \{1,\ldots,n_i\}}} \ker V_{i,m}^*\right).
$$
Since  \eqref{seq} is  a decreasing net of projections, it is clear that
$\left(\prod_{j\in A^c} P_j^{(c)}\right)\Delta_A$ is the orthogonal projection onto $\cL_A$.
Consequently,  using relation \eqref{PA} and the fact that  the operator
$$
V_{i_1,\alpha_{i_1}}\cdots V_{i_p,\alpha_{i_p}}\left[\left( \prod_{j\in A^c} P_j^{(c)}\right)\Delta_A\right]V_{i_p,\alpha_{i_p}}^*\cdots V_{i_1,\alpha_{i_1}}^*
$$
is the orthogonal projection onto the subspace  $V_{i_1,\alpha_{i_1}}\cdots V_{i_p,\alpha_{i_p}}(\cL_A)$, we deduce that  $P_A$ is the orthogonal projection onto
$$
\cK_A:=\bigoplus_{\alpha_{i_1}\in \FF_{n_{i_1}}^+,\ldots, \alpha_{i_p}\in \FF_{n_{i_p}}^+} V_{i_1,\alpha_{i_1}}\cdots V_{i_p,\alpha_{i_p}}\left(\cL_A\right).
$$

Now, we prove the second part of the theorem.  As we saw above,
$P_{\cL_A}=\left(\prod_{j\in A^c} P_j^{(c)}\right)\Delta_A$. Due to Proposition \ref{commutativity},
$\prod_{j\in A^c} P_j^{(c)}$ commutes with  all the isometries $V_{i,m}$, where $i\in \{1,\ldots, k\}$ and $m\in \{1,\ldots, n_i\}$. Moreover, Proposition \ref{rel} implies commutativity of $\Delta_A$ with all the isometries $V_{j,m}$, where $j\in A^c$ and $m\in \{1,\ldots, n_j\}$. Consequently,
$P_{\cL_A}$ commutes  with  $V_{j,m}$. This shows that $\cL_A$ is a reducing subspace for all the isometries $V_{j,m}$, where $j\in A^c$ and $m\in \{1,\ldots, n_j\}$.
 Since $\cL_A\subset \cK_A$ and, due to Theorem \ref{Wold2},
 $\cK_A\subset \cap_{j\in A^c} \cK_j^{(c)}$, we can apply Proposition \ref{charact} to deduce that
$[V_{j,1}|_{\cL_A}\cdots V_{j,n_j}|_{\cL_A}]$  is a Cuntz row isometry for any $j\in A^c$.

Since $V_i:=[V_{i,1}\cdots V_{i,n_i}]$, $i\in \{1,\ldots, k\}$,  are     doubly $\Lambda$-commuting row isometries and   $\cL_A$ is a reducing subspace for all the isometries $V_{j,m}$, where $j\in A^c$ and $m\in \{1,\ldots, n_j\}$, part (ii) is clear.

To prove part (iii), note that if  $i_0\in A$, then due to  the fact that $V_i:=[V_{i,1}\cdots V_{i,n_i}]$, $i\in \{1,\ldots, k\}$,  are     doubly $\Lambda$-commuting row isometries, we deduce that
\begin{equation*}
\begin{split}
&V_{i_0,m}^*V_{j_1,\alpha_{j_1}}\cdots V_{j_{k-p}, \alpha_{j_{k-p}}}\left(
\bigcap_{i\in A, s\in \{1,\ldots, n_i\}}\ker V_{i,s}^*\right)\\
&\qquad \qquad=V_{j_1,\alpha_{j_1}}\cdots V_{j_{k-p}, \alpha_{j_{k-p}}}V_{i_0,m}^*\left(
\bigcap_{i\in A, s\in \{1,\ldots, n_i\}}\ker V_{i,s}^*\right)
=\{0\}.
\end{split}
\end{equation*}
Now, the relation defining the subspace $\cL_A$ shows that  $V_{i_0,m}^*|_{\cL_A}=0$,
which implies
 part (iii). Finally, part (iv) is due to the fact that
 $V_{i_r, \beta}$ commutes with $V_{i_1,\alpha_{i_1}},\ldots, V_{i_r,\alpha_{i_r}},\ldots, V_{i_p,\alpha_{i_p}}$ up o some constants.

It remains to prove the uniqueness of the wandering subspace $\cL_A$. To this end, assume that
$\cM\subset \cK_A$ is  a subspace   invariant under the isometries $V_{j_1,t_1},\ldots, V_{j_{k-p},t_{k-p}}$, where $t_1\in \{1,\ldots, n_{j_1}\},\ldots, t_{k-p}\in \{1,\ldots, n_{j_{k-p}}\}$, and such that
\begin{equation}
\label{M}
\cK_A=\bigoplus_{\alpha_{i_1}\in \FF_{n_{i_1}}^+,\ldots, \alpha_{i_p}\in \FF_{n_{i_p}}^+} V_{i_1,\alpha_{i_1}}\cdots V_{i_p,\alpha_{i_p}}\left(\cM\right).
\end{equation}
 Let $q\in \{1,\ldots, p\}$ and $m\in \{1,\ldots, n_{i_q}\}$. Since $V_{i_q,m}$ commutes up to some constants with any $V_{i,\alpha_i}$ if $i\neq i_q$, it is clear that  $V_{i_q,m}\cK_A\subset \cK_A$
 and $\cM\perp V_{i_q,m}\cK_A$. On the other hand, if $B\subset \{1,\ldots, k\}$ with $B\neq A$, then  $\cM\subset\cK_A\perp \cK_B\supset V_{i_q,m}\cK_B$.
 Consequently, using that
 $
\cK=\bigoplus_{C\subset \{1,\ldots, k\}} \cK_C,
$
 we deduce that $\cM\perp V_{i_q,m}\cK$ for any $q\in \{1,\ldots, p\}$ and $m\in \{1,\ldots, n_{i_q}\}$.
 Therefore,
 \begin{equation}
 \label{Mcap}
 \cM\subset \bigcap_{i\in A, m\in \{1,\ldots, n_i\}} \ker V_{i,m}^*.
 \end{equation}
 If $A=\{1,\ldots, k\}$, the latter relation implies $\cM\subset \cL_A$. Comparing relation \eqref{M} with
 \begin{equation}
 \label{KA}
\cK_A=\bigoplus_{\alpha_{i_1}\in \FF_{n_{i_1}}^+,\ldots, \alpha_{i_p}\in \FF_{n_{i_p}}^+} V_{i_1,\alpha_{i_1}}\cdots V_{i_p,\alpha_{i_p}}\left(\cL_A\right),
\end{equation}
 we deduce that  $\cM=\cL_A$. Now, consider the case when $A\neq \{1,\ldots, k\}$. Let $r\in \{1,\ldots, k-p\}$ and note that, due to Theorem \ref{Wold2}, we have
 $\oplus_{m=1}^{n_{j_r}} V_{j_r,m}\cK_A=\cK_A$. Hence, and using relation \eqref{M}, we obtain
 \begin{equation*}
 \begin{split}
\cK_A &=\bigoplus_{\alpha_{i_1}\in \FF_{n_{i_1}}^+,\ldots, \alpha_{i_p}\in \FF_{n_{i_p}}^+} V_{i_1,\alpha_{i_1}}\cdots V_{i_p,\alpha_{i_p}}\left(\cM\right)\\
 &=\bigoplus_{\alpha_{i_1}\in \FF_{n_{i_1}}^+,\ldots, \alpha_{i_p}\in \FF_{n_{i_p}}^+} V_{i_1,\alpha_{i_1}}\cdots V_{i_p,\alpha_{i_p}}\left(\oplus_{m=1}^{n_{j_r}} V_{j_r,m}\cM\right).
 \end{split}
 \end{equation*}
 Since $\oplus_{m=1}^{n_{j_r}} V_{j_r,m}\cM\subset \cM$, we deduce that
 $\oplus_{m=1}^{n_{j_r}} V_{j_r,m}\cM= \cM$ for any $r\in \{1,\ldots, k-p\}$. Consequently, we obtain
 $$
 \bigcap_{m_{j_1},\ldots, m_{j_{k-p}}=0}^\infty
\bigoplus_{{\alpha_{j_1}\in \FF^+_{n_{j_1}}}\atop{ |\alpha_{j_1}|=m_{j_1}}}\cdots
\bigoplus_{{\alpha_{j_{k-p}}\in \FF^+_{n_{j_{k-p}}}}\atop { |\alpha_{j_{k-p}}|=m_{j_{k-p}}}}
V_{j_1,\alpha_{j_1}}\cdots V_{j_{k-p}, \alpha_{j_{k-p}}}\left(
\cM\right)=\cM.
$$
 Now, due to relation \eqref{Mcap}, we have  $\cM\subset \bigcap_{i\in A, m\in \{1,\ldots, n_i\}} \ker V_{i,m}^*$, which, using the definition of $\cL_A$  implies $\cM\subset \cL_A$. Comparing relations \eqref{M}  and \eqref{KA}, we conclude that $\cM=\cL_A$.
 The proof is complete.
\end{proof}

 \begin{remark} \label{partic}
 Theorem \ref{Wold3} holds true in the
  particular cases  when   $A=\emptyset$ or $A=\{1,\ldots,k\}$.
  \begin{enumerate}
  \item[(i)]  If  $A=\emptyset$, then we have
$$
\cL_\emptyset=\bigcap_{m_{1},\ldots, m_{{k}}=0}^\infty
\left(\bigoplus_{{\alpha_{1}\in \FF^+_{n_{1}}}\atop{ |\alpha_{1}|=m_{1}}}\cdots
\bigoplus_{{\alpha_{{k}}\in \FF^+_{n_{{k}}}}\atop { |\alpha_{{k}}|=m_{{k}}}}
V_{1,\alpha_{1}}\cdots V_{{k}, \alpha_{{k}}}
\cK \right)\quad  \text{ and  }\quad \cK_\emptyset=\cL_\emptyset.
$$

 \item[(ii)]  If
 $A=\{1,\ldots,k\}$, then
  $$
 \cL_{\{1,\ldots,k\}}=\bigcap_{i\in \{1,\ldots,k\},  m\in \{1,\ldots, n_i\}}\ker V_{i,m}^*
 $$
 and
 $$
 \cK_{\{1,\ldots,k\}}=\bigoplus_{\alpha_{1}\in \FF_{n_{1}}^+,\ldots, \alpha_{k}\in \FF_{n_{k}}^+} V_{1,\alpha_{1}}\cdots V_{k,\alpha_{k}}\left(\cL_{\{1,\ldots,k\}}\right).
 $$
 \end{enumerate}
 \end{remark}
  Since the proof is similar and easier,  we leave it to the reader.

\begin{definition}
Let $A=\{i_1,\ldots, i_p\}\subset \{1,\ldots, k\}$ and $A^c=\{j_1,\ldots, j_{k-p}\}$ with $i_1<\cdots <i_p$ and $j_1<\cdots <j_{k-p}$.
The subspace $\cL_A$ in Theorem \ref{Wold2} is called the {\it $A$-wandering subspace} of the $k$-tuple  $(V_1,\ldots, V_k)$ of  doubly $\Lambda$-commuting row isometries. We also call the pair
$\cW_A:=\left( \cL_A, \left(V_{j_1}|_{\cL_A},\ldots, V_{j_{k-p}}|_{\cL_A}\right)\right)$ the {\it $A$-wandering data} of
 $(V_1,\ldots, V_k)$.
  \end{definition}
  In light of Remark \ref{partic}, we can consider a similar  definition  when $A=\emptyset$ or  $A=\{1,\ldots, k\}$.
The $\emptyset$-wandering data of
$(V_1,\ldots, V_k)$ is
$\cW_\emptyset:=\left( \cK_\emptyset, \left(V_{1}|_{\cK_\emptyset},\ldots, V_{k}|_{\cK_\emptyset}\right)\right)$ and the  $\{1,\ldots, k\}$-wandering data of $(V_1,\ldots, V_k)$ is
$ \cW_{\{1,\ldots, k\}}:=(\cL_{\{1,\ldots, k\}})$.

We remark that, according to Theorem \ref{Wold3},  $\left(V_{j_1}|_{\cL_A},\ldots, V_{j_{k-p}}|_{\cL_A}\right)$ is a $(k-p)$-tuple  of doubly $\Lambda$-commuting  Cuntz  row isometries.

\bigskip

\section{Standard  $k$-tuples of doubly $\Lambda$-commuting row isometries}

In this section, we  introduce the standard  $k$-tuples of doubly $\Lambda$-commuting row isometries which will play the  role of models among the $k$-tuples of doubly $\Lambda$-commuting row isometries. They will play an important role  in our investigation.

 Let $1\leq p\leq k-1$ and consider the Hilbert space $\ell^2(\FF_{n_1}^+\times\cdots \times \FF_{n_p}^+)$ with the standard basis  $\{\chi_{(\alpha_1,\ldots, \alpha_p)}\}$, where
 $\alpha_1\in \FF_{n_1}^+,\ldots, \alpha_p\in \FF_{n_p}^+$. Let $\cL$ be a Hilbert space and let $\{\widetilde W_j\}_{j\in \{p+1,\ldots,k\}}$ be  Cuntz row isometries $\widetilde W_j:=[\widetilde W_{j,1}\cdots \widetilde W_{j,n_j}]$ on $\cL$  satisfying the relations
\begin{equation}\label{doubly}
\widetilde W_{i,s}^* \widetilde W_{j,t}=\overline{\lambda_{ij}(s,t)}\widetilde W_{j,t}\widetilde W_{i,s}^*
\end{equation}
for any $i,j\in \{p+1,\ldots, k\}$ with $i\neq j$ and any $s\in \{1,\ldots, n_i\}$, $t\in \{1,\ldots, n_j\}$.
We associate with the given data $\left\{\cL, (\widetilde W_{p+1},\ldots, \widetilde W_k)\right\}$ a standard
$k$-tuple $({\bf S}_1,\ldots, {\bf S}_p, W_{p+1},\ldots W_k)$  of doubly $\Lambda$-commuting row isometries on the Hilbert tensor product $\ell^2(\FF_{n_1}^+\times\cdots \times \FF_{n_p}^+)\otimes \cL$, as follows.

For each  $i\in \{1,\ldots, p\}$ and $s\in \{1,\ldots, n_i\}$, we define the row operator  ${\bf S}_i:=[{\bf S}_{i,1}\cdots {\bf S}_{i,n_i}]$ by setting
\begin{equation}
\label{shift}
\begin{split}
&{\bf S}_{i,s}\left( \chi_{(\alpha_1,\ldots, \alpha_p)}\otimes h\right)\\
&\qquad \qquad :=\begin{cases} \chi_{(g_s^i\alpha_1,\alpha_2,\ldots, \alpha_p)}\otimes h, &\quad \text{ if } i=1\\
\boldsymbol{\lambda}_{i,1}(s,\alpha_1)\cdots \boldsymbol{\lambda}_{i,i-1}(s,\alpha_{i-1})
\chi_{\alpha_1,\ldots, \alpha_{i-1}, g_s^i \alpha_i, \alpha_{i+1},\ldots, \alpha_p)}\otimes h,&\quad \text{ if } i\in \{2,\ldots,p\}
\end{cases}
\end{split}
\end{equation}
for any $h\in \cL$, $\alpha_1\in \FF_{n_1}^+,\ldots, \alpha_{p}\in \FF_{n_{p}}^+$,
where, for any   $j\in \{1,\ldots, p\}$,
$$
\boldsymbol{\lambda}_{i,j}(s, \beta)
:= \begin{cases}\prod_{b=1}^q\lambda_{i,j}(s,j_b)&\quad  \text{ if }
 \beta=g_{j_1}^j\cdots g_{j_q}^j\in \FF_{n_j}^+  \\
 1& \quad  \text{ if } \beta=g_0^j.
 \end{cases}
$$

For each  $i\in \{p+1,\ldots, k\}$ and $s\in \{1,\ldots, n_i\}$, we define the row operator $W_i:=[W_{i,1}\cdots W_{i,n_i}]$ by setting
\begin{equation}\label{cuntz}
W_{i,s}\left( \chi_{(\alpha_1,\ldots, \alpha_p)}\otimes h\right):=\boldsymbol{\lambda}_{i,1}(s,\alpha_1)\cdots \boldsymbol{\lambda}_{i,p}(s,\alpha_p)
\chi_{\alpha_1,\ldots,   \alpha_p)}\otimes \widetilde W_{i,s}h,\qquad h\in \cL.
\end{equation}

Now, we consider the cases when $p=k$ or  $p=0$.
If $p=k$, then the standard $k$-tuple $({\bf S}_1,\ldots, {\bf S}_k)$ of doubly $\Lambda$-commuting row isometries acting on  $\ell^2(\FF_{n_1}^+\times\cdots \times \FF_{n_k}^+)\otimes \cL$  is defined by relation  \eqref{shift}, where we take $p=k$. When $p=0$ the standard $k$-tuple $(W_1,\ldots, W_k)$  of  doubly $\Lambda$-commuting Cuntz  row isometries $W_i:=[W_{i,1}\cdots W_{i,n_i}]$   is defined on  $  \cL$  by setting
$W_{i,s}:=  \widetilde W_{i,s}$.

From now on,  we shall assume that $0\leq p\leq k$.

\begin{theorem}\label{standard}  Given   data $\left\{\cL, (\widetilde W_{p+1},\ldots, \widetilde W_k)\right\}$ with the property that relation \eqref{doubly} holds,
the standard
$k$-tuple $$({\bf S}_1,\ldots, {\bf S}_p, W_{p+1},\ldots, W_k)$$  associated with it and  defined by relations \eqref{shift} and \eqref{cuntz} is a
 doubly $\Lambda$-commuting  sequence of row isometries on the Hilbert space  $\ell^2(\FF_{n_1}^+\times\cdots \times \FF_{n_p}^+)\otimes \cL$. Moreover, the following statements hold.
 \begin{enumerate}
 \item[(i)] For each $i\in \{1,\ldots, p\}$,  ${\bf S}_i:=[{\bf S}_{i,1}\cdots {\bf S}_{i,n_i}]$  is a pure row isometry.

 \item[(ii)] For each $i\in \{p+1,\ldots, k\}$,  $W_i:=[W_{i,1}\cdots W_{i,n_i}]$  is a Cuntz  row isometry.

 \item[(iii)] The associated wandering subspace $\cL_{\{1,\ldots,p\}}$ of $({\bf S}_1,\ldots, {\bf S}_p, W_{p+1},\ldots, W_k)$ is canonically isomorphic to $\cL$ and the $\{1,\ldots, p\}$-wandering data is
 $\left\{\cL, (\widetilde W_{p+1},\ldots, \widetilde W_k)\right\}$. All the other wandering data  of $({\bf S}_1,\ldots, {\bf S}_p, W_{p+1},\ldots, W_k)$  are the zero tuples.
 \end{enumerate}

\end{theorem}
\begin{proof} Let $i\in \{1,\ldots, p\}$ and $s\in \{1,\ldots, n_i\}$ and note that relation \eqref{shift} implies
\begin{equation}
\label{shift*}
\begin{split}
&{\bf S}_{i,s}^*\left( \chi_{(\alpha_1,\ldots, \alpha_p)}\otimes h\right)\\
&\qquad \qquad =\begin{cases}
\overline{\boldsymbol{\lambda}_{i,1}(s,\alpha_1)}\cdots \overline{\boldsymbol{\lambda}_{i,i-1}(s,\alpha_{i-1})}
\chi_{\alpha_1,\ldots, \alpha_{i-1}, \beta_i, \alpha_{i+1},\ldots, \alpha_p)}\otimes h,&\quad \text{ if } \alpha_i=g_s^i \beta_i\\
0, &\quad \text{ otherwise }
\end{cases}
\end{split}
\end{equation}
for any $h\in \cL$, $\alpha_1\in \FF_{n_1}^+,\ldots, \alpha_{p}\in \FF_{n_{p}}^+$. Hence, we deduce that
\begin{equation*}
\begin{split}
&\sum_{s=1}^{n_i} {\bf S}_{i,s}{\bf S}_{i,s}^*\left( \chi_{(\alpha_1,\ldots, \alpha_p)}\otimes h\right)\\
&\qquad \qquad =\begin{cases}
|\boldsymbol{\lambda}_{i,1}(s,\alpha_1)|^2\cdots |\boldsymbol{\lambda}_{i,i-1}(s,\alpha_{i-1})|^2
\chi_{(\alpha_1,\ldots, \alpha_{i-1}, \alpha_i, \alpha_{i+1},\ldots, \alpha_p)}\otimes h,&\quad \text{ if } |\alpha_i|\geq 1\\
0, &\quad \text{ otherwise }
\end{cases}\\
&\qquad \qquad =\begin{cases}
\chi_{(\alpha_1,\ldots, \alpha_p)}\otimes h,&\quad \text{ if } |\alpha_i|\geq 1\\
0, &\quad \text{ otherwise,}
\end{cases}
\end{split}
\end{equation*}
which shows that $[{\bf S}_{i,1}\cdots {\bf S}_{i, n_i}]$ is a  row isometry for any $i\in \{1,\ldots, p\}$.
Inductively, one can prove that
\begin{equation*}
\begin{split}
\left(\sum_{\alpha\in \FF_{n_i}^+, |\alpha|=q} {\bf S}_{i,\alpha} {\bf S}_{i,\alpha}^*\right)\left( \chi_{(\alpha_1,\ldots, \alpha_p)}\otimes h\right)& =\begin{cases}
\chi_{(\alpha_1,\ldots, \alpha_p)}\otimes h,&\quad \text{ if } |\alpha_i|\geq q\\
0, &\quad \text{ otherwise }
\end{cases}
\end{split}
\end{equation*}
for any $q\in \NN$.
Consequently, we have
$$
\lim_{q\to\infty}  \sum_{\alpha\in \FF_{n_i}^+, |\alpha|=q}
\left\| {\bf S}_{i,\alpha}^*\left( \chi_{(\alpha_1,\ldots, \alpha_p)}\otimes h\right)\right\|^2=0
$$
for any $h\in \cL$, $\alpha_1\in \FF_{n_1}^+,\ldots, \alpha_{p}\in \FF_{n_{p}}^+$.
Since $\sum_{\alpha\in \FF_{n_i}^+, |\alpha|=q} {\bf S}_{i,\alpha} {\bf S}_{i,\alpha}^*\leq I$, we conclude that
$[{\bf S}_{i,1}\cdots {\bf S}_{i, n_i}]$ is a pure row isometry for any $i\in \{1,\ldots, p\}$.

Now, note that,   for $i\in \{p+1,\ldots, k\}$ and $s\in \{1,\ldots, n_i\}$, relation \eqref{cuntz} shows that $W_{i,s}=D\otimes \widetilde{W}_{i,s}$, where $D\in B(\ell^2(\FF^+_{n_1}\times\cdots \times  \FF_{n_p}^+))$ is a unitary diagonal operator  and $\widetilde W_i:=[\widetilde W_{i,1}\cdots \widetilde W_{i,n_i}]$  is a Cuntz row isometry satisfying relation \eqref{doubly}. Consequently,
for each $i\in \{p+1,\ldots, k\}$,  $W_i:=[W_{i,1}\cdots W_{i,n_i}]$  is a Cuntz  row isometry satisfying
relation
\begin{equation*}
 W_{i,s}^*  W_{j,t}=\overline{\lambda_{ij}(s,t)} W_{j,t} W_{i,s}^*
\end{equation*}
for any $i,j\in \{p+1,\ldots, k\}$ with $i\neq j$ and any $s\in \{1,\ldots, n_i\}$, $t\in \{1,\ldots, n_j\}$.

Now, we prove that  the $k$-tuple $$({\bf S}_1,\ldots, {\bf S}_p, W_{p+1},\ldots, W_k)$$   is a
 doubly $\Lambda$-commuting  sequence of row isometries on the Hilbert space  $\ell^2(\FF_{n_1}^+\times\cdots \times \FF_{n_p}^+)\otimes \cL$.
  First, we show that,  if $i,j\in \{1,\ldots, p\}$ with $i\neq j$ and any $s\in \{1,\ldots, n_i\}$, $t\in \{1,\ldots, n_j\}$, then
\begin{equation}\label{lac-S} {\bf S}_{i,s}^* {\bf S}_{j,t}=\overline{\lambda_{i,j}(s,t)}{\bf S}_{j,t}{\bf S}_{i,s}^*.
\end{equation}
Assume that  $i<j$. If $\alpha_i\in \FF_{n_i}^+$ and $\alpha_i\neq g_s^i\beta_i$, then
$$
{\bf S}_{i,s}^* {\bf S}_{j,t}\left( \chi_{(\alpha_1,\ldots, \alpha_p)}\otimes h\right)=
{\bf S}_{j,t}{\bf S}_{i,s}^*\left( \chi_{(\alpha_1,\ldots, \alpha_p)}\otimes h\right)=0.
$$
On the other hand, if $\alpha_i= g_s^i\beta_i$, then we have
\begin{equation*}
\begin{split}
&{\bf S}_{i,s}^* {\bf S}_{j,t}\left( \chi_{(\alpha_1,\ldots, \alpha_p)}\otimes h\right)\\
&\quad =
\boldsymbol{\lambda}_{j,1}(t,\alpha_1)\cdots \boldsymbol{\lambda}_{j,j-1}(t,\alpha_{j-1})
{\bf S}_{i,s}^*\left( \chi_{(\alpha_1,\ldots,\alpha_{j-1}, g_t^j\alpha_j, \alpha_{j+1},\ldots, \alpha_p)}\otimes h\right)\\
&\quad=
\boldsymbol{\lambda}_{j,1}(t,\alpha_1)\cdots \boldsymbol{\lambda}_{j,j-1}(t,\alpha_{j-1})
\overline{\boldsymbol{\lambda}_{i,1}(s,\alpha_1)}\cdots \overline{\boldsymbol{\lambda}_{i,i-1}(s,\alpha_{i-1})}
\left( \chi_{(\alpha_1,\ldots, \alpha_{i-1},\beta_i,\alpha_{i+1}, \ldots, \alpha_{j-1}, g_t^j\alpha_j, \alpha_{j+1},\ldots, \alpha_p)}\otimes h\right).
\end{split}
\end{equation*}
This relation is used o deduce that
\begin{equation*}
\begin{split}
&{\bf S}_{j,t} {\bf S}_{i,s}^*\left( \chi_{(\alpha_1,\ldots, \alpha_p)}\otimes h\right)\\
&\quad =\overline{\boldsymbol{\lambda}_{i,1}(s,\alpha_1)}\cdots \overline{\boldsymbol{\lambda}_{i,i-1}(s,\alpha_{i-1})} {\bf S}_{j,t}
\left( \chi_{(\alpha_1,\ldots, \alpha_{i-1},\beta_i, \alpha_{i+1}, \ldots,  \alpha_p)}\otimes h\right)\\
&\quad=
\overline{\boldsymbol{\lambda}_{i,1}(s,\alpha_1)}\cdots \overline{\boldsymbol{\lambda}_{i,i-1}(s,\alpha_{i-1})}
\lambda_{j,1}(t,\alpha_1)\cdots \lambda_{j,i-1}(t,\alpha_{i-1})
\lambda_{j,i}(t,\beta_i) \lambda_{j,i+1}(t,\alpha_{i+1})\cdots \lambda_{j,j-1}(t, \alpha_{j-1})\\
&\qquad\qquad \times
\left( \chi_{(\alpha_1,\ldots, \alpha_{i-1},\beta_i,\alpha_{i+1}, \ldots, \alpha_{j-1}, g_t^j\alpha_j, \alpha_{j+1},\ldots, \alpha_p)}\otimes h\right)\\
&\quad=\overline{\lambda_{j,i}(t,s)}{\bf S}_{i,s}^* {\bf S}_{j,t}\left( \chi_{(\alpha_1,\ldots, \alpha_p)}\otimes h\right)=
\lambda_{i,j}(s,t) {\bf S}_{i,s}^* {\bf S}_{j,t}\left( \chi_{(\alpha_1,\ldots, \alpha_p)}\otimes h\right).
\end{split}
\end{equation*}
Consequently, relation \eqref{lac-S} holds. The case when $i>j$ follows taking adjoints and using the fact that $\overline{\lambda_{j,i}(t,s)}=\lambda_{i,j}(s,t)$.

It remains to prove that, if $i\in \{1,\ldots, p\}$,  $j\in \{p+1,\ldots, k\}$ and $s\in\{1,\ldots, n_i\}$,
$t\in \{1,\ldots, n_j\}$, then
\begin{equation}
\label{S*W}
{\bf S}_{i,s}^* W_{j,t}=\overline{\lambda_{i,j}(s,t)}W_{j,t}{\bf S}_{i,s}^*.
\end{equation}
First, note that if $\alpha_i\in \FF_{n_i}^+$ and $\alpha_i\neq g_s^i\beta_i$, then
$$
{\bf S}_{i,s}^* W_{j,t}\left( \chi_{(\alpha_1,\ldots, \alpha_p)}\otimes h\right)=
W_{j,t}{\bf S}_{i,s}^*\left( \chi_{(\alpha_1,\ldots, \alpha_p)}\otimes h\right)=0.
$$
 If $\alpha_i= g_s^i\beta_i$, then we have
\begin{equation*}
\begin{split}
&{\bf S}_{i,s}^* W_{j,t}\left( \chi_{(\alpha_1,\ldots, \alpha_p)}\otimes h\right)\\
&\quad=
\boldsymbol{\lambda}_{j,1}(t,\alpha_1)\cdots \boldsymbol{\lambda}_{j,p}(t,\alpha_{p})
{\bf S}_{i,s}^*\left( \chi_{(\alpha_1,\ldots, \alpha_p)}\otimes \widetilde W_{j,t}h\right)\\
&\quad=\boldsymbol{\lambda}_{j,1}(t,\alpha_1)\cdots \boldsymbol{\lambda}_{j,p}(t,\alpha_{p})\overline{\boldsymbol{\lambda}_{i,1}(s,\alpha_1)}\cdots \overline{\boldsymbol{\lambda}_{i,i-1}(s,\alpha_{i-1})}\left( \chi_{(\alpha_1,\ldots, \alpha_{i-1},\beta_i, \alpha_{i+1}, \ldots,  \alpha_p)}\otimes \widetilde W_{j,t}h\right),
\end{split}
\end{equation*}
which can be used to deduce that
\begin{equation*}
\begin{split}
&W_{j,t} {\bf S}_{i,s}^*\left( \chi_{(\alpha_1,\ldots, \alpha_p)}\otimes h\right)\\
&\quad =\overline{\boldsymbol{\lambda}_{i,1}(s,\alpha_1)}\cdots \overline{\boldsymbol{\lambda}_{i,i-1}(s,\alpha_{i-1})} W_{j,t}
\left( \chi_{(\alpha_1,\ldots, \alpha_{i-1},\beta_i, \alpha_{i+1}, \ldots,  \alpha_p)}\otimes h\right)\\
&\quad=
\overline{\boldsymbol{\lambda}_{i,1}(s,\alpha_1)}\cdots \overline{\boldsymbol{\lambda}_{i,i-1}(s,\alpha_{i-1})}
\lambda_{j,1}(t,\alpha_1)\cdots \lambda_{j,i-1}(t,\alpha_{i-1})
\lambda_{j,i}(t,\beta_i) \lambda_{j,i+1}(t,\alpha_{i+1})\cdots \lambda_{j,p}(t, \alpha_{p})\\
&\qquad\qquad \times
\left( \chi_{(\alpha_1,\ldots, \alpha_{i-1},\beta_i,\alpha_{i+1}, \ldots,   \alpha_p)}\otimes \widetilde W_{j,t}h\right)\\
&\quad=\overline{\lambda_{j,i}(t,s)}{\bf S}_{i,s}^* W_{j,t}\left( \chi_{(\alpha_1,\ldots, \alpha_p)}\otimes h\right)=
\lambda_{i,j}(s,t) {\bf S}_{i,s}^* W_{j,t}\left( \chi_{(\alpha_1,\ldots, \alpha_p)}\otimes h\right).
\end{split}
\end{equation*}
Therefore, relation \eqref{S*W} holds.
In what follows, we prove  part (iii) of the theorem.
Due to calculations above, we deduce that
$\prod_{i=1}^p \left(I-\sum_{s=1}^{n_i} {\bf S}_{i,s}{\bf S}_{i,s}^*\right)$ is the orthogonal projection of
 $\ell^2(\FF_{n_1}^+\times\cdots \times \FF_{n_p}^+)\otimes \cL$ onto
 $\CC \chi_{(g_0^1,\ldots, g_0^p)}\otimes \cL$. According to Theorem \ref{Wold3}, $\prod_{i=1}^p \left(I-\sum_{s=1}^{n_i} {\bf S}_{i,s}{\bf S}_{i,s}^*\right)$ is the orthogonal projection of
 $\ell^2(\FF_{n_1}^+\times\cdots \times \FF_{n_p}^+)\otimes \cL$ onto the $\{1,\ldots, p\}$-wandering subspace $\cL_{\{1,\ldots, p\}}$. Therefore,  $\cL_{\{1,\ldots, p\}}$ is canonically isomorphic to $\cL$.  Note also that under this identification, if $i\in\{p+1,\ldots, k\}$ and $s\in \{1,\ldots, n_i\}$, we have $W_{i,s}|_{\cL_{\{1,\ldots, p\}}}=\widetilde W_{i,s}$.
Due to Theorem \ref{Wold2} , all the other wandering data  of $({\bf S}_1,\ldots, {\bf S}_p, W_{p+1},\ldots, W_k)$  are the zero tuples.
The proof is complete.
\end{proof}

Given a subset $C\subset \{1,\ldots, k\}$, we introduce the twisted $\Lambda$-tensor algebra
$\otimes_{i\in C}^{\Lambda}\cO_{n_i}$ as the universal $C^*$-algebra generated by   isometries
$W_{i,s}$,  where $i\in C$, $s\in \{1,\ldots, n_i\}$,  with the property that
\begin{equation*} W_{i,s}^* W_{j,t}=\overline{\lambda_{i,j}(s,t)}W_{j,t}W_{i,s}^*
\end{equation*}
for any $i,j\in C$ with $i\neq j$ and any $s\in \{1,\ldots, n_i\}$, $t\in \{1,\ldots, n_j\}$,
and  such that
each row isometry  $W_i=[W_{i,1}\cdots W_{i,n_i}]$ satisfies the  Cuntz condition
$$W_{i,1}W_{i,1}^*+\cdots +W_{i, n_i}W_{i, n_i}^*=1,\qquad i\in C.
$$
Note that if $n_i=1$ for $i\in C$, then the corresponding twisted $\Lambda$-tensor   algebra
$\otimes_{i\in C}^{\Lambda}\cO_{n_i}$ coincides with a higher dimensional   noncommutative torus, which has been extensively studied in the literature.

We remark that versions of Theorem \ref{standard} hold true when $p=0$ or $p=k$.
Indeed, in the particular case $p=0$, we have $A=\emptyset$ and, given the data
$\left\{\cK, (\widetilde W_{1},\ldots, \widetilde W_k)\right\}$ with the property that relation \eqref{doubly} holds,
  the standard $k$-tuple
 $(W_1,\ldots, W_k)$ is  defined by  $W_{i}:=\widetilde W_{i}$. In this case, the
 $\emptyset$-wandering data of $(W_1,\ldots, W_k)$ is $\cW_\emptyset:=(\cK, (W_1,\ldots, W_n))$.
Note that all the other wandering data of $(W_1,\ldots, W_k)$ are the zero tuples.
Moreover, $(W_1,\ldots, W_k)$ provides a representation of the universal algebra
$\otimes_{i\in \{1,\ldots, k\}}^{\Lambda}\cO_{n_i}$.

In the particular case when $p=k$, we have $A=\{1,\ldots, k\}$ and the {\it standard $k$-tuple}
 $({\bf S}_1,\ldots, {\bf S}_k)$  consists of pure row isometries   on the Hilbert space  $\ell^2(\FF_{n_1}^+\times\cdots \times \FF_{n_k}^+)\otimes \cL$, which are doubly  $\Lambda$-commuting.
 The $\{1,\ldots, k\}$-wandering subspace of $({\bf S}_1,\ldots, {\bf S}_k)$ is
 $\CC \chi_{(g_0^1,\ldots, g_0^k)}\otimes \cL$ which is identified with $\cL$. With this identification the $\{1,\ldots, k\}$-wandering data of $({\bf S}_1,\ldots, {\bf S}_k)$ reduces to  $\{\cL\}$.
All the other wandering data of $({\bf S}_1,\ldots, {\bf S}_k)$ are the zero tuples.
We call $({\bf S}_1,\ldots, {\bf S}_k)$ the standard $k$-tuple with $\{1,\ldots, k\}$-wandering data $(\cL)$. When $\cL=\CC$, we use the notation $(S_1,\ldots, S_k)$  for the standard $k$-tuple.

Let  $(V_1,\ldots, V_k)$ and $(V_1',\ldots, V_k')$ be $k$-tuple of doubly $\Lambda$-commuting row isometries acting on the Hilbert spaces $\cK$ and $\cK'$, respectively. We say that they are unitarily equivalent if there is a unitary operator $U:\cK\to \cK'$ such that $UV_{i,s}=V_{i,s}'U$ for any $i\in \{1,\ldots, k\}$ and $s\in \{1,\ldots, n_i\}$.

\begin{theorem} \label{standard2}    Assume that $0\leq p\leq k\geq 1$. Let $(V_1,\ldots, V_k)$ be a $k$-tuple of doubly $\Lambda$-commuting row isometries on a Hilbert space $\cK$ such that $V_1,\ldots, V_p$ (if $p\geq 1$)  are pure row isometries and $V_{p+1},\ldots, V_k$  (if $p\leq k-1$) are Cuntz row isometries. Let $\cL$ be the $\{1,\ldots, p\}$-wandering subspace of $(V_1,\ldots, V_k)$ and let
$\widetilde W_{p+1}:=V_{p+1}|_\cL,\ldots, \widetilde W_{k}:=V_{k}|_\cL$. Then the $k$-tuple $(V_1,\ldots, V_k)$
is unitarily equivalent to
\begin{enumerate}
\item[(i)] the standard $k$-tuple  $({\bf S}_1,\ldots, {\bf S}_p, W_{p+1},\ldots, W_k)$   associated with the wandering  $\{1,\ldots, p\}$-data
 $\left\{\cL, (\widetilde W_{p+1},\ldots, \widetilde W_k)\right\}$,  if $1\leq p\leq k-1$;
 \item[(ii)] the standard $k$-tuple  $(V_1,\ldots, V_k)$    associated with the wandering
 $\emptyset$-data
 $\left\{\cK, (V_1,\ldots, V_k)\right\}$, if $p=0$;
 \item[(iii)] the standard $k$-tuple  $({\bf S}_1,\ldots, {\bf S}_k)$     associated with the wandering
 $\{1,\ldots,k\}$-data $\{\cL\}$, if $p=k$ .
 \end{enumerate}
\end{theorem}
\begin{proof}  Assume that $1\leq p\leq k-1$.
According to Theorem \ref{Wold2} and Theorem \ref{Wold3}, we have
$$
\cK=\cK_{\{1,\ldots, p\}}=\bigoplus_{\alpha_{1}\in \FF_{n_{1}}^+,\ldots, \alpha_{p}\in \FF_{n_{p}}^+} V_{1,\alpha_{1}}\cdots V_{p,\alpha_{p}}\left(\cL\right),
$$
where $\cL:=\cL_{\{1,\ldots, p\}}$ is invariant under the isometries $V_{i,s}$, where $i\in \{p+1,\ldots, k\}$ and $s\in \{1,\ldots, n_i\}$, and $[V_{i,1}|_\cL\cdots V_{i,n_i}|_\cL$ is a Cuntz isometry for all
$i\in \{p+1,\ldots, k\}$. Moreover, $\cL$ is reducing for the isometries above and, denoting
$\widetilde W_{i,s}:=V_{i,s}|_\cL$, we have
\begin{equation}\label{tilda}
\left(\widetilde W_{i,s}|_{\cL}\right)^* \left(\widetilde W_{j,t}|_{\cL}\right)=\overline{\lambda_{ij}(s,t)}\left(\widetilde W_{j,t}|_{\cL}\right)\left(\widetilde W_{i,s}|_{\cL}\right)^*
\end{equation}
for any  $i, j\in \{p+1,\ldots, k\}$ with $i\neq j$, and  $s\in \{1,\ldots, n_i\}$, $t\in \{1,\ldots, n_j\}$.
Let    $({\bf S}_1,\ldots, {\bf S}_p, W_{p+1},\ldots, W_k)$  be the standard $k$-tuple  associated with the wandering  $\{1,\ldots, p\}$-data
 $\left\{\cL, (\widetilde W_{p+1},\ldots, \widetilde W_k)\right\}$.

Now, we define the operator
$U:\cK\to \ell^2(\FF_{n_1}^+\times\cdots \times \FF_{n_p}^+)\otimes \cL$ by setting
$$
U\left(  \sum_{\alpha_{1}\in \FF_{n_{1}}^+,\ldots, \alpha_{p}\in \FF_{n_{p}}^+} V_{1,\alpha_{1}}\cdots V_{p,\alpha_{p}}  h_{(\alpha_1,\ldots, \alpha_p)}\right):=
 \sum_{\alpha_{1}\in \FF_{n_{1}}^+,\ldots, \alpha_{p}\in \FF_{n_{p}}^+}
 \chi_{(\alpha_1,\ldots, \alpha_p)}\otimes h_{(\alpha_1,\ldots, \alpha_p)},
$$
where  $h_{(\alpha_1,\ldots, \alpha_p)}\in \cL$  and $\{\chi_{(\alpha_1,\ldots, \alpha_p)}\}$ is the standard orthonormal basis
of $\ell^2(\FF_{n_1}^+\times\cdots \times \FF_{n_p}^+)$.
Since
$$
\left\|V_{1,\alpha_{1}}\cdots V_{p,\alpha_{p}}  h_{(\alpha_1,\ldots, \alpha_p)}\right\|=
\left\|\chi_{(\alpha_1,\ldots, \alpha_p)}\otimes h_{(\alpha_1,\ldots, \alpha_p)}\right\|,
$$
it is clear that $U$ is a unitary operator.
Let $i\in \{1,\ldots, p\}$, $s\in \{1,\ldots, n_i\}$, and note that
\begin{equation*}
\begin{split}
&V_{i,s}\left(V_{1,\alpha_{1}}\cdots V_{p,\alpha_{p}}  h_{(\alpha_1,\ldots, \alpha_p)}\right)\\
&\qquad=
\boldsymbol{\lambda}_{i,1}(s,\alpha_1)\cdots \boldsymbol{\lambda}_{i,i-1}(s,\alpha_{i-1})
V_{1,\alpha_{1}}\cdots V_{i-1, \alpha_{i-1}} V_{i,g_s^i\alpha_i} V_{i+1,\alpha_{i+1}}\cdots V_{p,\alpha_{p}}  h_{(\alpha_1,\ldots, \alpha_p)}.
\end{split}
\end{equation*}
Hence, and using the definition of $U$, we obtain
\begin{equation*}
\begin{split}
&UV_{i,s}U^{-1}\left(  \chi_{(\alpha_1,\ldots, \alpha_p)}\otimes h_{(\alpha_1,\ldots, \alpha_p)}\right)\\
&\qquad=\boldsymbol{\lambda}_{i,1}(s,\alpha_1)\cdots \boldsymbol{\lambda}_{i,i-1}(s,\alpha_{i-1})
\left( \chi_{(\alpha_1,\ldots,\alpha_{i-1}, g_s^i\alpha_i,\alpha_{i+1},\ldots, \alpha_p)}\otimes h_{(\alpha_1,\ldots, \alpha_p)}\right)
\end{split}
\end{equation*}
and, consequently $UV_{i,s}U^{-1}={\bf S}_{i,s}$ for any $i\in \{1,\ldots, p\}$ and  $s\in \{1,\ldots, n_i\}$.

On the other hand, due to the doubly $\Lambda$-commuting property, if $i\in \{p+1,\ldots, k\}$ and $s\in \{1,\ldots, n_i\}$, we have
$$
V_{i,s}\left(V_{1,\alpha_{1}}\cdots V_{p,\alpha_{p}}  h_{(\alpha_1,\ldots, \alpha_p)}\right)
=
\boldsymbol{\lambda}_{i,1}(s,\alpha_1)\cdots \boldsymbol{\lambda}_{i,p}(s,\alpha_{p})
V_{1,\alpha_{1}}\cdots V_{p,\alpha_p}\widetilde W_{i,s}h_{(\alpha_1,\ldots, \alpha_p)}.
$$
Consequently,
$$
UV_{i,s}U^{-1}\left(  \chi_{(\alpha_1,\ldots, \alpha_p)}\otimes h_{(\alpha_1,\ldots, \alpha_p)}\right)
=\boldsymbol{\lambda}_{i,1}(s,\alpha_1)\cdots \boldsymbol{\lambda}_{i,p}(s,\alpha_{p})
\chi_{(\alpha_1,\ldots, \alpha_p)}\otimes \widetilde W_{i,s} h_{(\alpha_1,\ldots, \alpha_p)},
$$
where  the isometries $\widetilde W_{i,s}$ satisfy  relation \eqref{tilda}.
Hence, we deduce that $UV_{i,s}U^{-1}=W_{i,s}$  for any  $i\in \{p+1,\ldots, k\}$ and $s\in \{1,\ldots, n_i\}$. This proves item (i). Note that item (ii) is obvious and the proof of (iii) is similar to that of (i), setting $p=k$.
The proof is complete.
\end{proof}

Consider now the general case when
$A=\{i_1,\ldots, i_p\}\subset \{1,\ldots, k\}$ and $A^c=\{j_1,\ldots, j_{k-p}\}$ with $i_1<\cdots <i_p$ and $j_1<\cdots <j_{k-p}$. As in the particular case when $A=\{1,\ldots, p\}$, given a $A$-wandering data   $\cW_A:=\left\{\cL, (\widetilde W_{j_1},\ldots, \widetilde W_{j_{k-p}})\right\}$, one can construct  a   $k$-tuple $V^{\cW_A}:=(V^{\cW_A}_1,\ldots, V^{\cW_A}_k)$, where
$V^{\cW_A}_i:=[V^{\cW_A}_{i,1}\cdots, V^{\cW_A}_{i,n_i}]$,  of doubly $\Lambda$-commuting row isometries  acting on the Hilbert space
$\ell^2(\FF_{n_{i_1}}^+\times \cdots \times\FF_{n_{i_p}}^+)\otimes \cL$,
that has the prescribed $A$-wandering data $\cW_A$  with the following properties:
\begin{enumerate}

\item[(i)]
if $i\in A$, then $V_i^{\cW_A}$ is a pure row isometry;
\item[(ii)]
if $i\in A^c$, then $V_i^{\cW_A}$ is a Cuntz row isometry;
\item[(iii)]
all the other wandering data are the zero tuples.
\end{enumerate}
We call $V^{\cW_A}$  the {\it standard $k$-tuple} of  doubly $\Lambda$-commuting row isometries with  prescribed $A$-wandering data  $\cW_A$.
In this setting, the analogue of Theorem \ref{standard2} is the following result. Since the proof is similar we shall omit it.

\begin{theorem} \label{standard3}  Let $1\leq p\leq k-1$ and let $A=\{i_1,\ldots, i_p\}\subset \{1,\ldots, k\}$ and $A^c=\{j_1,\ldots j_{k-p}\}$ with $j_1<\cdots <j_{k-p}$.
Suppose that $(V_1,\ldots, V_k)$ is  a $k$-tuple of doubly $\Lambda$-commuting row isometries on a Hilbert space $\cK$ such that $V_i$ is a pure row isometry  if $i\in A$ and  $V_i$  is a  Cuntz row isometry if $i\in A^c$, and let $\cW_A$ be its $A$-wandering data of $(V_1,\ldots, V_k)$. Then the $k$-tuple $(V_1,\ldots, V_k)$
is unitarily equivalent to the standard $k$-tuple     $V^{\cW_A}$ associated with the wandering data $\cW_A$.
\end{theorem}

Combining  now Theorem \ref{Wold2} and Theorem \ref{standard3}, we deduce the following Wold decomposition.

\begin{theorem}   \label{Wold4} Let $V:=(V_1,\ldots, V_k)$ be a $k$-tuple of doubly $\Lambda$-commuting row isometries. Then
$V$ is unitarily equivalent to   $k$-tuple
$$
\bigoplus_{A\subset \{1,\ldots, k\}} V^{\cW_A}:=\left(\bigoplus_{A\subset \{1,\ldots, k\}} V_1^{\cW_A},\ldots, \bigoplus_{A\subset \{1,\ldots, k\}} V_k^{\cW_A}   \right),
$$
where, for each $i\in \{1,\ldots, k \}$,
$$
\bigoplus_{A\subset \{1,\ldots, k\}} V_i^{\cW_A}:=\left[\bigoplus_{A\subset \{1,\ldots, k\}} V_{i,1}^{\cW_A}\cdots \bigoplus_{A\subset \{1,\ldots, k\}} V_{i,n_i}^{\cW_A}   \right],
$$
and
 $V^{\cW_A}$ is the  standard $k$-tuple of of doubly $\Lambda$-commuting row isometries with  prescribed $A$-wandering data  $\cW_A $, which coincides with the $A$-wandering data of $(V_1,\ldots, V_k)$.

\end{theorem}
\begin{proof}

According to Theorem \ref{Wold2}, the Hilbert space  $\cK$ admits a  unique orthogonal decomposition
$$
\cK=\bigoplus_{A\subset \{1,\ldots, k\}} \cK_A
$$
with the following properties:
\begin{enumerate}
\item[(i)]  for each  subset $A\subset \{1,\ldots, k\}$,  the subspace $\cK_A$ is  reducing  for all the isometries
$V_{i,m}$, where $i\in \{1,\ldots, k\}$ and $m\in \{1,\ldots, n_i\}$;

\item[(ii)]  if $i\in A$, then $V_i|_{\cK_A}:=[V_{i,1}|_{\cK_A}\cdots V_{i,n_i}|_{\cK_A}]$ is a pure row isometry;

\item[(iii)] if $i\in A^c$, then $V_i|_{\cK_A}:=[V_{i,1}|_{\cK_A}\cdots V_{i,n_i}|_{\cK_A}]$  is a Cuntz row isometry.
\end{enumerate}
If $A=\{i_1,\ldots, i_p\}$ and $A^c=\{j_1,\ldots j_{k-p}\}$ with $j_1<\cdots <j_{k-p}$,
 let
$\cW_A=\left( \cL_A, \left(V_{j_1}|_{\cL_A},\ldots, V_{j_{k-p}}|_{\cL_A}\right)\right)$ be  the {\it $A$-wandering data} of
 $V:=(V_1,\ldots, V_k)$, as provided by Theorem \ref{Wold2}. Let $V^{\cW_A}$  be the standard $k$-tuple  of  doubly $\Lambda$-commuting row isometries with  prescribed $A$-wandering data  $\cW_A$. According to Theorem \ref{standard2} and Theorem \ref{standard3}, there is a unitary operator $U_A:\cK_A\to \ell^2(\times_{i\in A}\FF_{n_i}^+)\otimes \cL_A$ such that
 $$U_A\left(V_{i,s}|_{\cK_A}\right)=V_{i,s}^{\cW_A} U_A
 $$
for any $i\in\{1,\ldots, k\}$ and $s\in \{1,\ldots, n_i\}$. Consequently, we deduce that

$$
UV_{i,s}=\left(\bigoplus_{A\subset \{1,\ldots, k\}}V_{i,s}^{\cW_A}\right) U,
 $$
where
 $$
 U=\bigoplus_{A\subset \{1,\ldots, k\}}U_A:\bigoplus_{A\subset \{1,\ldots, k\}}\cK_A\to \bigoplus_{A\subset \{1,\ldots, k\}}\left(\ell^2(\times_{i\in A}\FF_{n_i}^+)\otimes \cL_A\right).
 $$
 The proof is complete.
 \end{proof}

\bigskip

\section{Classification of doubly $\Lambda$-commuting row isometries}

In this section, we  present classification  and parametrization results   for  the $k$-tuples of doubly $\Lambda$-commuting row isometries. In addition, we obtain  a description and parametrization of the irreducible $k$-tuples of doubly $\Lambda$-commuting row isometries.

Let  $V:=(V_1,\ldots, V_k)$ be $k$-tuple of doubly $\Lambda$-commuting row isometries acting on the Hilbert space $\cK$ and let and $V':=(V_1',\ldots, V_k')$ be $k$-tuple of doubly $\Lambda'$-commuting row isometries acting on the Hilbert space $\cK'$.   Note that if $V$ and $V'$ are unitarily equivalent, then $\Lambda_{ij} =\Lambda'_{ij}$ for $i\neq j$. Indeed, assume that
  there is a unitary operator $U:\cK\to \cK'$ such that $UV_{i,s}=V_{i,s}'U$ for any $i\in \{1,\ldots, k$ and $s\in \{1,\ldots, n_i\}$.  Then, using the $\Lambda$-commutation relation \eqref{lac}, we obtain
  $$
  UV_{i,s}V_{j,t}U^{-1}=V_{i,s}'V_{j,t}'=\lambda_{i,j}'(s,t) V_{j,t}'V_{i,s}'=
  \lambda_{i,j}'(s,t)UV_{j,t}V_{i,s}U^{-1}
  $$
which implies $V_{i,s}V_{j,t}=\lambda_{i,j}'(s,t)V_{j,t}V_{i,s}$ and together with relation \eqref{lac} and the fact that $V_{i,s}$ are isometries  show that $\lambda_{i,j}=\lambda_{i,j}'$.

 Let $A\subset \{1,\ldots, k\}$ and $A^c=\{j_1,\ldots j_{k-p}\}$ with  $j_1<\cdots <j_{k-p}$,  and  consider two   $A$-wandering data
  $(\cL_1, (W_{j_1},\ldots, W_{j_{k-p}}))$ and $(\cL_2, (U_{j_1},\ldots, U_{j_{k-p}}))$. We say that the $A$-wandering data are equivalent if there is a unitary operator  $\varphi:\cL_1\to \cL_2$ such that
  $\varphi W_{j,t}=U_{j,t}\varphi$ for any $j\in \{j_1,\ldots, j_{k-p}\}$ and $t\in \{1,\ldots, n_j\}$.

\begin{theorem} \label{uni-equi} Let  $V:=(V_1,\ldots, V_k)$ and $V':=(V_1',\ldots, V_k')$ be $k$-tuples of doubly $\Lambda$-commuting row isometries acting on the Hilbert space $\cK$ and   $\cK'$, respectively. Then $V$ is unitarily equivalent to $V'$ if and only if, for any $A\subset \{1,\ldots, k\}$, the $A$-wandering data of $V$ is unitarily equivalent to the $A$-wandering data of $V'$.
\end{theorem}
\begin{proof}
Assume that  $V$ and $V'$ are unitarily equivalent and let  $U:\cK\to \cK'$ be a unitary operator such that  $UV_{i,s}=V_{i,s}'U$ for any $i\in \{1,\ldots, k\}$ and $s\in \{1,\ldots, n_i\}$. Then, we have
$U(\ker V_{i,s}^*)=\ker V_{i,s}^*$.  If $A=\{i_1,\ldots, i_p\}\subset \{1,\ldots, k\}$ and $A^c=\{j_1,\ldots j_{k-p}\}$ with  $j_1<\cdots <j_{k-p}$, we can use Theorem \ref{Wold3}, to deduce that
$$
U\cL_A=\cL_A',\quad U\cK_A=\cK_A',  \quad (U |_{\cL_A})(V_{i,s}|_{\cL_A})=
(V_{i,s}|_{\cL_A'})(U|_{\cL_A}),
$$
where $\cL_A$ and $\cL_A'$ are the $A$-wandering subspace of $V$ and $V'$, respectively.
This shows that the wandering data $\cW_A:=\{\cL_A, (V_{j_1}|_{\cL_A},\ldots, V_{j_{k-p}}|_{\cL_A})\}$ and   $\cW_A':=\{\cL_A', (V_{j_1}'|_{\cL_A'},\ldots, V_{j_{k-p}}'|_{\cL_A'})\}$  are unitarily equivalent.

Conversely, assume that the wandering data $\cW_A$ and $\cW_A'$ are unitarily equivalent, i.e.
there is a unitary operator $U_{\cL_A}:\cL_A\to \cL_A'$ such that
$$
U _{\cL_A}(V_{i,s}|_{\cL_A})=
(V_{i,s}|_{\cL_A'})U_{\cL_A}
$$
for any $i\in \{1,\ldots, k\}$ and $s\in \{1,\ldots, n_i\}$.
We consider the case when $1\leq p\leq k-1$ (the cases $p=0$ and $p=k$  can be treated similarly). Due to Theorem \ref{Wold3}, we have
$$
\cK_A=\bigoplus_{\alpha_{i_1}\in \FF_{n_{i_1}}^+,\ldots, \alpha_{i_p}\in \FF_{n_{i_p}}^+} V_{i_1,\alpha_{i_1}}\cdots V_{i_p,\alpha_{i_p}}\left(\cL_A\right)
$$
and  a similar relation  holds for $\cK_A'$. Define the operator $U_A:\cK_A\to \cK_A'$ by setting
$$
U_A\left( \sum_{\alpha_{i_1}\in \FF_{n_{i_1}}^+,\ldots, \alpha_{i_p}\in \FF_{n_{i_p}}^+} V_{i_1,\alpha_{i_1}}\cdots V_{i_p,\alpha_{i_p}} h_{(\alpha_{i_1},\ldots, \alpha_{i_p})} \right)
=\sum_{\alpha_{i_1}\in \FF_{n_{i_1}}^+,\ldots, \alpha_{i_p}\in \FF_{n_{i_p}}^+} V_{i_1,\alpha_{i_1}}'\cdots V_{i_p,\alpha_{i_p}}' U_{\cL_A} h_{(\alpha_{i_1},\ldots, \alpha_{i_p})},
$$
where $h_{(\alpha_{i_1},\ldots, \alpha_{i_p})}\in \cL_A$.
It is easy to see that $U_A$ is unitary and
$$
U _{A}(V_{i,s}|_{\cK_A})=
(V_{i,s}|_{\cK_A'})U_{A}
$$
for any $i\in \{1,\ldots, k\}$ and $s\in \{1,\ldots, n_i\}$. The latter relation is due to the definition of $U_A$ and  the $\Lambda$-commutation relations \eqref{lac} for the isometries $\{V_{i,s}\}$ and $\{V_{i,s}'\}$, respectively.
Now, employing Theorem \ref{Wold2}, it is clear that  the unitary operator  $$
 U:=\bigoplus_{A\subset \{1,\ldots, k\}}U_A:\cK=\bigoplus_{A\subset \{1,\ldots, k\}}\cK_A\to \cK'=\bigoplus_{A\subset \{1,\ldots, k\}}\cK_A'
 $$
 satisfies the relation
$
U V_{i,s}=
V_{i,s}U
$
for any $i\in \{1,\ldots, k\}$ and $s\in \{1,\ldots, n_i\}$.
The proof is complete.
\end{proof}

\begin{theorem} There is a one-to-one correspondence  between the unitary equivalence classes of $k$-tuples of doubly $\Lambda$-commuting row isometries and the enumerations of  $2^k$ unitary equivalence classes of unital representations of the   twisted $\Lambda$-tensor algebras $\otimes_{i\in A^c}^{\Lambda}\cO_{n_i}$,  as $A$ is any subset of $\{1,\ldots, k\}$, where $\cO_{n_i}$ is the Cuntz algebra with $n_i$ generators and $A^c$ is the complement of $A$ in $\{1,\ldots,k\}$.
\end{theorem}
\begin{proof}

Let  $\cV$ denote the set of all unitary equivalences classes of $k$-tuples of doubly $\Lambda$-commuting row isometries, and let
$$\cR:=\left\{ \times_{A\subset \{1,\ldots, k\}} \widehat \pi_{A^c} : \pi_A\in Rep(\otimes_{i\in A^c}^\Lambda\cO_{n_i})\right\},
$$
where $\widehat \pi_{A^c}$ is  the unitary equivalence class of the representations of  $\otimes_{i\in A^c}^\Lambda\cO_{n_i}$  associated with $\pi_{A^c}$.
If $V:=(V_1,\ldots, V_k)$ is a $k$-tuple of doubly $\Lambda$-commuting row isometries, denote by $\widehat V$
the unitary equivalence  class  of $V$. According to  Theorem \ref{Wold4},  to each $k$-tuple $V$ corresponds the set $\{\cW_A\}_{A\subset \{1,\ldots, k\}}$ of   $A$-wandering data $\cW_A=\left( \cL_A, \left(V_{j_1}|_{\cL_A},\ldots, V_{j_{k-p}}|_{\cL_A}\right)\right)$, indexed by all subsets $A$ of $\{1,\ldots, n\}$, where $A^c=\{j_1,\ldots, j_{k-p}\}$. Due to Theorem \ref{Wold3} (part (i)  and (ii)),  the $A$-wandering data $\cW_A$   can be used to  generate a  unital representation
$\pi_{A^c}$ of the universal algebra  $\otimes_{i\in A^c}^\Lambda\cO_{n_i}$. Let $\widehat \pi_{A^c}$ be the unitary equivalence class of the representations of  $\otimes_{i\in A^c}^\Lambda\cO_{n_i}$  which has as representative  $\pi_{A^c}$.
We define the map $\Psi:\cV\to \cR$ by setting
$$\Psi(\widehat V)=\times_{A\subset \{1,\ldots, k\}} \widehat \pi_{A^c}.
$$
Due to Theorem \ref{uni-equi}, the map $\Psi$ is well-defined and injective. On the other hand, according to  the results of Section 2, the map $ \Psi $ is surjective. The proof is complete.
\end{proof}

In what follows we obtain a description and parametrization of the irreducible $k$-tuples of doubly $\Lambda$-commuting row isometries.

\begin{theorem} \label{irreductible} Let  $V:=(V_1,\ldots, V_k)$ be  a $k$-tuple of doubly $\Lambda$-commuting row isometries acting on a Hilbert space $\cK\neq \{0\}$. Then the following statements are equivalent.
\begin{enumerate}
\item[(i)] The $C^*$-algebra $C^*(V_1,\ldots, V_k)$ is irreducible.

\item[(ii)] There is a subset $A\subset \{1,\ldots, k\}$ (possible empty) with  $A^c=\{j_1,\ldots, j_{k-p}\}$ such that $V$ is unitarily equivalent to a standard $k$-tuple $V^{\cW_A}$ of doubly $\Lambda$-commuting row isometries associated with a wandering data $\cW_A:=\{\cL_A,(W_{j_1},\ldots, W_{j_{k-p}})\}$, where the $A$-wandering subspace $\cL_A$ has only trivial subspaces that are invariant under $C^*(W_{j_1},\ldots, W_{j_{k-p}})$.
\end{enumerate}
Moreover, two such irreducible $k$-tuples $V$ and $V'$ are unitarily equivalent if and only if the corresponding sets $A$ and $A'$ are equal and the wandering data  $\cW_A$ and $\cW_{A'}$ are unitarily equivalent.
\end{theorem}
\begin{proof}
Assume that item (i) holds. According to Theorem \ref{Wold2}, there exists exactly one subset $A\subset \{1,\ldots, k\}$ (possible empty) such that $\cK=\cK_A\neq \{0\}$ and the following properties hold:
\begin{enumerate}
\item[(i)]   if $i\in A$, then $V_i:=[V_{i,1}\cdots V_{i,n_i}]$ is a pure row isometry;

\item[(ii)] if $i\in A^c$, then $V_i$  is a Cuntz row isometry.

\end{enumerate}
Using Theorem \ref{standard3}, we deduce that $V$ is unitarily equivalent to the standard $k$-tuple
$V^{\cW_A}$, where $\cW_A:=\{\cL_A,(W_{j_1},\ldots, W_{j_{k-p}})\}$ is the  wandering data of $V$. Now, item (ii) follows.

To prove the implication (ii)$\implies$(i), assume that  item (ii) holds. Since $\cK\neq \{0\}$, we must have $\cL_A\neq \{0\}$. Due two Theorem \ref{standard3}, we can assume that $V$ is the standard $k$-tuple associated with $A$. For  simplicity of exposition, we also assume that
$A=\{1,\ldots, p\}\subset\{1,\ldots, k\}$, $p\geq 1$, and, therefore, $V=({\bf S}_1, \ldots, {\bf S}_p,W_{p+1},\ldots, W_{k-p})$.
Let $\cM\subset \ell^2(\FF_{n_1}^+\times\cdots \times  \FF_{n_p}^+)\otimes \cL_A$ be a non-zero subspace which is invariant under $C^*({\bf S}_1, \ldots, {\bf S}_p,W_{p+1},\ldots, W_{k-p})$.
If
$$g= \sum_{\alpha_{1}\in \FF_{n_{1}}^+,\ldots, \alpha_{p}\in \FF_{n_{p}}^+}
 \chi_{(\alpha_1,\ldots, \alpha_p)}\otimes h_{(\alpha_1,\ldots, \alpha_p)},\qquad
 h_{(\alpha_1,\ldots, \alpha_p)}\in \cL_A,
 $$
is a non-zero element in $\cM$, then there exists $h_{(\beta_1,\ldots, \beta_p)}\neq \{0\}$
Note that
\begin{equation} \label{PL}
P_{\cL_A} {\bf S}_{1,\beta_1}^*\cdots {\bf S}_{p,\beta_p}^* g=c_{(\beta_1,\ldots, \beta_p)} \chi_{(g_0^1,\ldots, g_0^p)}\otimes h_{(\beta_1,\ldots, \beta_p)},
\end{equation}
for some constant $c_{(\beta_1,\ldots, \beta_p)} \in \TT$. On the other hand, using the definition of the standard $k$-tuple $(S_1,\ldots, S_k)$ (see the proof of Theorem \ref{standard}),  we deduce that
$$
(id-\Phi_{{\bf S}_1})\circ\cdots \circ (id-\Phi_{{\bf S}_k})(I)=P_{\cL_A},
$$
which together with relation \eqref{PL} and the fact that $\cM$ is reducing subspace for  all the shifts $\{{\bf S}_{i,s}\}$ imply
$\chi_{(g_0^1,\ldots, g_0^p)}\otimes h_{(\beta_1,\ldots, \beta_p)}\in \cM$.
Since $\cL_A$ has only trivial subspace that is invariant under $C^*(W_{j_1},\ldots, W_{j_{k-p}})$, we deduce that  $\cL_A\subset \cM$. Since $\cM$ is also invariant subspace under the isometries $\{{\bf S}_{i,s}\}$,  we deduce that
$ \ell^2(\FF_{n_1}^+\times\cdots \times  \FF_{n_p}^+)\otimes \cL_A\subset \cM$. Consequently, we have $ \ell^2(\FF_{n_1}^+\times\cdots \times  \FF_{n_p}^+)\otimes \cL_A= \cM$, which proves the implication (ii)$\implies$(i).
The last part of the theorem follows from the results above, Theorem \ref{Wold2}, and Theorem \ref{uni-equi}.
The proof is complete.
\end{proof}

Now, one can easily obtain the following result.
\begin{corollary}
The unitary equivalence classes of the non-zero irreducible representations of the $C^*$-algebra generated by $k$-tuples of doubly $\Lambda$-commuting  row isometries  are parameterised by the unitary equivalence classes of the non-zero irreducible representations of the $2^k$   universal  $\Lambda$-tensor algebras $\otimes_{i\in A^c}^{\Lambda}\cO_{n_i}$.
\end{corollary}

An important particular case of Theorem \ref{irreductible} is the following.

\begin{corollary} \label{irred}
If $S=(S_1,\ldots, S_k)$ is the standard $k$-tuple of pure row isometries acting on the Hilbert space $\ell^2(\FF_{n_1}^+\times \cdots \times \FF_{n_k}^+)$, then the $C^*$-algebra  $C^*(\{S_{i,s}\})$  is irreducible.
\end{corollary}

\bigskip

\section{Regular $\Lambda$-polyballs, noncommutative Berezin transforms, and von Neumann inequalities}

 We saw, in Section 2, that the pure $k$-tuples of doubly  $\Lambda$-commuting  row isometries are  unitarily equivalent to  the standard $k$-tuples  ${\bf S}=({\bf S}_1,\ldots, {\bf S}_k)$ acting on $\ell^2(\FF_{n_1}^+\times\cdots \times  \FF_{n_k}^+)\otimes \cD$, where $\cD$ is a Hilbert space. A natural question that arises is the following. What are the $k$-tuples  $T=(T_1,\ldots, T_k)$ of row operators $T_i=[T_{i,1}\ldots  T_{i,n_i}]$, acting on a Hilbert space $\cH$, which admit ${\bf S}$ as universal model, i.e.
 there is a Hilbert space $\cD$  such that  $\cH$ is jointly co-invariant for  ${\bf S}_{i,s}$ and
 $$
 T_{i,s}^*={\bf S}_{i,s}^*|_\cH,
 $$
 for any $i\in \{1,\ldots, k\}$ and $j\in \{1,\ldots, n_i\}$.  In this section, we answer this question introducing   the regular $\Lambda$-polyball ${\bf B}_\Lambda(\cH)$.  Employing noncommutative Berezin transforms, we provide a noncommutative von Neumann inequality  for the elements of ${\bf B}_\Lambda(\cH)$ and show that the $C^*(\{S_{i,s}\})$   is the universal  $C^*$-algebra generated by  a $k$-tuple of  doubly $\Lambda$-commuting  row isometries.

For each $i,j\in \{1,\ldots, k\}$ with $i\neq j$, let $\Lambda_{ij}:=[\lambda_{i,j}(s,t)]$, where $s\in \{1,\ldots, n_i\}$ and $t\in \{1,\ldots, n_j\}$ be an $n_i\times n_j$-matrix  with the entries in the torus $\TT:=\{z\in \CC: \  |z|=1\}$, and assume that $\Lambda_{j,i}=\Lambda_{i,j}^*$. Given   row contractions $T_i:=[T_{i,1}\cdots T_{i,n_i}]$, $i\in \{1,\ldots, k\}$,  acting on a Hilbert space $\cH$, we say that  $T=(T_1,\ldots, T_k)$  is a $k$-tuple of {\it  $\Lambda$-commuting row contractions}  if
\begin{equation}
\label{commuting}
 T_{i,s} T_{j,t}=\lambda_{ij}(s,t)T_{j,t}T_{i,s}
\end{equation}
for any $i,j\in \{1,\ldots, k\}$ with $i\neq j$ and any $s\in \{1,\ldots, n_i\}$, $t\in \{1,\ldots, n_j\}$.
 If, in addition, the relation
 \begin{equation}
 \label{*-commuting}
  T_{i,s}^* T_{j,t}=\overline{\lambda_{ij}(s,t)}T_{j,t}T_{i,s}^*
\end{equation}
is satisfied, we say that $T$ is  a
  $k$-tuple of {\it doubly $\Lambda$-commuting row contractions}. We denote by   ${\bf B}^d_\Lambda(\cH)$ the set of all  $k$-tuples of doubly $\Lambda$-commuting row contractions.
Finally, we say that $T$ is in the {\it regular $\Lambda$-polyball}, which we denote by ${\bf B}_\Lambda(\cH)$, if $T$ is  a $k$-tuple of  $\Lambda$-commuting row contractions  and
$$
\Delta_{rT}(I):=
(id-\Phi_{rT_1})\circ\cdots \circ (id-\Phi_{rT_k})(I)\geq 0,\qquad r\in [0,1),
$$
where $\Phi_{rT_i}:B(\cH)\to B(\cH)$  is the completely positive linear map defined by $\Phi_{rT_i}(X):=\sum_{s=1}^{n_i} r^2T_{i,s}XT_{i,s}^*$.

\begin{lemma}\label{commutative} If $T=(T_1,\ldots, T_k)$  is a $k$-tuple of  $\Lambda$-commuting row contractions, then
$$\Phi_{T_i} \Phi_{T_j}=\Phi_{T_j}\Phi_{T_i},\qquad i,j\in \{1,\ldots, k\}.
$$
\end{lemma}
\begin{proof} Due to relation \eqref{commuting} and  the fact that  $\lambda_{ij}(s,t)\in \TT$, we have
\begin{equation*}
\begin{split}
\Phi_{T_i} \Phi_{T_j}(X)&= \sum_{s=1}^{n_i} \sum_{t=1}^{n_j} T_{i,s} T_{j,t}XT_{j.t}^* T_{i,s}^*\\
&= \sum_{t=1}^{n_j}\sum_{s=1}^{n_i} |\lambda_{i,j}(s,t)|^2  T_{j,t}T_{i,s}XT_{i,s}^*T_{j.t}^*\\
&=\Phi_{T_j}\Phi_{T_i}(X).
\end{split}
\end{equation*}
The proof is complete.
\end{proof}

\begin{proposition} \label{regular-equiv}
Let $T=(T_1,\ldots, T_k)$  be a $k$-tuple of  $\Lambda$-commuting row contractions. Then the following statements are equivalent:
\begin{enumerate}
\item[(i)] $T\in {\bf B}_\Lambda(\cH)$;
\item[(ii)] $
(id-\Phi_{rT_1})\circ\cdots \circ (id-\Phi_{rT_k})(I)\geq 0,\qquad r\in [0,1)$;

\item[(iii)] $(id-\Phi_{T_1})^{p_1}\circ \cdots \circ (id-\Phi_{T_k})^{p_k}(I)\geq 0$ for any $p_i\in \{0,1\}$.
\end{enumerate}
If, in addition, $T$ is a pure $k$-tuple, i.e. $\Phi_{T_i}^p(I)\to 0$ strongly as $p\to \infty$, then
$T\in {\bf B}_\Lambda(\cH)$ if and only if
$$
(id-\Phi_{T_1})\circ\cdots \circ (id-\Phi_{T_k})(I)\geq 0.
$$
\end{proposition}
\begin{proof} The equivalence of (i) and (ii) is due to the definition. Let us prove that (iii)$\implies$(ii).  Assume that (iii) holds. Then $
(id-\Phi_{T_1})\circ\cdots \circ (id-\Phi_{T_k})(I)\geq 0
$
and, consequently,
$\Phi_{T_1}(\Delta_{(T_2,\ldots, T_k)}(I))\leq \Delta_{(T_2,\ldots, T_k)}(I)$, where
$$\Delta_{(T_2,\ldots, T_k)}(I):=(id-\Phi_{T_2})\circ\cdots \circ (id-\Phi_{T_k})(I)\geq 0.
$$
It is clear now that $0\leq \Phi_{rT_1}(\Delta_{(T_2,\ldots, T_k)}(I))\leq \Delta_{(T_2,\ldots, T_k)}(I)$
for any $r\in [0,1)$, which implies the inequality
$
(id-\Phi_{rT_1})\circ\cdots \circ (id-\Phi_{T_k})(I)\geq 0
$
and, due to the commutativity of $\Phi_{T_1},\ldots, \Phi_{T_k}$ (see Lemma \ref{commutative}), we deduce that
\begin{equation}\label{rT}
(id-\Phi_{T_2})\circ\cdots \circ (id-\Phi_{T_k})(id-\Phi_{rT_1})(I)\geq 0.
\end{equation}
A similar  argument  as above, starting with the inequality
$
(id-\Phi_{T_1})\circ(id-\Phi_{T_3})\circ \cdots \circ (id-\Phi_{T_k})(I)\geq 0
$
leads to the inequality  $(id-\Phi_{T_3})\circ\cdots \circ (id-\Phi_{T_k})(id-\Phi_{rT_1})(I)\geq 0$.
Repeating the argument above but starting with the inequality \eqref{rT}, shows that
$$
(id-\Phi_{T_3})\circ\cdots \circ (id-\Phi_{T_k})(id-\Phi_{rT_1})(id-\Phi_{rT_2})(I)\geq 0.
$$
Iterating this process, we conclude that
$(id-\Phi_{rT_1})\circ\cdots \circ (id-\Phi_{rT_k})(I)\geq 0$ for any $r\in [0,1)$, which completes the proof of the implication
(iii)$\implies$(ii).

Since $(rT_1,\ldots, rT_k)\in {\bf B}_\Lambda(\cH)$ is  a pure $k$-tuple, the implication (ii)$\implies$(iii) will follow immediately if we can prove that for any pure $k$-tuple $A=(A_1,\ldots, A_k)\in {\bf B}_\Lambda(\cH)$, we have
 $(id-\Phi_{A_1})^{p_1}\circ \cdots \circ (id-\Phi_{A_k})^{p_k}(I)\geq 0$ for any $p_i\in \{0,1\}$.
We prove the latter statement.  Since $A=(A_1,\ldots, A_k)\in {\bf B}_\Lambda(\cH)$, it is clear that
$(id-\Phi_{A_1})\circ\cdots \circ (id-\Phi_{A_k})(I)\geq 0$. Hence, we deduce that
$\Phi_{A_1}(\Delta_{(A_2,\ldots, A_k)}(I))\leq \Delta_{(A_2,\ldots, A_k)}(I)$, where
$$
\Delta_{(A_2,\ldots, A_k)}(I):=(id-\Phi_{A_2})\circ\cdots \circ (id-\Phi_{A_k})(I)
$$
is a self-adjoint operator. Now, it is easy to see that
$\Phi_{A_1}^m(\Delta_{(A_2,\ldots, A_k)}(I))\leq \Delta_{(A_2,\ldots, A_k)}(I)$, $m\in \NN$,  and
$$
-\|\Delta_{(A_2,\ldots, A_k)}(I)\| \Phi_{A_1}^m(I)\leq \Phi_{A_1}^m(\Delta_{(A_2,\ldots, A_k)}(I))
\leq \|\Delta_{(A_2,\ldots, A_k)}(I)\| \Phi_{A_1}^m(I)
$$
for any $m\in \NN$. Taking $m\to \infty$ and using the act that $\Phi_{A_1}^m(I)\to 0$ strongly as $m\to\infty$, we conclude that $\Delta_{(A_2,\ldots, A_k)}(I)\geq 0$.
Using  similar arguments and the commutativity of the maps $\Phi_{A_1},\ldots, \Phi_{A_k}$, we can deduce that
$(id-\Phi_{A_1})^{p_1}\circ \cdots \circ (id-\Phi_{A_k})^{p_k}(I)\geq 0$ for any $p_i\in \{0,1\}$, which proves our assertion. Note also that the latter result  and the implication (iii)$\implies$(ii) prove also the last part of the proposition.
\end{proof}

 \begin{lemma}  \label{tech} Let  $T=(T_1,\ldots, T_k)$ be a $k$-tuple of doubly $\Lambda$-commuting row contractions.
 If $i,j\in \{1,\ldots, k\}$ with $i\neq j$, then   $T_{i,s}$ commutes with $T_{j,\alpha} T_{j,\alpha}^*$ for any $s\in \{1,\ldots, n_i\}$ and $\alpha\in \FF_{n_j}^+$.
 Moreover,  for any $i_1,\ldots, i_p$ distinct elements in $\{1,\ldots, k\}$ and $\alpha_1\in \FF_{n_{i_1}}^+$,\ldots, $\alpha_p\in \FF_{n_{i_p}}^+$,
  $$\left(T_{i_1,\alpha_1}T_{i_1,\alpha_1}^*\right)\cdots \left(T_{i_p,\alpha_p}T_{i_p,\alpha_p}^*\right)
  =T_{i_1,\alpha_1}\cdots T_{i_p,\alpha_p} T_{i_p,\alpha_p}^*\cdots T_{i_1,\alpha_1}^*.
  $$
 \end{lemma}

 \begin{proof}
 Let $\alpha=g_{p_1}^j\cdots g_{p_m}^j\in \FF_{n_j}^+$, where $p_1,\ldots, p_m\in \{1,\ldots, n_j\}$. Using the relations  \eqref{commuting}, \eqref{*-commuting},  and the fact that  $\lambda_{j,i}(t,s)=\overline{\lambda_{i,j}(s,t)}$ and $\lambda_{i,j}(s,t)\in \TT$, we have
\begin{equation*}
\begin{split}
T_{i,s}\left(T_{j,\alpha}T_{j,\alpha}^*\right)&= T_{i,s}T_{j,p_1}\cdots T_{j,p_m}T_{j,p_m}^*\cdots T_{j,p_1}^*\\
&=\lambda_{i,j}(s,p_1)\cdots \lambda_{i,j}(s,p_m)T_{j,p_1}\cdots T_{j,p_m} T_{i,s}
T_{j,p_m}^*\cdots T_{j,p_1}^*\\
&=
\lambda_{i,j}(s,p_1)\cdots \lambda_{i,j}(s,p_m) \lambda_{j,i}(p_m,s)\cdots \lambda_{j,i}(p_1,s)T_{j,p_1}\cdots T_{j,p_m}
T_{j,p_m}^*\cdots T_{j,p_1}^*T_{i,s}\\
&=|\lambda_{i,j}(s,p_1)|^2\cdots |\lambda_{i,j}(s,p_m)|^2T_{j,p_1}\cdots T_{j,p_m}
T_{j,p_m}^*\cdots T_{j,p_1}^*T_{i,s}\\
&=\left(T_{j,\alpha}T_{j,\alpha}^*\right)T_{i,s},
\end{split}
\end{equation*}
which proves the first part of the proposition. Note also that $T_{i,s}^*$ commutes with $T_{j,\alpha} T_{j,\alpha}^*$ for any $s\in \{1,\ldots, n_i\}$ and $\alpha\in \FF_{n_j}^+$. This property will be  used repeatedly  to prove, by induction,  the last part of the proposition.
Indeed,  we have
$$(T_{i_1,\alpha_1}T_{i_1,\alpha_1}^*)(T_{i_2,\alpha_2}T_{i_2,\alpha_2}^*)
=
T_{i_1,\alpha_1}T_{i_2,\alpha_2} T_{i_2,\alpha_2}^* T_{i_1,\alpha_1}^*.
$$
Assume that
$$(T_{i_1,\alpha_1}T_{i_1,\alpha_1}^*)\cdots (T_{i_p,\alpha_p}T_{i_p,\alpha_p}^*)
=
T_{i_1,\alpha_1}\cdots T_{i_p,\alpha_p} T_{i_p,\alpha_p}^* \cdots T_{i_1,\alpha_1}^*.
$$
Consequently, using the $\Lambda$-commutation relation \eqref{commuting}, we obtain
\begin{equation*}
\begin{split}
&(T_{i_1,\alpha_1}T_{i_1,\alpha_1}^*)\cdots (T_{i_p,\alpha_p}T_{i_p,\alpha_p}^*)(T_{i_{p+1},\alpha_{p+1}}T_{i_{p+1},\alpha_{p+1}}^*)\\
&\qquad\qquad=T_{i_{p+1},\alpha_{p+1}}\left[(T_{i_1,\alpha_1}T_{i_1,\alpha_1}^*)\cdots (T_{i_p,\alpha_p}T_{i_p,\alpha_p}^*)\right]T_{i_{p+1},\alpha_{p+1}}^*\\
&\qquad\qquad=
T_{i_{p+1},\alpha_{p+1}}T_{i_1,\alpha_1}\cdots T_{i_p,\alpha_p} T_{i_p,\alpha_p}^* \cdots T_{i_1,\alpha_1}^*T_{i_{p+1},\alpha_{p+1}}^*\\
&\qquad\qquad =
T_{i_{p+1},\alpha_{p+1}}T_{i_1,\alpha_1}\cdots T_{i_p,\alpha_p} T_{i_{p+1},\alpha_{p+1}}T_{i_{p+1},\alpha_{p+1}}^*T_{i_p,\alpha_p}^* \cdots T_{i_1,\alpha_1}^*T_{i_{p+1},\alpha_{p+1}}^*.
\end{split}
\end{equation*}
 The proof is complete.
 \end{proof}

 \begin{proposition} Let $T=(T_1,\ldots, T_k)$ be  a $k$-tuple of row contractions with $T_i:=[T_{i,1}\cdots T_{i,n_i}]$ and $T_{i,s}\in B(\cH)$. Then the following statements hold.
 \begin{enumerate}
   \item[(i)]
  If $T=(T_1,\ldots, T_k)$ is a $k$-tuple of doubly $\Lambda$-commuting row contractions, then $T\in {\bf B}_\Lambda(\cH)$.
   \item[(ii)]
   If  $T=(T_1,\ldots, T_k)$  is a $k$-tuple of $\Lambda$-commuting row contractions and
   $\sum_{i=1}^k \sum_{s=1}^{n_i} T_{i,s}T_{i,s}^*\leq I_\cH$, then $T\in {\bf B}_\Lambda(\cH)$.
    \item[(iii)]
    If $T$ is $\Lambda$-commuting and $\sum_{s=1}^{n_i} T_{i,s}T_{i,s}^*=I_\cH$ for any $i\in \{1,\ldots, k\}$, then
    $T\in {\bf B}_\Lambda(\cH)$.
 \end{enumerate}

 \end{proposition}
 \begin{proof}
 To prove (i),  note that Lemma \ref{tech}  shows that
 $$
 (T_{i,s}T_{i,s}^*)(T_{j,t}T_{j,t}^*)=(T_{j,t}T_{j,t}^*)(T_{i,s}T_{i,s}^*)
 $$
 for any $i,j\in\{1,\ldots,k\}$ with $i\neq j$. Consequently,
 $$
 \prod_{i=1}^k \left(I-\sum_{s=1}^{n_i} r^2T_{i,s}T_{i,s}^*\right)\geq 0
 $$
 for any $r\in [0,1)$. On the other hand,  using  again Lemma \ref{tech} part (iii), we can prove that
 $$
 \Delta_{rT}(I):=(id-\Phi_{rT_k})\circ\cdots \circ (id-\Phi_{rT_1})(I)=\prod_{i=1}^k \left(I-\sum_{s=1}^{n_i} r^2T_{i,s}T_{i,s}^*\right).
 $$
 Now it is  clear that $\Delta_{rT}(I)\geq 0$ for any $r\in [0,1)$, which proves part (i).

 To part (ii), assume that $\sum_{i=1}^k \sum_{s=1}^{n_i} T_{i,s}T_{i,s}^*\leq I_\cH$.
  Note that $I\geq (id-\Phi_{rT_1}(I)\geq 0$ $r\in [0,1)$,  implies
 \begin{equation*}
 \begin{split}
 I\geq (id-\Phi_{rT_1})(I)&\geq  (id-\Phi_{rT_1})(I)-\Phi_{rT_2}\circ(id-\Phi_{rT_1})(I)\\
 &=(id-\Phi_{rT_2})\circ(id-\Phi_{rT_1})(I)\geq I-\Phi_{rT_1}(I)-\Phi_{rT_2}(I)\geq 0.
 \end{split}
 \end{equation*}
 This can be used to deduce that
 \begin{equation*}
 \begin{split}
 I\geq  (id-\Phi_{rT_2})\circ(id-\Phi_{rT_1})(I)&\geq  (id-\Phi_{rT_2})\circ(id-\Phi_{rT_1})(I)-\Phi_{rT_3}\circ (id-\Phi_{rT_2})\circ(id-\Phi_{rT_1})(I)\\
 &=(id-\Phi_{rT_3})\circ(id-\Phi_{rT_2})\circ(id-\Phi_{rT_1})(I)\\
 &\geq
 I-\Phi_{rT_1}(I)-\Phi_{rT_2}(I)-\Phi_{rT_3}(I)\geq 0.
 \end{split}
 \end{equation*}
 Iterating  this process, we deduce that
  \begin{equation*}
 \begin{split}
I  &\geq (id-\Phi_{rT_k})\circ \cdots \circ (id-\Phi_{rT_1})(I)\\
&\geq I-\Phi_{rT_1}(I)-\cdots -\Phi_{rT_k}(I)\geq 0
\end{split}
 \end{equation*}
 for any $r\in[0,1)$,
 which proves that $T\in {\bf B}_\Lambda(\cH)$.

 If $T$ is $\Lambda$-commuting and $\sum_{s=1}^{n_i} T_{i,s}T_{i,s}^*=I_\cH$ for any $i\in \{1,\ldots, k\}$, then $(id-\Phi_{T_1})^{p_1}\circ \cdots \circ (id-\Phi_{T_k})^{p_k}(I)= 0$ for any $p_i\in \{0,1\}$. Proposition \ref{regular-equiv} shows that $T\in {\bf B}_\Lambda(\cH)$.
 The proof is complete.
\end{proof}

Let $T=(T_1,\ldots, T_k)$ be  a {\it pure} $k$-tuple in the {\it regular $\Lambda$-polyball} ${\bf B}_\Lambda(\cH)$, i.e $\Phi_{T_i}^m(I)\to 0$ strongly as $m\to \infty$. We define the noncommutative Berezin kernel
$$
K_{T}:\cH\to \ell^2(\FF_{n_1}^+\times\cdots \times \FF_{n_k}^+)\otimes \cH,
$$
 by setting
$$
K_{T}h :=\sum_{\beta_1\in \FF_{n_1}^+,\ldots, \beta_k\in \FF_{n_k}^+}
\chi_{(\beta_1,\ldots, \beta_k)}\otimes   \Delta_{T}(I)^{1/2}T_{k,\beta_k}^*\cdots T_{1,\beta_1}^*h,\qquad h\in \cH,
$$
where $ \Delta_{T}(I):=(id-\Phi_{T_k})\circ\cdots \circ (id-\Phi_{T_1})(I)$.

\begin{theorem}  \label{Berezin} Let $T=(T_1,\ldots, T_k)$ be  a pure $k$-tuple in the regular $\Lambda$-polyball.  Then the following statements hold.
\begin{enumerate}
\item[(i)] The noncommutative Berezin kernel $K_T$ is an isometry.

\item[(ii)] For any $i\in \{1,\ldots, k\}$ and $s\in \{1,\ldots, n_i\}$,
$$
K_T T_{i,s}^*=\left(S_{i,s}^*\otimes I\right) K_T.
$$

\end{enumerate}

\end{theorem}
\begin{proof} First, note that
\begin{equation*}
\begin{split}
\|K_Th \|^2&=\sum_{\beta_1\in \FF_{n_1}^+,\ldots, \beta_k\in \FF_{n_k}^+}
 \left\|   \Delta_{T}(I)^{1/2}T_{k,\beta_k}^*\cdots T_{1,\beta_1}^* h\right\|^2 \\
 &=
 \left<\sum_{\beta_1\in \FF_{n_1}^+,\ldots, \beta_k\in \FF_{n_k}^+}
   T_{1,\beta_1} \cdots T_{k,\beta_k} \Delta_{T}(I)T_{k,\beta_k}^*\cdots T_{1,\beta_1}^* h, h\right>
\end{split}
\end{equation*}
for any $h\in \cH$.
We remark that

\begin{equation*}
\begin{split}
&\sum_{p_k=0}^\infty \Phi_{T_k}^{p_k} [(id-\Phi_{T_k})\circ\cdots \circ (id-\Phi_{T_1})(I)] \\
&\qquad =
\lim_{q_k\to \infty} \sum_{p_k=0}^{q_k}\left\{ \Phi_{T_k}^{p_k}
[(id-\Phi_{T_{k-1}})\circ\cdots \circ (id-\Phi_{T_1})(I)]
-\Phi_{T_k}^{p_k+1}[(id-\Phi_{T_{k-1}})\circ\cdots \circ (id-\Phi_{T_1})(I)]\right\}\\
&\qquad= (id-\Phi_{T_{k-1}})\circ\cdots \circ (id-\Phi_{T_1})(I)
-\lim_{q_k\to \infty} \Phi_{T_k}^{q_k+1}[(id-\Phi_{T_{k-1}})\circ\cdots \circ (id-\Phi_{T_1})(I)]\\
&\qquad =(id-\Phi_{T_{k-1}})\circ\cdots \circ (id-\Phi_{T_1})(I).
\end{split}
\end{equation*}
The latter equality is due to the fact that
$$
-\|\Delta_{(T_{k-1},\ldots, T_1)}(I)\| \Phi_{T_k}^{q_k+1}(I)\leq  \Phi_{T_k}^{q_k+1}(\Delta_{(T_{k-1},\ldots, T_1)}(I))  \leq  \Phi_{T_k}^{q_k+1}(I)    \|\Delta_{(T_{k-1},\ldots, T_1)}(I)\|,
$$
where $\Delta_{(T_{k-1},\ldots, T_1)}(I):=(id-\Phi_{T_{k-1}})\circ\cdots \circ (id-\Phi_{T_1})(I)$, and that $\lim_{q_k\to \infty}\Phi_{T_k}^{q_k+1}(I)=0$.
Continuing this process, we obtain
$$
\sum_{p_1=0}^\infty \Phi_{T_1}^{p_1}\left(\sum_{p_2=0}^\infty \Phi_{T_2}^{p_2}\left(\cdots \sum_{p_k=0}^\infty \Phi_{T_k}^{p_k} [(id-\Phi_{T_k})\circ\cdots \circ (id-\Phi_{T_1})(I)]\cdots \right)\right)=I.
$$
Since   we can rearrange the series  of positive terms,  we obtain
$$
\sum_{p_1,\ldots, p_k=0}^\infty \Phi_{T_1}^{p_1}\circ \cdots \circ \Phi_{T_k}^{p_k}[ \Delta_{T}(I)]=I.
$$
Combining the results above, we obtain that $\|K_Th \|=\|h\|$ for any $h\in \cH$, which proves part (i).

To prove item (ii),
note that, for any $h,h'\in \cH$,
\begin{equation*}
\begin{split}
\left<K_T T_{i,s}^*h,  \chi_{(\alpha_1,\ldots, \alpha_k)}\otimes h'\right>
&=\left< \sum_{\beta_1\in \FF_{n_1}^+,\ldots, \beta_k\in \FF_{n_k}^+}
\chi_{(\beta_1,\ldots, \beta_k)}\otimes   \Delta_{T}(I)^{1/2}T_{k,\beta_k}^*\cdots T_{1,\beta_1}^* T_{i,s}^*h,
 \chi_{(\alpha_1,\ldots, \alpha_k)}\otimes h'
\right>\\
&=\left<\Delta_{T}(I)^{1/2}T_{k,\alpha_k}^*\cdots T_{1,\alpha_1}^*T_{i,s}^* h,h'
\right>\\
&=
\left< h, T_{i,s} T_{1,\alpha_1}\cdots T_{i-1,\alpha_{i-1}}T_{i,\alpha_i}\cdots T_{k,\alpha_k}\Delta_{T}(I)^{1/2}h'\right>\\
&= \overline{\boldsymbol{\lambda}_{i,1}(s,\alpha_1)}\cdots \overline{\boldsymbol{\lambda}_{i,i-1}(s,\alpha_{i-1})}
\left<h,   T_{1,\alpha_1}\cdots T_{i-1,\alpha_{i-1}}T_{i,g_s^i\alpha_i} \cdots T_{k,\alpha_k}\Delta_{T}(I)^{1/2}h'\right>
\end{split}
\end{equation*}
for any $\alpha_1\in \FF_{n_1}^+,\ldots, \alpha_k\in \FF_{n_k}^+$ where, for any   $j\in \{1,\ldots, k\}$,
$$
\boldsymbol{\lambda}_{i,j}(s, \beta)
:= \begin{cases}\prod_{b=1}^q\lambda_{i,j}(s,j_b)&\quad  \text{ if }
 \beta=g_{j_1}^j\cdots g_{j_q}^j\in \FF_{n_j}^+  \\
 1& \quad  \text{ if } \beta=g_0^j.
 \end{cases}
$$
On the other hand, using relation \eqref{shift*} and the definition of $K_T$,  we obtain
\begin{equation*}
\begin{split}
&\left< \left(S_{i,s}^*\otimes I\right) K_T h, \chi_{(\alpha_1,\ldots, \alpha_k)}\otimes h'\right>\\
&\qquad =
\left< S_{i,s}^*(\chi_{(\alpha_1,\ldots, \alpha_{i-1}, g_s^i \alpha_i, \alpha_{i+1},\ldots,  \alpha_k)})
\otimes \Delta_{T}(I)^{1/2} T_{k,\alpha_k}^*\cdots T_{i+1,\alpha_{i+1}} ^*T_{i, g_s^i\alpha_i}^*T_{i-1,\alpha_{i-1}}^*\cdots T_{1,\alpha_1}^*h, \chi_{(\alpha_1,\ldots, \alpha_k)}\otimes h'\right>\\
&\qquad =
\overline{\boldsymbol{\lambda}_{i,1}(s,\alpha_1)}\cdots \overline{\boldsymbol{\lambda}_{i,i-1}(s,\alpha_{i-1})}
\left<h,   T_{1,\alpha_1}\cdots T_{i-1,\alpha_{i-1}}T_{i,g_s^i\alpha_i} \cdots T_{k,\alpha_k}\Delta_{T}(I)^{1/2}h'\right>.
\end{split}
\end{equation*}
Combining the results above, we conclude that  item (ii) holds. The proof is complete.
\end{proof}

\begin{theorem}  \label{model}  Let $T=(T_1,\ldots, T_k)\in B(\cH)^{n_1}\times \cdots \times B(\cH)^{n_k}$. Then $T$  is a pure element  in ${\bf B}_\Lambda(\cH)$ if and only if there is a Hilbert space $\cD$ such that $\cH$ can be identified with  a co-invariant subspace  under the shifts $S_{i,s}\otimes I_\cD$ and
$$
T_{i,s}^*=\left(S_{i,s}^*\otimes I_\cD\right)|_\cH
$$
for any $i\in \{1,\ldots, k\}$ and  $s\in \{1,\ldots, n_i\}$. Moreover,
 $T\in{\bf B}^d_\Lambda(\cH)$ if and only if, in addition,
the following relation holds
$$
P_\cH\ (S_{i,s}^*\otimes I_\cD)|_\cH P_\cH(S_{j,t}\otimes I_\cD)|_\cH=P_\cH\ (S_{i,s}^*\otimes I_\cD)(S_{j,t}\otimes I_\cD)|_\cH
$$
for any $i, j\in  \{1,\ldots, k\}$  with $i\neq j$ and  $s\in \{1,\ldots, n_i\}$, $t\in \{1,\ldots, n_j\}$.
\end{theorem}
\begin{proof}
The direct implication is due to Theorem \ref{Berezin}  and the identification of $\cH$ with $K_T \cH$.
To prove the converse, assume that $
T_{i,s}^*=\left(S_{i,s}^*\otimes I_\cD\right)|_\cH
$
for any $i\in \{1,\ldots, k\}$ and  $s\in \{1,\ldots, n_i\}$. Then $T_{j,t}^*T_{i,s}^*=(S_{j,t}^*S_{i,s}^*\otimes I_\cD)|_\cH$. Consequenly, using the $\Lambda$-commutation relations for the universal model
$S=(S_1,\ldots, S_k)$, we have
\begin{equation*}
\begin{split}
T_{i,s} T_{j,t}&=P_\cH (S_{i,s}S_{j,t}\otimes I_\cD)|_\cH\\
&=\lambda_{i,j}(s,t)P_\cH(S_{j,t}S_{i,s}\otimes I_\cD)|_\cH\\
&=\lambda_{i,j}(s,t) T_{j,t} T_{i,s}.
\end{split}
\end{equation*}
Therefore, $T\in B_\Lambda(\cH)$.  On the other hand,  we have
$$
\sum_{\alpha\in \FF_{n_i}^+, |\alpha|=q} T_{i,\alpha}T_{i,\alpha}^*=P_\cH \left(\sum_{\alpha\in \FF_{n_i}^+, |\alpha|=q} S_{i,\alpha}S_{i,\alpha}^*\right)|_\cH.
$$
Since $[S_{i,1}\cdots S_{i,n_i}]$ is a pure row isometry, we deduce that $T$ is a pure $k$-tuple in
$B_\Lambda(\cH)$.

Now, we prove the second part of the theorem. If $T\in B_\Lambda^d(\cH)$, then we can apply the first part of the theorem and get
$
T_{i,s}^*=\left(S_{i,s}^*\otimes I_\cD\right)|_\cH
$
for any $i\in \{1,\ldots, k\}$ and  $s\in \{1,\ldots, n_i\}$. Using relation \eqref{*-commuting}, we deduce that
\begin{equation*}
\begin{split}
P_\cH\ (S_{i,s}^*\otimes I_\cD)|_\cH P_\cH(S_{j,t}\otimes I_\cD)|_\cH
&=\overline{\lambda_{i,j}(s,t)}P_\cH\ (S_{j,t}^*\otimes I_\cD)|_\cH P_\cH(S_{i,s}\otimes I_\cD)|_\cH\\
&=\overline{\lambda_{i,j}(s,t)}P_\cH\ (S_{j,t}^*S_{i,s}\otimes I_\cD)|_\cH\\
&=P_\cH(S_{i,s}^* S_{j,t}\otimes I_\cD)|_\cH
\end{split}
\end{equation*}
for any $i,j\in \{1,\ldots, k\}$ with $i\neq j$  and  $s\in \{1,\ldots, n_i\}$, $t\in \{1,\ldots, n_j\}$.
Conversely,  assume that
$$
P_\cH\ (S_{i,s}^*\otimes I_\cD)|_\cH P_\cH(S_{j,t}\otimes I_\cD)|_\cH=P_\cH\ (S_{i,s}^*\otimes I_\cD)(S_{j,t}\otimes I_\cD)|_\cH.
$$
Using relation $T_{i,s}^*=\left(S_{i,s}^*\otimes I_\cD\right)|_\cH$, we deduce that
\begin{equation*}
\begin{split}
T_{i,s}^* T_{i,t} &=P_\cH(S_{i,s}^* S_{j,t}\otimes I_\cD)|_\cH\\
&=\overline{\lambda_{i,j}(s,t)}P_\cH\ (S_{j,t}^*S_{i,s}\otimes I_\cD)|_\cH\\
&=\overline{\lambda_{i,j}(s,t)} T_{i,t}T_{i,s}^*,
\end{split}
\end{equation*}
which proves that $T\in B_\Lambda^d(\cH)$.
The proof is complete.
\end{proof}

We remark that, in Theorem \ref{model}, the relation
$$
P_\cH\ (S_{i,s}^*\otimes I_\cD)|_\cH P_\cH(S_{j,t}\otimes I_\cD)|_\cH=P_\cH\ (S_{i,s}^*\otimes I_\cD)(S_{j,t}\otimes I_\cD)|_\cH
$$
for any $i, j\in  \{1,\ldots, k\}$  with $i\neq j$ and  $s\in \{1,\ldots, n_i\}$, $t\in \{1,\ldots, n_j\}$,
is equivalent with the fact that  $\{P_\cH(S_{i,s}\otimes I_\cD)|_\cH\}_{{i\in \{1,\ldots,k\}}\atop{s\in\{1,\ldots, n_i\}}}$ is doubly $\Lambda$-commuting .

Note that due to the doubly  $\Lambda$-commutativity  relations satisfied by the standard shift $S=(S_1,\ldots, S_k)$  and the fact that $S_{i,s}^* S_{i,t}=\delta_{st}I$ for any $i\in \{1,\ldots, k\}$ and $s,t\in \{1,\ldots, n_i\}$, any polynomial in
$\{S_{i,s}\}$ and $ \{S_{i,s}^*\}$ can be represented uniquely as  a finite sum the form
\begin{equation}
\label{polynomial}
p(\{S_{i,s}\}, \{S_{i,s}^*\})= \sum a_{(\alpha_1,\ldots, \alpha_k,\beta_1,\ldots, \beta_k)}S_{1,\alpha_1}\cdots S_{k,\alpha_k}S_{1,\beta_1}^*\cdots S_{k,\beta_k}^*,
\end{equation}
where $a_{(\alpha_1,\ldots, \alpha_k,\beta_1,\ldots, \beta_k)}\in \CC$ , $\alpha_1\in \FF_{n_{1}}^+,\ldots, \alpha_k\in \FF_{n_{k}}^+$ and $\beta_1\in \FF_{n_1}^+,\ldots, \beta_k\in \FF_{n_k}^+$.
To prove the uniqueness of the representation, we need  some notation.
If $\boldsymbol {\alpha}:=(\alpha_1,\ldots, \alpha_k)\in \FF_{n_1}\times \cdots \times \FF_{n_k}$,
we define its length $|\boldsymbol{\alpha}|:=|\alpha_1|+\cdots +|\alpha_k|$, and use the notation
${\bf S}_{\boldsymbol{\alpha}}:=S_{1,\alpha_1}\cdots S_{k,\alpha_k}$

\begin{proposition}\label{uniqueness}
Let $\Gamma$  be a finite set of pairs $(\boldsymbol {\alpha},\boldsymbol {\beta})$ with
$\boldsymbol {\alpha},\boldsymbol {\beta}\in \FF_{n_1}^+\times \cdots \times \FF_{n_k}^+$, and consider the polynomial
$$
p(\{S_{i,s}\}, \{S_{i,s}^*\})=\sum_{(\boldsymbol {\alpha},\boldsymbol {\beta})\in \Gamma}a_{(\boldsymbol {\alpha},\boldsymbol {\beta})}{\bf S}_{\boldsymbol{\alpha}}{\bf S}_{\boldsymbol{\beta}}^*.
$$
If $p(\{S_{i,s}\}, \{S_{i,s}^*\})=0$, then $a_{(\boldsymbol {\alpha},\boldsymbol {\beta})}=0$ for any
$(\boldsymbol {\alpha},\boldsymbol {\beta})\in \Gamma$.
\end{proposition}
\begin{proof}
Assume that there are non-zero coefficients $a_{(\boldsymbol {\alpha},\boldsymbol {\beta})}$ with $(\boldsymbol {\alpha},\boldsymbol {\beta})\in \Gamma  $  and let
$$
p:=\min\left\{|\boldsymbol{\beta}|:\
\text{ there is } \boldsymbol {\alpha}
\text{ with } (\boldsymbol {\alpha},\boldsymbol {\beta})\in \Gamma
\text{ and }  a_{(\boldsymbol {\alpha},\boldsymbol {\beta})}\neq 0\right\}.
$$
Fix  $\boldsymbol{\sigma}$ such that there is $\boldsymbol{\gamma}$ with the property that
$(\boldsymbol {\gamma},\boldsymbol {\sigma})\in \Gamma$,  $a_{(\boldsymbol {\gamma},\boldsymbol {\sigma})}\neq 0$, and
$|\boldsymbol{\sigma}|=p$.
Let $\boldsymbol {\sigma}':=(\sigma_1',\ldots, \sigma_k')\in \FF_{n_1}\times \cdots \times \FF_{n_k}$
be such that $|\boldsymbol{\sigma}|=|\boldsymbol{\sigma}'|=p$.
Note that, if there is  $i\in \{1,\ldots, k\}$ such that $|\sigma_i'|>|\sigma_i|$, then
$S_{\boldsymbol{\sigma}'}^*\chi_{\boldsymbol{\sigma}}=0$. Consequently, if
$S_{\boldsymbol{\sigma}'}^*\chi_{\boldsymbol{\sigma}}\neq 0$, then  $|\sigma_i'|\leq |\sigma_i|$
for any $i\in \{1,\ldots, k\}$.  Since $|\boldsymbol{\sigma}|=|\boldsymbol{\sigma}'|$, we conclude that
if $S_{\boldsymbol{\sigma}'}^*\chi_{\boldsymbol{\sigma}}\neq 0$, then  $|\sigma_i'|= |\sigma_i|$
for any $i\in \{1,\ldots, k\}$.
Taking now  into account that, for each $i\in \{1,\ldots, k\}$, the isometries $S_{i,1},\ldots, S_{i,n_i}$ have orthogonal ranges  and using the definition of the standad shift  $S=(S_1,\ldots, S_k)$, we deduce that
$$
S_{\boldsymbol{\sigma}'}^*\chi_{\boldsymbol{\sigma}}\neq 0\quad \text{ if and only if }\quad
\sigma_i'= \sigma_i \ \text{ for any }\  i\in \{1,\ldots, k\}.
$$
Using this result, we deduce that
\begin{equation*}
\begin{split}
p(\{S_{i,s}\}, \{S_{i,s}^*\})\chi_{\boldsymbol{\sigma}}
&=
    \sum_{\boldsymbol{\alpha}: \  (\boldsymbol{\alpha},\boldsymbol{\sigma})\in \Gamma}
    a_{ (\boldsymbol{\alpha},\boldsymbol{\sigma})}
    S_{\boldsymbol{\alpha} }S_{\boldsymbol{\sigma}}^* \chi_{\boldsymbol{\sigma}}\\
    &=c(\boldsymbol{\sigma})
    \sum_{\boldsymbol{\alpha}: \  (\boldsymbol{\alpha},\boldsymbol{\sigma})\in \Gamma}
    a_{ (\boldsymbol{\alpha},\boldsymbol{\sigma})}
    S_{\boldsymbol{\alpha} }\chi_{(g_0^1,\ldots, g_0^k)}\\
    &=c(\boldsymbol{\sigma})
    \sum_{\boldsymbol{\alpha}: \  (\boldsymbol{\alpha},\boldsymbol{\sigma})\in \Gamma}
    a_{ (\boldsymbol{\alpha},\boldsymbol{\sigma})} d(\boldsymbol{\alpha})
    \chi_{\boldsymbol{\alpha}},
\end{split}
\end{equation*}
where $c(\boldsymbol{\sigma}) $ and $d(\boldsymbol{\alpha})$ are some unimodular constants.
Since $p(\{S_{i,s}\}, \{S_{i,s}^*\})=0$ and $\{\chi_{\boldsymbol{\beta}}\}_{\boldsymbol{\beta}\in\FF_{n_1}^+\times \cdots \times \FF_{n_k}^+}$ is an orthonormal basis of $\ell^2(\FF_{n_1}^+\times \cdots \times \FF_{n_k}^+)$, we conclude that
$a_{ (\boldsymbol{\alpha},\boldsymbol{\sigma})} =0$  for any  $\boldsymbol{\alpha}$ with
 $(\boldsymbol{\alpha},\boldsymbol{\sigma})\in \Gamma$.  In particular, we have
 $a_{ (\boldsymbol{\gamma},\boldsymbol{\sigma})} =0$, which contradicts our assumption.
In conclusion, $a_{(\boldsymbol {\alpha},\boldsymbol {\beta})}=0$   for any  $(\boldsymbol {\alpha},\boldsymbol {\beta})\in \Gamma$.
The proof is complete.
\end{proof}

If $p(\{S_{i,s}\}, \{S_{i,s}^*\}$  is a polynomial of the form \eqref{polynomial} and $T\in{\bf B}_\Lambda(\cH)$, we define
$$
p(\{T_{i,s}\}, \{T_{i,s}^*\}):= \sum a_{(\alpha_1,\ldots, \alpha_k,\beta_1,\ldots, \beta_k)}T_{1,\alpha_1}\cdots T_{k,\alpha_k}T_{1,\beta_1}^*\cdots T_{k,\beta_k}^*
$$
and note that the definition is correct due to the following von Neumann inequality.

\begin{corollary}  \label{vN}
If  $T=(T_1,\ldots, T_k)$ is a  $k$-tuple in the {\it regular $\Lambda$-polyball}, then
$$
\|p(\{T_{i,s}\}, \{T_{i,s}^*\})\|\leq  \|p(\{S_{i,s}\}, \{S_{i,s}^*\})\|
$$
for any polynomial  $p(\{S_{i,s}\}, \{S_{i,s}^*\})$ of the form \eqref{polynomial}.
\end{corollary}
\begin{proof}
Let $r\in [0,1)$ and note that $rT=(rT_1,\ldots, rT_k)$ is a pure $k$-tuple in $B_\Lambda(\cH)$.
Due to Theorem \ref{Berezin}, we deduce that
$$
p(\{rT_{i,s}\}, \{rT_{i,s}^*\})=K_{rT}^* \left(p(\{S_{i,s}\}, \{S_{i,s}^*\})\otimes I\right)K_{rT}
$$
and, consequently,
$$
\|p(\{rT_{i,s}\}, \{rT_{i,s}^*\})\|\leq  \|p(\{S_{i,s}\}, \{S_{i,s}^*\})\|
$$
for any $r\in [0,1)$. Taking $r\to 1$, we complete the proof.
\end{proof}

We introduce the $\Lambda$-polyball algebra $\cA({\bf B}_\Lambda)$ as the normed closed non-self-adjoint   algebra  generated by  the isometries $S_{i,s}$ and the identity. For general results on completely bonded maps and dilations we refer the reader to \cite{P-book}.

\begin{theorem} \label{Berezin-transf}
If $T\in {\bf B}_\Lambda (\cH)$, then the map
$$
\Psi_T(f):=\lim_{r\to 1}K_{rT}^*[f\otimes I] K_{rT}, \qquad f\in C^*(\{S_{i,s}\}),
$$
where the limit is in the operator norm topology, is a is completely contractive linear map. Moreover, its restriction to the  $\Lambda$-polyball algebra $\cA({\bf B}_\Lambda)$ is a completely contractive homomorphism.
If, in addition, $T$ is a pure $k$-tuple, then $\Psi_T(f)=K_{T}^*[f\otimes I] K_{T}$.

\end{theorem}
\begin{proof}
Let $f\in C^*(\{S_{i,s})$ and let $\{p_m(\{S_{i,s}\}, \{S_{i,s}^*\})\}_{m\in \NN}$ be a sequence of polynomials of the form \eqref{polynomial} such that
$p_m(\{S_{i,s}\}, \{S_{i,s}^*\})\to f$ in the operator norm topology, as $m\to \infty$. Due to Corollary
\ref{vN}, the sequence  $\{p_m(\{T_{i,s}\}, \{T_{i,s}^*\})\}_{m\in \NN}$ is Cauchy and, consequently, convergent. Denote $F_T(f):=\lim_{m\to \infty}p_m(\{T_{i,s}\}, \{T_{i,s}^*\})$ and note that $F_T(f)$ is well-defined and
$\|F_T(f)\|\leq \|f\|$. We remark that  there is a matricial version of Corollary \ref{vN}, i.e.
$$
\|[p_{ab}(\{T_{i,s}\}, \{T_{i,s}^*\})]_{q\times q}\|\leq  \|[p_{ab}(\{S_{i,s}\}, \{S_{i,s}^*\})]_{q\times q}\|
$$
for any matrix $[p_{ab}(\{S_{i,s}\}, \{S_{i,s}^*\})]_{q\times q}$ of polynomials of the form \eqref{polynomial}. Using this result, it is easy to see that the map $F_T:C^*(\{S_{i,s}\})\to B(\cH)$ defined as above is a unital completely contractive linear map. Now, we prove that $F_T=\Psi_T$.
First, note that Theorem \ref{Berezin} implies
$$
p_m(\{rT_{i,s}\}, \{rT_{i,s}^*\})=K_{rT}^*[p_m(\{S_{i,s}\}, \{S_{i,s}^*\})]\otimes I]K_{rT}
$$
for any $m\in \NN$ and $r\in [0,1)$. Taking $m\to \infty$, we deduce that
\begin{equation}\label{FT}
F_{rT}(f)=K_{rT}^*(f\otimes I)K_{rT}
\end{equation}
for any $r\in [0,1)$ and  $f\in C^*(\{S_{i,s}\})$. For any $\epsilon>0$, let $n_\epsilon\in \NN$ be such that
\begin{equation}\label{pne}
\|p_{n_\epsilon}(\{S_{i,s}\}, \{S_{i,s}^*\}) -f\|<\epsilon.
\end{equation}
Due to the considerations above we have
\begin{equation}\label{FPF}
\|F_{rT}(f)-p_{n_\epsilon}(\{T_{i,s}\}, \{T_{i,s}^*\})\|\leq\|f-p_{n_\epsilon}(\{S_{i,s}\}, \{S_{i,s}^*\}) \|<\epsilon.
\end{equation}
On  the other hand, we can find $\delta\in (0,1)$ such that
\begin{equation}
\label{pp}
\|p_{n_\epsilon}(\{rT_{i,s}\}, \{rT_{i,s}^*\})-p_{n_\epsilon}(\{T_{i,s}\}, \{T_{i,s}^*\})\|<\epsilon
\end{equation}
for any $r\in (\delta,1)$. Using relation \eqref{FT}, \eqref{pne}, \eqref{FPF}, and \eqref{pp}, we obtain
\begin{equation*}
\begin{split}
\|F_T(f)-K_{rT}^*(f\otimes I)K_{rT}\| &= \|F_T(f)-F_{rT}(f)\|\\
&\leq \|F_T(f)-p_{n_\epsilon}(\{T_{i,s}\}, \{T_{i,s}^*\})\|-\|p_{n_\epsilon}(\{T_{i,s}\}, \{T_{i,s}^*\})-
p_{n_\epsilon}(\{rT_{i,s}\}, \{rT_{i,s}^*\})\|\\
&\qquad\qquad +\|p_{n_\epsilon}(\{rT_{i,s}\}, \{rT_{i,s}^*\})-F_{rT}(f)\|\leq 3\epsilon
\end{split}
\end{equation*}
for any $r\in (\delta, 1)$, Therefore $\lim_{r\to 1}K_{rT}^*[f\otimes I] K_{rT}=F_T(f)$, which proves that $F_T=\Psi_T$. Now, it is easy to see that  $\Psi_T|_{\cA({\bf B}_\Lambda)}$ is a completely contractive homomorphism. If, in addition, $T$ is a pure $k$-tuple, then Theorem \ref{Berezin} implies
$$
K_{T}^*(p_m(\{S_{i,s}\}, \{S_{i,s}^*\})\otimes I)K_T=p_m(\{T_{i,s}\}, \{T_{i,s}^*\}).
$$
Taking $m\to\infty$, we obtain  $K_T^*(f\otimes I)K_T=\Psi_T(f)$ for any $f\in C^*(\{S_{i,s}\})$.
The proof is complete.
\end{proof}

\begin{corollary} \label{cp}
Let $T=(T_1,\ldots, T_k)$ with $T_i=(T_{i,1},\ldots, T_{i,n_i})$ and $T_{i,s}\in B(\cH)$. Then
$T\in {\bf B}_\Lambda(\cH)$ if and only if
there is a completely positive linear map  $\Psi: C^*(\{S_{i,s}\})\to B(\cH)$ such that
$$
\Psi(p(\{S_{i,s}\}, \{S_{i,s}^*\}))=p(\{T_{i,s}\}, \{T_{i,s}^*\})
$$
for any polynomial $p(\{S_{i,s}\}, \{S_{i,s}^*\})$ of the form \eqref{polynomial}.
\end{corollary}

\begin{proof}
The direct implication is a consequence of Theorem \ref{Berezin} and Theorem \ref{Berezin-transf}. To prove the converse, note that
$$
\Delta_{rT}(I):=(id-\Phi_{rT_k})\circ\cdots \circ (id-\Phi_{rT_1})(I)=\Psi_{T}(\Delta_{rS}(I))\geq 0
$$
for any $r\in [0,1)$. On the other hand, we have
\begin{equation*}
 T_{i,s} T_{j,t}-\lambda_{ij}(s,t)T_{j,t}T_{i,s}=\Psi_T\left(S_{i,s} S_{j,t}-\lambda_{ij}(s,t)S_{j,t}S_{i,s}\right)=0
\end{equation*}
for any $i,j\in \{1,\ldots, k\}$ with $i\neq j$ and any $s\in \{1,\ldots, n_i\}$, $t\in \{1,\ldots, n_j\}$.
Therefore, $T\in B_\Lambda(\cH)$. The proof is complete.
\end{proof}

The noncommutative  Berezin transform at  a pure element $T\in {\bf B}_{\Lambda}(\cH)$  is  the mapping
 $$
{\bf B}_T:B(\ell^2(\FF_{n_1}^+\times \cdots \times \FF_{n_k}^+)\to B(\cH),\qquad
{\bf B}_T[f]:=K_T^*[f\otimes I]K_T.
$$
 When $T\in {\bf B}_{\Lambda}(\cH)$, we set
 ${\bf B}_T[f]:= \lim_{r\to 1}K_{rT}^*[f\otimes I] K_{rT}, \qquad f\in C^*(\{S_{i,s}\})$.

In what follows, we prove that the noncommutative Berezin transform is  invariant under a class of automorphims of the $C^*$-algebra $C^*(\{S_{i,s}\})$.

If ${\bf z}=(z_1,\ldots, z_k)$, with $z_i:=(z_{i,1},\ldots, z_{i,n_i})\in \TT^{n_i}$, and
$T=(T_1,\ldots, T_k)\in B_\Lambda(\cH)$, then we define
${\bf z}T:=(z_1T_1,\ldots z_k T_k)$,  where $z_iT_i:=(z_{i,1}T_{i,1},\ldots, z_{i,n_i}T_{i,n_i})$ and set
$\rho_{\bf z}(T):={\bf z}T$. It is easy to  see that the map $\rho_{\bf z}:B_\Lambda(\cH)\to B_\Lambda(\cH)$ is is a well-defined  bijection and $\rho_{\bf z}\circ\rho_{\bar{\bf z}}=id$.
On the other hand,
$\rho_{\bf z}$
generates a  $*$-endomorphism $ \boldsymbol\rho_{\bf z}$  of $C^*(\{S_{i,s}\})$  such that
$$
\boldsymbol\rho_{\bf z}(S_{i,s})=z_{i,s}S_{i,s}.
$$
Set  $\boldsymbol\sigma_{{\bf z}}:=\boldsymbol\rho_{\bar{\bf z}}\boldsymbol\rho_{\bf z}$ and note that $\boldsymbol\sigma_{{\bf z}}(S_{i,s})=S_{i,s}$  for any $i\in \{1,\ldots, k\}$ and $s\in \{1,\ldots, n_i\}$.
Therefore, $ \boldsymbol\rho_{\bf z}$ is an automorphism of  $C^*(\{S_{i,s}\})$.

\begin{proposition}  Let $T\in B_\Lambda(\cH)$ and let ${\bf z}\in \TT^{n_1+\cdots +n_k}$. Then
$$
{\bf B}_T[\boldsymbol\rho_{\bf z}(f)]={\bf B}_{\rho_{\bf z}(T)} [f],\qquad f\in C^*(\{S_{i,s}\}).
$$
\end{proposition}
\begin{proof} Due to Theorem \ref{Berezin-transf}, it is enough to prove the proposition when
$f=S_{1,\alpha_1}\cdots S_{k,\alpha_k}S_{1,\beta_1}^*\cdots S_{k,\beta_k}^*$ and $\alpha_i,\beta_i\in \FF_{n_i}^+$.
Note that
\begin{equation*}
\begin{split}
{\bf B}_{\rho_{\bf z}(T)} [f]&=z_{1,\alpha_1}\cdots z_{k,\alpha_k} \bar z_{1,\beta_1}\cdots \bar z_{k,\beta_k} T_{1,\alpha_1}\cdots T_{k,\alpha_k}T_{1,\beta_1}^*\cdots T_{k,\beta_k}^*\\
&=
{\bf B}_T[z_{1,\alpha_1}\cdots z_{k,\alpha_k} \bar z_{1,\beta_1}\cdots \bar z_{k,\beta_k} S_{1,\alpha_1}\cdots S_{k,\alpha_k}S_{1,\beta_1}^*\cdots S_{k,\beta_k}^*]\\
&=
{\bf B}_T[\boldsymbol\rho_{\bf z}(f)].
\end{split}
\end{equation*}
The proof is complete.
\end{proof}

In what follows, we deduce the following Wold decomposition which can be deduce from Theorem \ref{Wold2}.

\begin{theorem} \label{Wold-partic} Let  $V=(V_1,\ldots, V_k)$ be  a $k$-tuple of doubly $\Lambda$-commuting row isometries acting on a Hilbert space $\cK$. Then there is a unique
othogonal decomposition
$$
\cK=\cK^{(s)}\oplus \cK^{(c)}\oplus \cK^{(r)},
$$
where $\cK^{(s)}, \cK^{(c)},  \cK^{(r)}$ are reducing  subspaces under all isometries $V_{i,s}$,  for $i\in \{1,\ldots, k\}$, $s\in \{1,\ldots, n_i\}$, with the following properties.
\begin{enumerate}
\item[(i)] $V|_{\cK^{(s)}}$ is  a  $k$-tuple of doubly $\Lambda$-commuting pure row isometries, which is isomorphic to the standard $k$-tuple ${\bf S}=({\bf S}_1,\ldots, {\bf S}_k)$ with wandering subspace of dimension  equal to $\dim \Delta_V(I)\cK$.

\item[(ii)]   $V|_{\cK^{(c)}}$ is  a   $k$-tuple of doubly $\Lambda$-commuting  Cuntz row isometries.

\item[(iii)]  $V|_{\cK^{(r)}}$ is  a  $k$-tuple of doubly $\Lambda$-commuting row isometries  having no nontrivial jointly reducing subspace $\cM\subset \cK^{(r)}$ of  all $V_{i,s}$ such that
$V|_\cM$ is a  $k$-tuple of doubly $\Lambda$-commuting  pure  (or  Cuntz) row isometries.
\end{enumerate}
Moreover, we have
$$
 \cK^{(s)} =\bigoplus_{\alpha_{1}\in \FF_{n_{1}}^+,\ldots, \alpha_{k}\in \FF_{n_{k}}^+} V_{1,\alpha_{1}}\cdots V_{k,\alpha_{k}}\left(\Delta_V(I)\cK\right),
 $$
 $$
\cK^{(c)}=\bigcap_{m_{1},\ldots, m_{{k}}=0}^\infty
\left(\bigoplus_{{\alpha_{1}\in \FF^+_{n_{1}}}\atop{ |\alpha_{1}|=m_{1}}}\cdots
\bigoplus_{{\alpha_{{k}}\in \FF^+_{n_{{k}}}}\atop { |\alpha_{{k}}|=m_{{k}}}}
V_{1,\alpha_{1}}\cdots V_{{k}, \alpha_{{k}}}
\cK \right),
$$
and $\cK^{(r)}=(\cK^{(s)} \oplus \cK^{(c)})^\perp$.
\end{theorem}

\begin{proof} According to Theorem \ref{Wold2}, the Hilbert space $\cK$ admits a unique orthogonal decomposition
$$
\cK=\cK_{\{1,\ldots, k\}}\oplus \cK_\emptyset \oplus \cK'
$$
with the following properties.
\begin{enumerate}
\item[(i)]   The subspaces $\cK_{\{1,\ldots, k\}}$,  $ \cK_\emptyset$,   and  $\cK'$ are reducing    for all the isometries
$V_{i,s}$, where $i\in \{1,\ldots, k\}$ and $s\in \{1,\ldots, n_i\}$;

\item[(ii)]    $V_i|_{\cK_{\{1,\ldots, k\}}}:=[V_{i,1}|_{\cK_{\{1,\ldots, k\}}}\cdots V_{i,n_i}|_{\cK_{\{1,\ldots, k\}}}]$ is a pure row isometry for any $i\in \{1,\ldots,k\}$.

\item[(iii)]   $V_i|_{\cK_\emptyset}:=[V_{i,1}|_{\cK_\emptyset}\cdots V_{i,n_i}|_{\cK_\emptyset}]$  is a Cuntz row isometry for any $i\in \{1,\ldots,k\}$.

\end{enumerate}
Using Theorem \ref{Wold3} and Remark \ref{partic}, one can see that
$$
\cK_{\{1,\ldots, k\}}=\cK^{(s)}=\cap_{i=1}^k \cK_i^{(s)}\quad \text{  and }\quad
\cK_\emptyset=\cK^{(c)}=\cap_{i=1}^k \cK_i^{(c)},
$$
where $\cK_i^{(s)}$ and $\cK_i^{(c)}$ are defined by relation \eqref{alt}.
Due to the results of Section 2, $V|_{\cK^{(s)}}$ is   isomorphic to the standard $k$-tuple ${\bf S}=({\bf S}_1,\ldots, {\bf S}_k)$ with wandering subspace of dimension  equal to $\dim \Delta_V(I)\cK$.
To prove part (iii) of the theorem,  we need to use Proposition \ref{charact}. Indeed,
if $\cM\subset \cK^{(r)}$  is a reducing subspace for   all  isometries $V_{i,s}$ such that
$V|\cM$ is a  $k$-tuple of doubly $\Lambda$-commuting  pure  (resp.  Cuntz) row isometries,
then $\cM\subset \cK^{(s)}$ (resp. $\cM\subset \cK^{(c)}$). Since $\cK^{(r)}:=(\cK^{(s)}\oplus \cK^{(c)})^\perp$, we conclude that $\cM=\{0\}$.
The proof is complete.
\end{proof}

\begin{corollary} \label{W-partic} Let  $V=(V_1,\ldots, V_k)$ be  a $k$-tuple of doubly $\Lambda$-commuting row isometries acting on a Hilbert space $\cK$. Then the following  statements hold:
\begin{enumerate}
\item[(i)] $V$ is a pure $k$-tuple if and only if  $\cK=\cK^{(s)}$;
\item[(ii)] $V$ is a  $k$-tuple of Cuntz row isometries if and only if $\cK=\cK^{(c)}$,
\end{enumerate}
where $\cK^{(s)}$ and $\cK^{(c)}$ are defined in Theorem \ref{Wold-partic}.
\end{corollary}

Now, we are ready to prove the following result concerning the $C^*$-algebra $C^*(\{S_{i,s}\})$ and its representations.

\begin{theorem} \label{C*} Let  $V=(V_1,\ldots, V_k)$ be  a $k$-tuple of doubly $\Lambda$-commuting row isometries acting on a Hilbert space $\cK$.  Then the following statements hold.
\begin{enumerate}
\item[(i)] There is a  $*$-representation $\pi:C^*(\{S_{i,s}\})\to B(\cK)$ such that $\pi(S_{i,s})=V_{i,s}$ for any $i\in  \{1,\ldots, k\}$ and  $s\in \{1,\ldots, n_i\}$. Moreover, any
 $*$-representation of $C^*(\{S_{i,s}\})$ is determined by a  a $k$-tuple of doubly $\Lambda$-commuting row isometries.
 \item[(ii)] $C^*(\{S_{i,s}\})$ is the universal  $C^*$-algebra generated by   a $k$-tuple of  doubly
$\Lambda$-commuting   row isometries.

 \item[(iii)] If
$$\bigcap_{{i\in \{1,\ldots,k\}}\atop {s\in \{1,\ldots, n_i\}}} \ker V_{i,s}^*\neq \{0\},
$$
then the $C^*$-algebra $C^*(\{V_{i,s}\})$ is isomorphic to $C^*(\{S_{i,s}\})$.
 \end{enumerate}
\end{theorem}
\begin{proof} Since  $V=(V_1,\ldots, V_k)$ is   a $k$-tuple of doubly $\Lambda$-commuting row isometries,  Corollary  \ref{vN} implies
$$
\|p(\{V_{i,s}\}, \{V_{i,s}^*\})\|\leq  \|p(\{S_{i,s}\}, \{S_{i,s}^*\})\|
$$
for any polynomial  $p(\{S_{i,s}\}, \{S_{i,s}^*\})$ in   $\{S_{i,s}\}$ and  $\{S_{i,s}^*\}$.
Set
 $$\pi(p(\{S_{i,s}\}, \{S_{i,s}^*\})):=p(\{V_{i,s}\}, \{V_{i,s}^*\})
 $$
 and note that $\pi$ is a well-defined $*$-homomorphisms on the algebra of all polynomials in   $\{S_{i,s}\}$ and  $\{S_{i,s}^*\}$. For each $f\in C^*(\{S_{i,s}\})$ let
 $\{p_m(\{S_{i,s}\}, \{S_{i,s}^*\})\}_{m\in \NN}$ be a sequence with the property that $p_m(\{S_{i,s}\}, \{S_{i,s}^*\})\to f$ in norm, as $m\to \infty$.
Since $\{p_m(\{S_{i,s}\}, \{S_{i,s}^*\})\}_{m\in \NN}$ is a Cauchy sequence, the inequality above shows that the sequence $\{p_m(\{V_{i,s}\}, \{V_{i,s}^*\})\}_{m\in \NN}$  is also a Cauchy sequence
and, consequently, $\pi(f):=\lim_{m\to\infty} p_m(\{V_{i,s}\}, \{V_{i,s}^*\})$ exists in $C^*(\{V_{i,s}\})$.
Note that $\pi$ is well-defined linear map and $\|\pi(f)\|\leq \|f\|$. Using standard arguments, one can easily see that $\pi$ is a $*$-representation of the $C^*$ algebra $C^*(\{S_{i,s}\})$.

Conversely, let $\pi:C^*(\{S_{i,s}\})\to B(\cK)$ be a $*$-representation of $C^*(\{S_{i,s}\})$ on the Hilbert space $\cK$ and set
$V_{i,s}:=\pi(S_{i,s})$ for any $i\in  \{1,\ldots, k\}$ and  $s\in \{1,\ldots, n_i\}$. Since
$S=(S_1,\ldots, S_k)$ is   a $k$-tuple of doubly $\Lambda$-commuting row isometries, so is  the $k$-tuple $V=(V_1,\ldots, V_k)$. This completes the proof of part (i).  Note that item (ii) follows from item (i).

To prove part (iii), we use the Wold decomposition of Theorem \ref{Wold-partic}.  As a consequence, we have the decomposition
$$
V_{i,s}=(S_{i,s}\otimes I_\cD)\bigoplus V_{i,s}'
$$
for any $i\in  \{1,\ldots, k\}$ and  $s\in \{1,\ldots, n_i\}$, where
$$ \cD:=\bigcap_{{i\in \{1,\ldots,k\}}\atop {s\in \{1,\ldots, n_i\}}} \ker V_{i,s}^*,$$
 $S=(S_1,\ldots, S_k)$ is the standard shift, and $V'=(V_1',\ldots, V_k')$ is a $k$-tuple of doubly
  $\Lambda$-commuting row isometries. Consequently,  we have
 $$
 p(\{V_{i,s}\}, \{V_{i,s}^*\})=\left(p(\{S_{i,s}\}, \{S_{i,s}^*\})\otimes I_\cD\right)
 \bigoplus p(\{V_{i,s}'\},
  \{V_{i,s}'^*\})
 $$
which, due to the fact that
$$
\|p(\{V_{i,s}'\}, \{V_{i,s}'^*\})\|\leq  \|p(\{S_{i,s}\}, \{S_{i,s}^*\})\|
$$
implies
$$
\| p(\{V_{i,s}\}, \{V_{i,s}^*\})\|=\max\left\{ \| p(\{S_{i,s}\}, \{S_{i,s}^*\})\|, \| p(\{V_{i,s}'\}, \{V_{i,s}'^*\})\|\right\}=\| p(\{S_{i,s}\}, \{S_{i,s}^*\})\|.
$$
We define  $\pi(p(\{S_{i,s}\}, \{S_{i,s}^*\}):=p(\{V_{i,s}\}, \{V_{i,s}^*\})$ and note that $\pi$ can be extended uniquely to an isometric $*$-representation
$\pi:C^*(\{S_{i,s}\})\to C^*(\{V_{i,s}\})$. We remark that   $\pi$ is surjective. Indeed, for each $g\in C^*(\{V_{i,s}\})$, there exists a sequence $q_m(\{V_{i,s}\}, \{V_{i,s}^*\})$ which converges to $g$ in norm. Since $\{q_m(\{V_{i,s}\}, \{V_{i,s}^*\})\}$ is Cauchy, so is  the sequence $\{q_m(\{S_{i,s}\}, \{S_{i,s}^*\})\}$ and, consequently, there exists $f\in C^*(\{S_{i,s}\})$ such that
$f=\lim_{m\to\infty} q_m(\{S_{i,s}\}, \{S_{i,s}^*\})$. Now it clear that $\pi(f)=g$ and
$\|\pi(f)\|=\|f\|$.
The proof is complete.
\end{proof}

\begin{theorem} \label{exact1} Let $S=(S_1,\ldots, S_k)$ be  the standard $k$-tuple of doubly $\Lambda$-commuting row  isometries  $S_i=[S_{i,1}\cdots S_{i,n_i}]$ and let $\cJ_\Lambda$ be the closed
two-sided ideal generated by the projections $I -\sum_{s=1}^{n_1} S_{1,s}S_{1,s}^*$, ..., $I-\sum_{s=1}^{n_k} S_{k,s} S_{k,s}^*$  in the $C^*$-algebra $C^*(\{S_{i,s}\})$. Then the sequence
of $C^*$-algebras
$$
0\to \cJ_\Lambda\to C^*(\{S_{i,s}\})\to \otimes_{i\in \{1,\ldots,k\}}^{\Lambda}\cO_{n_i}\to 0
$$
is exact.
\end{theorem}
\begin{proof}

We consider the collection of all  $k$-tuples $V^\omega=(V^\omega_1,\ldots, V^\omega_k)$
of doubly $\Lambda$-commuting Cuntz row isometries $V^\omega_i=(V^\omega_{i,1}\cdots V^\omega_{i,n_i}]$ such that $C^*(\{V^\omega_{i,s}\})$ is irreducible. We define the $k$-tuple
$\widetilde V=(\widetilde V_1,\ldots, \widetilde V_k)$ with
 $\widetilde V_i=[\widetilde V_{i,1}\cdots \widetilde V_{i,n_i}]$ by setting

 $$
 \widetilde V_{i,s}:=\bigoplus_{\omega}  V^\omega_{i,s},\qquad i\in \{1,\ldots, k\}, s\in \{1,\ldots, n_i\}.
 $$
We prove that $C^*(\{\widetilde V_{i,s}\})$ is the universal $C^*$-algebra generated by a $k$-tuple of doubly $\Lambda$-commuting Cuntz row isometries.
Let $V=(V_1,\ldots, V_k)$  be another $k$-tuple of doubly $\Lambda$-commuting Cuntz row isometries.  It is enough to show that there is a surjective $*$-homomorphism
$\varphi:C^*(\{\widetilde V_{i,s}\})\to C^*(\{ V_{i,s}\})$ such that
$\varphi(\widetilde V_{i,s})= V_{i,s}$ for any $ i\in \{1,\ldots, k\}$ and $ s\in \{1,\ldots, n_i\}$.
For this, it suffices to show that
\begin{equation}
\label{ine-vN}
\|p(\{V_{i,s}\}, \{V_{i,s}^*\})\|\leq  \|p(\{\widetilde V_{i,s}\}, \{\widetilde V_{i,s}^*\})\|
\end{equation}
for any noncommutative polynomial  in $V_{i,s}$ and $V_{i,s}^*$. Due to the GNS construction, given a polynomial $p$, there is an irreducible representation $\rho$ of $C^*(\{ V_{i,s}\})$ such that $\|\rho(p(\{V_{i,s}\}, \{V_{i,s}^*\}))\|=\|p(\{V_{i,s}\}, \{V_{i,s}^*\})|$.
Define the $k$-tuple $V'=(V_1',\ldots, V_k')$, $V_i'=[V'_{i,1}\cdots V'_{i,n_i}]$, by setting
$$
V'_{i,s}:=\rho(V_{i,s}),\qquad  i\in \{1,\ldots, k\}, s\in \{1,\ldots, n_i\}.
 $$
Note that $V'$ is a doubly $\Lambda$-commuting $k$-tuple of Cuntz isometries and
$$
\|p(\{\widetilde V_{i,s}\}, \{\widetilde V_{i,s}^*\})\|\geq \|p(\{ V'_{i,s}\}, \{ V_{i,s}'^{*}\})\|=
\|\rho(p(\{V_{i,s}\}, \{V_{i,s}^*\}))\|=\|p(\{V_{i,s}\}, \{V_{i,s}^*\})\|.
$$
Consequently, setting
$$
\varphi\left(p(\{\widetilde V_{i,s}\}, \{\widetilde V_{i,s}^*\}\right):=
p(\{V_{i,s}\}, \{V_{i,s}^*\}
$$
for any polynomial $p$ as above, we have a well defined contractive $*$-homomorphism  on the $*$-algebra generated by $V_{i,s}$ and $V_{i,s}^*$. It is clear that $\varphi$ extends by continuity to a $*$-homomorphism of  $C^*(\{\widetilde V_{i,s}\})$ onto $ C^*(\{ V_{i,s}\})$.

Now, we prove that  $C^*(\{\widetilde V_{i,s}\})$ is isomorphic to the $C^*$-algebra
$C^*(\{ S_{i,s}\})/_{{\cJ}_\Lambda}$.
According to Theorem \ref{C*},  there is a $*$-representation
$\pi: C^*(\{ S_{i,s}\})\to C^*(\{ \widetilde V_{i,s}\})$ such that
$\pi(S_{i,s})=\widetilde V_{i,s}$  for any $ i\in \{1,\ldots, k\}$ and $ s\in \{1,\ldots, n_i\}$.
Since $\pi(\cJ_\Lambda)=0$, $\pi$ induces  a $*$-representation
$\psi:C^*(\{ S_{i,s}\})/_{{\cJ}_\Lambda} \to C^*(\{\widetilde V_{i,s}\}) $ such that
 $$
\psi\left(p(\{S_{i,s}\}, \{S_{i,s}^*\})+\cJ_\Lambda\right)=p(\{\widetilde V_{i,s}\}, \{\widetilde V_{i,s}^*\}).
$$
Consequently, $\psi$ is surjective and
\begin{equation}
\label{ine-vN2}
\|p(\{\widetilde V_{i,s}\}, \{\widetilde V_{i,s}^*\})\|\leq \|p(\{S_{i,s}\}, \{S_{i,s}^*\})+\cJ_\Lambda\|.
\end{equation}
On the other hand, let $q:C^*(\{ S_{i,s}\})\to C^*(\{ S_{i,s}\})/_{{\cJ}_\Lambda}$ be the canonical quotient map and note that
$\widehat S=(\widehat S_1,\ldots, \widehat S_k)$ with $\widehat S_{i,s}:=[q(S_{i,1}\cdots q(S_{i,n_i})]$ is a $k$-tuple of doubly $\Lambda$-commuting Cuntz row isometries in the $C^*$-algebra  $C^*(\{ S_{i,s}\})/_{{\cJ}_\Lambda}$. Due to the  inequality \eqref{ine-vN}, we have
$$
\|p(\{q(S_{i,s})\}, \{q(S_{i,s})^*\})\|\leq  \|p(\{\widetilde V_{i,s}\}, \{\widetilde V_{i,s}^*\})\|
$$
which together with  inequality \eqref{ine-vN2} implies
$$
\|p(\{\widetilde V_{i,s}\}, \{\widetilde V_{i,s}^*\})\|= \|p(\{S_{i,s}\}, \{S_{i,s}^*\})+\cJ_\Lambda\|.
$$
Consequently, $\psi$ is an isomorphism of $C^*$-algebras  and, therefore, the sequence
$$
0\to \cJ_\Lambda\to C^*(\{S_{i,s}\})\to C^*(\{\widetilde V_{i,s}\}) \to 0
$$
is exact. The proof is complete.
\end{proof}

\bigskip

\section{Invariant subspaces and classification of the pure elements in the regular  $\Lambda$-polyball}

The goal of this section is to classify the Beurling type  jointly invariant subspaces of   the universal $k$-tuple   $S=(S_1,\ldots, S_k)$ and  to classify the pure elements  $T=(T_1,\ldots, T_k)$ in the regular $\Lambda$-polyball, with   defect  of rank $1,2,\ldots, \infty$.

\begin{theorem} \label{co-inv} Let $S=(S_1,\ldots, S_k)$ be the universal model associated with  the $\Lambda$-polyball and
let $\cM\subset \ell^2(\FF_{n_1}^+\times\cdots \times \FF_{n_k}^+)\otimes\cH$ be a jointly co-invariant subspace under the  operator $S_{i,s}\otimes I_\cH$, where $i\in\{1,\ldots, k\}$ and  $s\in \{1,\ldots, n_i\}$. Then
$$
\overline{\text{\rm span}}\{(S_{1,\alpha_1}\dots S_{k,\alpha_k}\otimes I_\cH)\cM:\ \alpha_1\in \FF_{n_1}^+,\ldots, \alpha_{n_k}\in \FF_{n_k}^+\}
=\ell^2(\FF_{n_1}^+\times\cdots \times \FF_{n_k}^+)\otimes \cL,
$$
where $\cL:=(P_\CC\otimes I_\cH)\cM\subset \cH$.
In particular, any jointly reducing subspace $\cM\subset \ell^2(\FF_{n_1}^+\times\cdots \times \FF_{n_k}^+)\otimes\cH$ for  all isometries $S_{i,s}\otimes I_\cH$ has the form
$$\cM=\ell^2(\FF_{n_1}^+\times\cdots \times \FF_{n_k}^+)\otimes \cG
$$
for some Hilbert space $\cG$.
\end{theorem}
 \begin{proof} In what follows, we identify $\CC$ with  the subspace $\CC \chi_{(g_0^1,\ldots, g_0^k)}$ of
 $\ell^2(\FF_{n_1}^+\times\cdots \times \FF_{n_k}^+)$. Fix $x\in \cM$ with $x\neq 0$ and suppose that it has the representation
 \begin{equation}
 \label{x-rep}
 x=\sum_{\alpha_{1}\in \FF_{n_{1}}^+,\ldots, \alpha_{k}\in \FF_{n_{k}}^+}
 \chi_{(\alpha_1,\ldots, \alpha_k)}\otimes h_{(\alpha_1,\ldots, \alpha_k)},\qquad
 h_{(\alpha_1,\ldots, \alpha_k)}\in \cH.
 \end{equation}
 Since $x\neq 0$, there is $h_{(\sigma_1,\ldots, \sigma_k)}\neq 0$. Note that  using the definition of the standard $k$-tuple $S=(S_1,\ldots, S_k)$ and the $\Lambda$-commutation relations satisfied by $S$, we obtain
 $$
 (P_\CC\otimes I_\cH)(S_{1,\sigma_1}^*\cdots S_{k,\sigma_k}^*\otimes I_\cH)x=c_{(\sigma_1,\ldots, \sigma_k)} \chi_{(g_0^1,\ldots, g_0^k)}\otimes h_{(\sigma_1,\ldots, \sigma_k)}
 $$
for some constant $c_{(\sigma_1,\ldots, \sigma_k)}\in \TT$.
Since $\cM$ is a jointly co-invariant subspace under the  operator $S_{i,s}\otimes I_\cH$, where $i\in\{1,\ldots, k\}$ and  $s\in \{1,\ldots, n_i\}$,  we have $(S_{1,\sigma_1}^*\cdots S_{k,\sigma_k}^*\otimes I_\cH)x \in \cM$ and, consequently,  $\chi_{(g_0^1,\ldots, g_0^k)}\otimes h_{(\sigma_1,\ldots, \sigma_k)}\in \cL:=(P_\CC\otimes I_\cH)\cM$.
Since
$$
(S_{1,\sigma_1}\cdots S_{k,\sigma_k}\otimes I_\cH)\left(\chi_{(g_0^1,\ldots, g_0^k)}\otimes h_{(\sigma_1,\ldots, \sigma_k)}\right)=d_{(\sigma_1,\ldots, \sigma_k)} \chi_{(\sigma_1,\ldots, \sigma_k)}\otimes h_{(\sigma_1,\ldots, \sigma_k)}
$$
for some constant $d_{(\sigma_1,\ldots, \sigma_k)}\in \TT$, we deduce that
$\chi_{(\sigma_1,\ldots, \sigma_k)}\otimes h_{(\sigma_1,\ldots, \sigma_k)}\in \ell^2(\FF_{n_1}^+\times\cdots \times \FF_{n_k}^+)\otimes \cL$.
Using now relation \eqref{x-rep}, we  conclude that $x\in \ell^2(\FF_{n_1}^+\times\cdots \times \FF_{n_k}^+)\otimes \cL$, which shows that $\cM\subset \ell^2(\FF_{n_1}^+\times\cdots \times \FF_{n_k}^+)\otimes \cL$ and, therefore, the subspace
\begin{equation}
\label{K-def}
\cK:=\overline{\text{\rm span}}\{(S_{1,\alpha_1}\dots S_{k,\alpha_k}\otimes I_\cH)\cM:\ \alpha_1\in \FF_{n_1}^+,\ldots, \alpha_{n_k}\in \FF_{n_k}^+\}
\end{equation}
is included in $\ell^2(\FF_{n_1}^+\times\cdots \times \FF_{n_k}^+)\otimes \cL$.

To prove the reverse inclusion, we show first that
$\chi_{(g_0^1,\ldots, g_0^k)}\otimes \cL\subset \cK$.
To this end, let $y\in \cL$ with $y\neq 0$. Then there is $z\in \cM\subset \ell^2(\FF_{n_1}^+\times\cdots \times \FF_{n_k}^+)\otimes \cL$ of the form
\begin{equation*}
z=\chi_{(g_0^1,\ldots, g_0^k)}\otimes y+\sum_{{\alpha_{1}\in \FF_{n_{1}}^+,\ldots, \alpha_{k}\in \FF_{n_{k}}^+}\atop{|\alpha_1|+\cdots +|\alpha_k|\geq 1}}
 \chi_{(\alpha_1,\ldots, \alpha_k)}\otimes h_{(\alpha_1,\ldots, \alpha_k)},\qquad
 h_{(\alpha_1,\ldots, \alpha_k)}\in \cL,
\end{equation*}
such that $\chi_{(g_0^1,\ldots, g_0^k)}\otimes y=(P_\CC\otimes I_\cH)z$. On the other hand,  due to the proof of Theorem \ref{standard}, we have
$$
(id-\Phi_{S_1})\circ\cdots \circ (id-\Phi_{S_k})(I)=P_\CC,
$$
which implies
$$
\chi_{(g_0^1,\ldots, g_0^k)}\otimes y=(P_\CC\otimes I_\cH)z=(id-\Phi_{S_1\otimes I_\cH})\circ\cdots \circ (id-\Phi_{S_k\otimes I_\cH})(I)z.
$$
Since $\cM$ is co-invariant under all operators $S_{i,s}\otimes I_\cH$ and using relation \eqref{K-def},  we deduce that
$\chi_{(g_0^1,\ldots, g_0^k)}\otimes y\in \cK$ for any $y\in \cL$, which implies
$
(S_{1,\alpha_1}\cdots S_{k,\alpha_k}\otimes I_\cH)\left(\CC\chi_{(g_0^1,\ldots, g_0^k)}\otimes \cL\right)\subset \cK$ for any $\alpha_{i}\in \FF_{n_{i}}^+$.
Using the definition of the standard $k$-tuple $S=(S_1,\ldots, S_k)$, one can see that
$d_{(\alpha_1,\ldots, \alpha_k)} \chi_{(\alpha_1,\ldots, \alpha_k)}\otimes \cL\subset \cK$,
for some constant $d_{(\alpha_1,\ldots, \alpha_k)}\in \TT$, which  shows that
$ \chi_{(\alpha_1,\ldots, \alpha_k)}\otimes \cL\subset \cK$. Now it is clear that
$\ell^2(\FF_{n_1}^+\times\cdots \times \FF_{n_k}^+)\otimes \cL\subset \cK$, which proves the first part of the theorem. Note that part two of the theorem follows immediately from part one.
The proof is complete.
\end{proof}

An operator $A:\ell^2(\FF_{n_1}^+\times\cdots \times \FF_{n_k}^+)\otimes \cH\to \ell^2(\FF_{n_1}^+\times\cdots \times \FF_{n_k}^+)\otimes \cK$ is called {\it multi-analytic}
 with respect to the universal model $S=(S_1,\ldots, S_k)$
if
$$
A(S_{i,s}\otimes I_\cH)=(S_{i,s}\otimes I_\cK)A
$$
for any $i\in\{1,\ldots, k\}$ and  $s\in \{1,\ldots, n_i\}$. If, in addition,  $A$ is a partial isometry, we call it {\it inner multi-analytic} operator.

\begin{theorem} \label{Beurling-fac}
Let $S=(S_1,\ldots, S_k)$ be the universal model associated with  the $\Lambda$-polyball and let $Y$ be a selfadjoint operator on the Hilbert space $\ell^2(\FF_{n_1}^+\times\cdots \times \FF_{n_k}^+)\otimes \cK$. Then the following statements are equivalent.
\begin{enumerate}
\item[(i)] There is a multi-analytic operator $A: \ell^2(\FF_{n_1}^+\times\cdots \times \FF_{n_k}^+)\otimes \cL\to
\ell^2(\FF_{n_1}^+\times\cdots \times \FF_{n_k}^+)\otimes \cK$ such that
$$Y=AA^*.$$

\item[(ii)] $(id-\Phi_{S_1\otimes I_\cK})\circ\cdots \circ (id-\Phi_{S_k\otimes I_\cK})(Y)\geq 0$.
\end{enumerate}
\end{theorem}
\begin{proof}
The implication (i)$\implies$(ii) is due to the fact that
$$
\Delta_{S\otimes I_\cL}(I):=(id-\Phi_{S_1\otimes I_\cL})\circ\cdots \circ (id-\Phi_{S_k\otimes I_\cL})=P_\CC\otimes I_\cL\geq 0
$$
and $\Delta_{S\otimes I_\cK}(AA^*)=A \Delta_{S\otimes I_\cL} (I)A^*\geq 0$.
To prove the implication (ii)$\implies$(i), assume that $\Delta_{S\otimes I_\cK}(Y)\geq 0$ and denote
$$
\Delta_{S_2\otimes I_\cK,\ldots, S_k\otimes I_\cK}:=(id-\Phi_{S_2\otimes I_\cK})\circ\cdots \circ (id-\Phi_{S_k\otimes I_\cK}).
$$
Note that in this case we have
$\Phi_{S_1\otimes I_\cK}(\Delta_{S_2\otimes I_\cK,\ldots, S_k\otimes I_\cK}(Y))\leq \Delta_{S_2\otimes I_\cK,\ldots, S_k\otimes I_\cK}(Y)
$
which implies
$$
\Phi^m_{S_1\otimes I_\cK}(\Delta_{S_2\otimes I_\cK,\ldots, S_k\otimes I_\cK}(Y))\leq \Delta_{S_2\otimes I_\cK,\ldots, S_k\otimes I_\cK}(Y), \qquad m\in \NN.
$$
Since
$$
-\| \Delta_{S_2\otimes I_\cK,\ldots, S_k\otimes I_\cK}(I)\|\Phi_{S_1\otimes I_\cK}^m(I)\leq
\Phi^m_{S_1\otimes I_\cK}(\Delta_{S_2\otimes I_\cK,\ldots, S_k\otimes I_\cK}(Y))\leq
\| \Delta_{S_2\otimes I_\cK,\ldots, S_k\otimes I_\cK}(I)\|\Phi_{S_1\otimes I_\cK}^m(I)
$$
and  $\Phi_{S_1\otimes I_\cK}^m(I)\to 0$, as $m\to  \infty$, we deduce that
$\Delta_{S_2\otimes I_\cK,\ldots, S_k\otimes I_\cK}(Y)\geq 0$, which implies
$$
\Phi_{S_2\otimes I_\cK}(\Delta_{S_3\otimes I_\cK,\ldots, S_k\otimes I_\cK}(Y))\leq \Delta_{S_3\otimes I_\cK,\ldots, S_k\otimes I_\cK}(Y).
$$
As above, the latter inequality implies
$\Delta_{S_3\otimes I_\cK,\ldots S_k\otimes I_\cK}(Y)\geq 0$. Continuing this process, we conclude that $\Phi_{S_k\otimes I_\cK}(Y)\leq Y$ and $Y\geq 0$. Due to Lemma \ref{commutative}, one can show,  in a similar manner  as above, that $\Phi_{S_i\otimes I_\cK}(Y)\leq Y$ for any $i\in \{1,\ldots, k\}$.

Consider the subspace $\cG:=\overline{Y^{1/2}\left(\ell^2(\FF_{n_1}^+\times\cdots \times \FF_{n_k}^+)\otimes \cK\right)}$ and set
\begin{equation}
\label{def-C}
C_{i,s}(Y^{1/2} y):=Y^{1/2}(S_{i,s}^*\otimes I_\cK)y
\end{equation}
for any $y\in \ell^2(\FF_{n_1}^+\times\cdots \times \FF_{n_k}^+)\otimes \cK$, $i\in \{1,\ldots, k\}$ and $s\in \{1,\ldots, n_i\}$.
To see that $C_{i,s}$ is well-defined on the range of $Y^{1/2}$, note that, for each $i\in \{1,\ldots, k\}$,
\begin{equation*}
\begin{split}
\sum_{s=1}^{n_i} \|C_{i,s}Y^{1/2} y\|^2&=\left< \sum_{s=1}^{n_i} (S_{i,s}\otimes I_\cK)Y(S_{i,s}\otimes I_\cK)y,y\right>\\
&=\left<\Phi_{S_i\otimes I_\cK}(Y)y,y\right>\leq \|Y^{1/2}y\|^2
\end{split}
\end{equation*}
for any  $y\in \ell^2(\FF_{n_1}^+\times\cdots \times \FF_{n_k}^+)\otimes \cK$. Consequently, we can extend by continuity each operator $C_{i,s}$ to the space $\cG$. Setting $T_{i,s}:=C_{i,s}^*$, one can easily see that the inequality above  also implies  $\sum_{s=1}^{n_i} T_{i,s}T_{i,s}^*\leq I_\cG$. On the other hand, relation \eqref{def-C} implies
$$
Y^{1/2}\left[(id-\Phi_{T_1})\circ\cdots \circ (id-\Phi_{T_k})(I_\cG)\right]Y^{1/2} =(id-\Phi_{S_1\otimes I_\cK})\circ\cdots \circ (id-\Phi_{S_k\otimes I_\cK})(Y)\geq 0.
$$
Now, let us show that for each $i\in \{1,\ldots, k\}$, $T_i:=[T_{i,1}\cdots T_{i,n_i}]$ is a pure row contraction.
Indeed, note that
\begin{equation*}
\begin{split}
\left<\Phi_{T_i}^m(I_\cG)Y^{1/2}y, Y^{1/2}y\right>&=\left< \Phi_{S_i\otimes I_\cK}^m(Y)y,y\right>\\
&\leq \|Y\|\left< \Phi_{S_i\otimes I_\cK}^m(I)y,y\right>
\end{split}
\end{equation*}
for any $y\in  \ell^2(\FF_{n_1}^+\times\cdots \times \FF_{n_k}^+)\otimes \cK$. Since $S_i\otimes I_\cK$ is a pure row isometry, we have $\Phi_{S_i\otimes I_\cK}^m(I)y\to 0$ as $m\to\infty$. Consequently, $\Phi_{T_i}^m(I_\cG)\to 0$ as $m\to 0$, which proves that $T_i$ is a pure row contraction.

Now, we show that $T=(T_1,\ldots, T_k)$ is in the regular $\Lambda$-polyball. Indeed, due to relation \eqref{def-C} and using the fact that the standard $k$-tuple $S=(S_1,\ldots, S_k)$ is in the regular $\Lambda$-polyball, we deduce that
\begin{equation*}
\begin{split}
Y^{1/2} T_{i,s}T_{j,t}&=(S_{i,s}\otimes I_\cK)Y^{1/2}T_{j,t}= (S_{i,s}\otimes I_\cK)(S_{j,t}\otimes I_\cK)Y^{1/2}\\
&= \lambda_{i,j}(s,t)(S_{j,t}\otimes I_\cK)(S_{i,s}\otimes I_\cK)Y^{1/2}\\
&=\lambda_{i,j}(s,t) Y^{1/2} T_{j,t} T_{i,s}.
\end{split}
\end{equation*}
Hence, we have $Y^{1/2}\left(T_{i,s}T_{j,t}-\lambda_{i,j}(s,t) T_{j,t} T_{i,s}\right)=0$. Since $Y^{1/2}$ is injective on the Hilbert space $\cG$, we conclude that
$$T_{i,s}T_{j,t}=\lambda_{i,j}(s,t) T_{j,t} T_{i,s}
$$
for any $ i,j\in \{1,\ldots, k\}$ with $i\neq j$ and any $s\in \{1,\ldots, n_i\}$, $t\in \{1,\ldots, n_j\}$. This proves that  the $k$-tuple $T=(T_1,\ldots, T_k)$ is $\Lambda$-commuting.
Due to Theorem \ref{Berezin}, the associated noncommutative Berezin kernel $K_T:\cG\to \ell^2(\FF_{n_1}^+\times\cdots \times \FF_{n_k}^+)\otimes \cG$
is an isometry and
$$
K_T T_{i,s}^*=\left(S_{i,s}^*\otimes I_\cG\right) K_T.
$$
for any $i \in \{1,\ldots, k\}$ and  $s\in \{1,\ldots, n_i\}$.
Consequently, using relation \eqref{def-C}, one can see that
the operator
$$A:=Y^{1/2}K_T^*:\ell^2(\FF_{n_1}^+\times\cdots \times \FF_{n_k}^+)\otimes \cG\to\ell^2(\FF_{n_1}^+\times\cdots \times \FF_{n_k}^+)\otimes \cK
$$
satisfies the relation
\begin{equation*}
\begin{split}
A(S_{i,s}\otimes \cG)&=Y^{1/2} K_T^*(S_{i,s}\otimes I_\cG)=Y^{1/2} T_{i,s} K_T^*\\
&(S_{i,s}\otimes I_\cK)Y^{1/2} K_T^*=(S_{i,s}\otimes I_\cK)A
\end{split}
\end{equation*}
for any $i \in \{1,\ldots, k\}$ and  $s\in \{1,\ldots, n_i\}$. Hence, $A$ is a multi-analytic operator. The proof is complete.
\end{proof}

We say that $\cM\subset \ell^2(\FF_{n_1}^+\times\cdots \times \FF_{n_k}^+)\otimes \cK$   is a {\it Beurling type jointly invariant subspace} under the operators $S_{i,s}\otimes I_\cK$, where $i\in\{1,\ldots, k\}$ and  $s\in \{1,\ldots, n_i\}$,  if there is an inner multi-analytic operator
$\Psi: \ell^2(\FF_{n_1}^+\times\cdots \times \FF_{n_k}^+)\otimes \cL\to
\ell^2(\FF_{n_1}^+\times\cdots \times \FF_{n_k}^+)\otimes \cK$ such that
$$
\cM=\Psi \left(\ell^2(\FF_{n_1}^+\times\cdots \times \FF_{n_k}^+)\otimes \cL\right).
$$
In what follows, we use the notation $\left((S_1\otimes I_\cK)|_\cM,\ldots, (S_k\otimes I_\cK) |_\cM\right)$, where $$(S_i\otimes I_\cK)|_\cM:=(S_{i,1}\otimes I_\cK)|_\cM,\ldots, (S_{i,n_i}\otimes I_\cK)|_\cM).
$$

Now, we can prove the following characterization of the Beurling type jointly invariant subspaces under the universal model of the regular $\Lambda$-polyball.
\begin{theorem}  Let $\cM\subset \ell^2(\FF_{n_1}^+\times\cdots \times \FF_{n_k}^+)\otimes \cK$   be a jointly invariant subspace under $S_{i,s}\otimes I_\cK$, where $i\in\{1,\ldots, k\}$ and  $s\in \{1,\ldots, n_i\}$. Then the following statements are equivalent.
\begin{enumerate}
\item[(i)] $\cM$ is a Beurling type  jointly invariant subspace.

\item[(ii)]  $(id-\Phi_{S_1\otimes I_\cK})\circ\cdots \circ (id-\Phi_{S_k\otimes I_\cK})(P_\cM)\geq 0$.

\item[(iii)] The  $k$-tuple $\left((S_1\otimes I_\cK)|_\cM,\ldots, (S_k\otimes I_\cK) |_\cM\right)$ is doubly $\Lambda$-commuting.
\item[(iv)] There is an isometric multi-analytic operator
$\Psi: \ell^2(\FF_{n_1}^+\times\cdots \times \FF_{n_k}^+)\otimes \cL\to
\ell^2(\FF_{n_1}^+\times\cdots \times \FF_{n_k}^+)\otimes \cK$ such that
$$
\cM=\Psi \left(\ell^2(\FF_{n_1}^+\times\cdots \times \FF_{n_k}^+)\otimes \cL\right).
$$
\end{enumerate}
\end{theorem}

\begin{proof} The equivalence of (i) with (ii) is due to Theorem \ref{Beurling-fac}. Indeed, if $\Psi: \ell^2(\FF_{n_1}^+\times\cdots \times \FF_{n_k}^+)\otimes \cL\to
\ell^2(\FF_{n_1}^+\times\cdots \times \FF_{n_k}^+)\otimes \cK$ is an  inner multi-analytic operator and $\cM=\text{\rm range}\,\Psi$, then $P_\cM=\Psi\Psi^*$. Since
$
\Psi(S_{i,s}\otimes I_\cH)=(S_{i,s}\otimes I_\cK)\Psi
$
for any $i\in\{1,\ldots, k\}$ and  $s\in \{1,\ldots, n_i\}$, we have
\begin{equation*}
\begin{split}
(id-\Phi_{S_1\otimes I_\cK})\circ\cdots \circ (id-\Phi_{S_k\otimes I_\cK})(P_\cM)&=\Psi^*(id-\Phi_{S_1\otimes I_\cL})\circ\cdots \circ (id-\Phi_{S_k\otimes I_\cL})(I)\Psi^*\\
&=\Psi(P_\CC\otimes I_\cL)\Psi^*\geq 0,
\end{split}
\end{equation*}
and, consequently,
the direct implications follows. Conversely, applying    Theorem \ref{Beurling-fac} to $Y=P_\cM$, we find a  multi-analytic operator $A:\ell^2(\FF_{n_1}^+\times\cdots \times \FF_{n_k}^+)\otimes \cL\to
\ell^2(\FF_{n_1}^+\times\cdots \times \FF_{n_k}^+)\otimes \cK$ such that $P_\cM=AA^*$. Since $P_\cM$ is an orthogonal projection, $A$ must be a partial isometry.

Now, we prove the implication (i)$\implies$(iv). If (i) holds,
then  there is an inner multi-analytic operator
$\Psi: \ell^2(\FF_{n_1}^+\times\cdots \times \FF_{n_k}^+)\otimes \cH\to
\ell^2(\FF_{n_1}^+\times\cdots \times \FF_{n_k}^+)\otimes \cK$ such that $P_\cM=\Psi\Psi^*$.
Note that
$$
\text{\rm range}\,\Psi^*=\left\{x\in \ell^2(\FF_{n_1}^+\times\cdots \times \FF_{n_k}^+)\otimes \cH: \ \|\psi(x)\|=\|x\|\right\}
$$
is the initial space of $\Psi$ and, due to the fact that $
\Psi(S_{i,s}\otimes I_\cH)=(S_{i,s}\otimes I_\cK)\Psi
$,  it is invariant under all the isometries $S_{i,s}\otimes I_\cH$. Since  $\left(\text{\rm range}\,\Psi^*\right)^\perp=\ker \Psi$ and
$
\Psi(S_{i,s}\otimes I_\cH)=(S_{i,s}\otimes I_\cK)\Psi,
$
it is clear that  $\left(\text{\rm range}\,\Psi^*\right)^\perp$ is invariant under all isometries $S_{i,s}\otimes I_\cH$ and, therefore, it is jointly reducing for these operators. On the other hand, the support of $\Psi$ is the smallest reducing subspace $\text{\rm supp}(\Psi)\subset \ell^2(\FF_{n_1}^+\times\cdots \times \FF_{n_k}^+)\otimes \cH$ under the operators $S_{i,s}\otimes I_\cH$ containing the co-invariant subspace $\text{\rm range}\,\Psi^*$. Consequently, we have $\text{\rm supp}(\Psi)=\text{\rm range}\,\Psi^*$.
Note that $\Phi:=\Psi|_{\text{\rm supp}(\Psi)}$ is an isometric multi-analytic operator .
Using  Theorem \ref{co-inv}, we deduce that
$\text{\rm supp}(\Psi)=\ell^2(\FF_{n_1}^+\times\cdots \times \FF_{n_k}^+)\otimes \cL$, where $\cL:=(P_\CC\otimes I_\cH)\text{\rm range}\,\Psi^*$ and, using relation $P_\cM=\Psi\Psi^*$, we also have $\cM=\Phi(\ell^2(\FF_{n_1}^+\times\cdots \times \FF_{n_k}^+)\otimes \cL)$.

In what follows, we prove that (iv)$\implies$(iii).  Assume that item (iv) holds. Then $\Psi\Psi^*=P_\cM$ and $\Psi$ is an isometric multi-analytic operator.
It is easy to see that the  $k$-tuple $\left((S_1\otimes I_\cK)|_\cM,\ldots, (S_k\otimes I_\cK) |_\cM\right)$ is doubly $\Lambda$-commuting if and only if $k$-tuple $\left(P_\cM(S_1\otimes I_\cK)|_\cM,\ldots, P_\cM(S_k\otimes I_\cK) |_\cM\right)$ is doubly $\Lambda$-commuting, where
$P_\cM(S_i\otimes I_\cK)|_\cM:=(P_\cM(S_{i,1}\otimes I_\cK)|_\cM,\ldots, P_\cM(S_{i,n_i}\otimes I_\cK)|_\cM)$. In what follows, we prove that the latter statement holds. Using the fact that the standard $k$-tuple $S=(S_1,\ldots, S_k)$ is doubly $\Lambda$-commuting, we obtain
\begin{equation*}
\begin{split}
P_\cM(S_{i,s}\otimes I_\cK)P_\cM (S_{j,t}^*\otimes I_\cK)P_\cM
&=\Psi\Psi^*(S_{i,s}\otimes I_\cK)\Psi\Psi^* (S_{j,t}^*\otimes I_\cK)\Psi\Psi^*\\
&= \Psi\Psi^*\Psi(S_{i,s}\otimes I_\cK) (S_{j,t}^*\otimes I_\cK)\Psi^*\Psi\Psi^*\\
&= \Psi (S_{i,s}\otimes I_\cK) (S_{j,t}^*\otimes I_\cK) \Psi^*\\
&=\Psi \lambda_{j,i}(S_{j,t}^*\otimes I_\cK)(S_{i,s}\otimes I_\cK)\Psi^*
\end{split}
\end{equation*}
for any $i,j\in \{1,\ldots, k\}$ with $i\neq j$, and $s\in \{1,\ldots,n_i\}$, $t\in \{1,\ldots, n_j\}$.
Similarly, one can prove that
\begin{equation*}
\begin{split}
P_\cM(S_{j,t}\otimes I_\cK)P_\cM (S_{i,s}^*\otimes I_\cK)P_\cM
= \Psi (S_{j,t}\otimes I_\cK) (S_{i,s}^*\otimes I_\cK) \Psi^*.
\end{split}
\end{equation*}
Combining these relations, we deduce that
$$
P_\cM(S_{j,t}\otimes I_\cK)P_\cM (S_{i,s}^*\otimes I_\cK)P_\cM=
\overline{\lambda_{j,i}(t,s)} P_\cM(S_{i,s}\otimes I_\cK)P_\cM (S_{j,t}^*\otimes I_\cK)P_\cM,
$$
which proves that  $\left((S_1\otimes I_\cK)|_\cM,\ldots, (S_k\otimes I_\cK) |_\cM\right)$ is doubly $\Lambda$-commuting. Thus item (iii) holds.

It remains to prove that (iii)$\implies$(ii). To this end, assume that item (iii) holds. Using the fact that  $k$-tuple $\left(P_\cM(S_1\otimes I_\cK)|_\cM,\ldots, P_\cM(S_k\otimes I_\cK) |_\cM\right)$ is doubly $\Lambda$-commuting and the subspace  $\cM^\perp$ is invariant under all operators $S_{i,s}^*\otimes I_\cK$, we obtain

\begin{equation*}
\begin{split}
&(S_{1,\alpha_1}\otimes I_\cK)\cdots (S_{k,\alpha_k}\otimes I_\cK)
P_\cM (S_{k,\alpha_k}^*\otimes I_\cK)\cdots (S_{1,\alpha_1}^*\otimes I_\cK)\\
&\qquad\qquad=
\left[P_\cM(S_{1,\alpha_1}\otimes I_\cK)P_\cM\cdots P_\cM(S_{k,\alpha_k}\otimes I_\cK)
P_\cM\right] \left[P_\cM(S_{k,\alpha_k}^*\otimes I_\cK)P_\cM\cdots P_\cM(S_{1,\alpha_1}^*\otimes I_\cK)P_\cM\right]\\
&\qquad\qquad=
[P_\cM(S_{1,\alpha_1}\otimes I_\cK)P_\cM][P_\cM(S_{1,\alpha_1}^*\otimes I_\cK)P_\cM]\cdots
[P_\cM(S_{k,\alpha_k}\otimes I_\cK)
P_\cM][P_\cM(S_{k,\alpha_k}^*\otimes I_\cK)P_\cM]\\
&\qquad\qquad=
[(S_{1,\alpha_1}\otimes I_\cK)P_\cM (S_{1,\alpha_1}^*\otimes I_\cK)]
\cdots [(S_{k,\alpha_k}\otimes I_\cK)P_\cM (S_{k,\alpha_k}^*\otimes I_\cK)]
\end{split}
\end{equation*}
for any $\alpha_1\in \FF_{n_1}^+,\ldots, \alpha_k\in \FF_{n_k}^+$.
Using these relations, one can see that
\begin{equation}
\label{PM}
(id-\Phi_{S_1\otimes I_\cK})\circ\cdots \circ (id-\Phi_{S_k\otimes I_\cK})(P_\cM)
=(P_\cM-\Phi_{S_1\otimes I_\cK}(P_\cM)\cdots (P_\cM-\Phi_{S_k\otimes I_\cK}(P_\cM).
\end{equation}
On the other end, since $\cM$ is an invariant subspace under all isometries $S_{i,s}$, we have
$$P_\cM-\Phi_{S_i\otimes I_\cK}(P_\cM)\geq 0 \quad \text{ for any }\quad  i\in \{1,\ldots, k\}.
$$
Indeed, for any $x\in \cM$ and $y\in \cM^\perp$, we have
\begin{equation*}
\begin{split}
\left<\sum_{s=1}^{n_i} (S_{i,s}\otimes I_\cK)P_\cM (S_{i,s}^*\otimes I_\cK)(x+y), (x+y)\right>
&=\left<\sum_{s=1}^{n_i} (S_{i,s}\otimes I_\cK)P_\cM (S_{i,s}^*\otimes I_\cK)x, x\right>\\
&=\sum_{s=1}^{n_i} \|P_\cM(S_{i,s}\otimes I_\cK)x\|^2\leq \sum_{s=1}^{n_i} \|(S_{i,s}\otimes I_\cK)x\|^2\\
&\leq \|x\|^2=\|P_\cM x\|^2=\|P_\cM(x+y)\|^2.
\end{split}
\end{equation*}
Now, note that $P_\cM-\Phi_{S_i\otimes I_\cK}(P_\cM)$ commutes with $P_\cM-\Phi_{S_j\otimes I_\cK}(P_\cM)$  for any $i,j\in \{1,\ldots, k\}$. Indeed, according to our calculations preceding relation  \eqref{PM} and using the $\Lambda$-commutativity of the standard  $k$-tuple $S=(S_1,\ldots, S_k)$, we  have
\begin{equation*}
\begin{split}
&[(S_{i,s}\otimes I_\cK)P_\cM (S_{i,s}^*\otimes I_\cK)] [(S_{j,t}\otimes I_\cK)P_\cM (S_{j,t}^*\otimes I_\cK)]\\
&\qquad=(S_{i,s}S_{j,t}\otimes I_\cK)P_\cM
(S_{j,t}^*S_{i,s}^*\otimes I_\cK)\\
&\qquad=(S_{j,t}S_{i,s}\otimes I_\cK)P_\cM
(S_{i,s}^*S_{j,t}^*\otimes I_\cK)\\
&\qquad=[(S_{j,t}\otimes I_\cK)P_\cM (S_{j,t}^*\otimes I_\cK)][(S_{i,s}\otimes I_\cK)P_\cM (S_{i,s}^*\otimes I_\cK)]
\end{split}
\end{equation*}
for any $i,j\in \{1,\ldots, k\}$ with $i\neq j$, and $s\in \{1,\ldots, n_i\}$, $t\in \{1,\ldots, n_j\}$.
Hence, we deduce that $\Phi_{S_i\otimes I_\cK}(P_\cM)$ commutes with  $\Phi_{S_j\otimes I_\cK}(P_\cM)$ and, consequently, that
 the operators $\left\{P_\cM-\Phi_{S_i\otimes I_\cK}(P_\cM)\right\}_{i=1}^k$ are commuting.   Now, we can use  relation \eqref{PM}, to  conclude that
$$(id-\Phi_{S_1\otimes I_\cK})\circ\cdots \circ (id-\Phi_{S_k\otimes I_\cK})(P_\cM)\geq0,
$$
 which completes the proof.
\end{proof}

We remark that   an  extension of Theorem \ref{Beurling-fac}  to  the pure elements in the $\Lambda$-polyball holds. We omit the proof which is very similar.

\begin{theorem} \label{Beurl-pure}
Let $X=(X_1,\ldots, X_k)\in B_\Lambda(\cH)$ be  a pure $k$-tuple and let $Y\in B(\cH)$ be a selfadjoint operator.  Then the following statements are equivalent.
\begin{enumerate}
\item[(i)] There is a Hilbert space $\cL$ and an   operator
 $\Psi: \ell^2(\FF_{n_1}^+\times\cdots \times \FF_{n_k}^+)\otimes\cL\to \cH$
  such that
$Y=\Psi\Psi^*$ and  $\Psi(S_{i,s}\otimes I_\cL)=X_{i,s}\Psi$ for any $i\in \{1,\ldots, k\}$ and $s\in \{1,\ldots, n_i\}$.

\item[(ii)] $(id-\Phi_{X_1} )\circ\cdots \circ
(id-\Phi_{X_k })(Y)\geq 0$.
\end{enumerate}
\end{theorem}

As a consequence of Theorem \ref{Beurl-pure}, we obtain the following characterization of the  Beurling type invariant subspaces for the  pure elements in $B_\Lambda(\cH)$.

\begin{corollary} Let $X=(X_1,\ldots, X_k)\in B_\Lambda(\cH)$ be  a pure $k$-tuple and let $\cM\subset \cH$ be a jointly invariant subspace under all $X_{i,s}$. Then the following statements are equivalent.
\begin{enumerate}
\item[(i)] There is a Hilbert space $\cE$ and a partial isometry
$\Psi: \ell^2(\FF_{n_1}^+\times\cdots \times \FF_{n_k}^+)\otimes \cE\to \cH$ such that
$$
\cM=\Psi\left(\ell^2(\FF_{n_1}^+\times\cdots \times \FF_{n_k}^+)\otimes \cE\right)
$$
and  $\Psi(S_{i,s}\otimes I_\cE)=X_{i,s}\Psi$ for any $i\in \{1,\ldots, k\}$ and $s\in \{1,\ldots, n_i\}$.

\item[(ii)] $(id-\Phi_{X_1})\circ\cdots \circ (id-\Phi_{X_k})(P_\cM)\geq 0$.
\end{enumerate}

\end{corollary}

In what follows, we prove that the dilation provided by Theorem \ref{Berezin} is minimal and unique up to an isomorphism.

 \begin{theorem}   \label{dil-Berezin} Let $T=(T_1,\ldots, T_k)$ be  a pure $k$-tuple in the regular $\Lambda$-polyball and let
  $$
K_{T}:\cH\to \ell^2(\FF_{n_1}^+\times\cdots \times \FF_{n_k}^+)\otimes \overline{\Delta_T(I)(\cH)},
$$
be
the noncommutative Berezin kernel . Then  the subspace $K_T\cH$ is jointly co-invariant under the operators  $S_{i,s}\otimes I_{\overline{\Delta_T(I)(\cH)}}$, where $i\in\{1,\ldots, k\}$ and  $s\in \{1,\ldots, n_i\}$, and the dilation provided by  Theorem \ref{Berezin} is minimal and unique up to an isomorphism.
  \end{theorem}
\begin{proof}
The fact that $K_T\cH$ is a jointly co-invariant  subspace under the operators  $S_{i,s}\otimes I_{\overline{\Delta_T(I)(\cH)}}$ is due to the relation
$$
K_T T_{i,s}^*=\left(S_{i,s}^*\otimes I_{\cD_T}\right) K_T.
$$
where $\cD_T:=\overline{\Delta_T(I)(\cH)}$, which was proved in Theorem \ref{Berezin}.
Note also that $K_T$ is an isometry and $(P_\CC\otimes I_{\cD_T})K_T\cH=\cD_T$.
Applying Theorem \ref{co-inv} to the subspace $K_T\cH$, we deduce that
\begin{equation}
\label{span1}
\overline{\text{\rm span}}\{(S_{1,\alpha_1}\dots S_{k,\alpha_k}\otimes I_\cH)K_T\cH:\ \alpha_1\in \FF_{n_1}^+,\ldots, \alpha_{n_k}\in \FF_{n_k}^+\}
=\ell^2(\FF_{n_1}^+\times\cdots \times \FF_{n_k}^+)\otimes \cD_T,
\end{equation}
which proves the minimality  of the dilation provided by the relation
\begin{equation}
\label{dil1}
T_{1,\alpha_1}\cdots T_{k,\alpha_k}=K_T^*(S_{1,\alpha_1}\dots S_{k,\alpha_k}\otimes I_{\cD_T})K_T.
\end{equation}
Now, we prove the uniqueness of the minimal dilation of $T$. Let $V:\cH\to \ell^2(\FF_{n_1}^+\times\cdots \times \FF_{n_k}^+)\otimes \cD$ be an isometry such that $V\cH$  is jointly co-invariant under the operators  $S_{i,s}\otimes I_{\overline{\Delta_T(I)(\cH)}}$, where $i\in\{1,\ldots, k\}$ and  $s\in \{1,\ldots, n_i\}$, and assume that
\begin{equation}
\label{dil2}
T_{1,\alpha_1}\cdots T_{k,\alpha_k}=V^*(S_{1,\alpha_1}\dots S_{k,\alpha_k}\otimes I_{\cD_T})V
\end{equation}
and
\begin{equation}
\label{span2}
\overline{\text{\rm span}}\{(S_{1,\alpha_1}\dots S_{k,\alpha_k}\otimes I_\cH)V\cH:\ \alpha_1\in \FF_{n_1}^+,\ldots, \alpha_{n_k}\in \FF_{n_k}^+\}
=\ell^2(\FF_{n_1}^+\times\cdots \times \FF_{n_k}^+)\otimes \cD.
\end{equation}
According to Corollary \ref{cp}, there is a completely positive linear map  $\Psi: C^*(\{S_{i,s}\})\to B(\cH)$ such that
\begin{equation}
\label{PST}
\Psi(p(\{S_{i,s}\}, \{S_{i,s}^*\}))=p(\{T_{i,s}\}, \{T_{i,s}^*\})
\end{equation}
for any polynomial  $p(\{S_{i,s}\}, \{S_{i,s}^*\})$ of the form \eqref{polynomial}.
Consider the $*$-representations
\begin{equation*}
\begin{split}
\pi_1&: C^*(\{S_{i,s}\})\to B(\ell^2(\FF_{n_1}^+\times\cdots \times \FF_{n_k}^+)\otimes \cD_T),\quad \pi_1(a)=a\otimes I_{\cD_T},\\
\pi_2&: C^*(\{S_{i,s}\})\to B(\ell^2(\FF_{n_1}^+\times\cdots \times \FF_{n_k}^+)\otimes \cD),\quad \pi_1(a)=a\otimes I_{\cD}
\end{split}
\end{equation*}
and note that relations \eqref{dil1}, \eqref{dil2}, and   \eqref{PST}, and the fact that the subspaces  $K_T\cH$  and $V\cH$ are jointly co-invariant under the operators  $S_{i,s}\otimes I_{\cD_T}$ and $S_{i,s}\otimes I_\cD$, repectively, imply
\begin{equation*}
\begin{split}
\Psi(p(\{S_{i,s}\}, \{S_{i,s}^*\}))&= K_T^*\pi_1(p(\{S_{i,s}\}, \{S_{i,s}^*\}))K_T\\
&=V^*\pi_2(p(\{S_{i,s}\}, \{S_{i,s}^*\})V
\end{split}
\end{equation*}
for any polynomial  $p(\{S_{i,s}\}, \{S_{i,s}^*\})$ of the form \eqref{polynomial}. Due to relations \eqref{span1} and \eqref{span2}, $\pi_1$ and $\pi_2$ are minimal Stinespring dilations of the completely positive linear map $\Psi$. Using now Stinespring result \cite{St}, we deduce that these representations are isomorphic. Therefore, there exists a unitary operator
$$
U:\ell^2(\FF_{n_1}^+\times\cdots \times \FF_{n_k}^+)\otimes \cD_T  \to \ell^2(\FF_{n_1}^+\times\cdots \times \FF_{n_k}^+)\otimes \cD
$$
such that
$
U(S_{i,s}\otimes I_{\cD_T})=(S_{i,s}\otimes I_\cD)U
$
 for any $i\in \{1,\dots, k\}$, $s\in \{1,\ldots, n_i\}$, and $UK_T=V$. Since $U$ is a unitary operator, we also deduce that
$
U(S_{i,s}^*\otimes I_{\cD_T})=(S_{i,s}^*\otimes I_\cD)U.
$
Now, using the fact that the $C^*$-algebra  $C^*(\{S_{i,s}\})$ is irreducible  (see Corollary  \ref{irred}), we conclude that $U=I\otimes \Omega$, where $\Omega:\cD_T\to \cD$ is a unitary operator. Consequently, we have $\dim \cD_T=\dim \cD$ and $UK_T\cH=V\cH$.
The proof is complete.
\end{proof}

In what follows, we prove a classification result for the pure $k$-tuples in the regular $\Lambda$-polyball.

\begin{theorem}  \label{classif1} Let $T=(T_1,\ldots, T_k)\in B(\cH)^{n_1}\times \cdots \times B(\cH)^{n_k}$.
Then $T$ is a pure $k$-tuple in the regular $\Lambda$-polyball, with
$\rank \Delta_T(I)=m$,  where $m\in \NN$ or $m=\infty$,  if and only if $T=(T_1,\ldots, T_k)$ is unitarily equivalent to the compression of $(S_1\otimes I_{\CC^m},\ldots, S_k\otimes I_{\CC^m})$ to a jointly co-invariant subspace $\cM\subset \ell^2(\FF_{n_1}^+\times\cdots \times \FF_{n_k}^+)\otimes \CC^n$ under $S_{i,s}\otimes I_{\CC^m}$,  where $i\in\{1,\ldots, k\}$ and  $s\in \{1,\ldots, n_i\}$, with the property that
$\dim [(P_\CC\otimes I_{\CC^m})\cM]=m$.
\end{theorem}
\begin{proof}
Note that the direct implication is due to Theorem  \ref{dil-Berezin} and its proof.
To prove the converse, assume that
\begin{equation}\label{dil3}
T_{1,\alpha_1}\cdots T_{k,\alpha_k}=P_\cH(S_{1,\alpha_1}\dots S_{k,\alpha_k}\otimes I_{\cD_T})|_\cH, \qquad \alpha_i\in \FF_{n_i}^+,
\end{equation}
where
$\cH\subset \ell^2(\FF_{n_1}^+\times\cdots \times \FF_{n_k}^+)\otimes \CC^m$  is a co-invariant subspace under $S_{i,s}\otimes I_{\CC^m}$,  where
$i\in\{1,\ldots, k\}$ and  $s\in \{1,\ldots, n_i\}$, such that $\dim [(P_\CC\otimes I_{\CC^m})\cH]=m$.
First, we note that $T$ is a pure $k$-tuple in the regular $\Lambda$-polyball. Next, we consider that case when $m\in \NN$. Since $(P_\CC\otimes I_{\CC^m})\cH\subset \CC^m$ and
$\dim [(P_\CC\otimes I_{\CC^m})\cH]=m$, it clear that  $(P_\CC\otimes I_{\CC^m})\cH= \CC^m$.
Hence, $\cH^\perp \cap \CC^m=\{0\}$.
On the other hand, we have
$(id-\Phi_{S_1})\circ\cdots \circ (id-\Phi_{S_k})(I)=P_\CC$  where $\CC$ is identified with $\CC\chi_{(g_0^1,\ldots, g_0^k)}$.
Consequently, relation \eqref{dil3} implies
\begin{equation*}
\begin{split}
\Delta_T(I)&:=(id-\Phi_{T_1})\circ\cdots \circ (id-\Phi_{T_k})(I)=P_\cH(\Delta_S(I)\otimes I_{\CC^m})|_\cH\\
&=P_\cH(P_\CC \otimes I_{\CC^m})|_\cH=P_\CC\CC^m.
\end{split}
\end{equation*}
Hence, we deduce that  $\dim \cD_T=\dim P_\CC \CC^m\leq m$. If we assume that $\dim \cD_T<m$, then there is $h\in \CC^m$, $h\neq 0$ such that $P_\cH h\neq 0$. This shows that $h\in \cH^\perp$, which contradicts the relation $\cH^\perp \cap \CC^m=\{0\}$. In conclusion,
$\dim \cD_T=m$.

Now, we consider the case when $m=\infty$. Due to relation \eqref{dil3} and using the proof of Theorem \ref{co-inv} , we have
$$
\overline{\text{\rm span}}\{(S_{1,\alpha_1}\dots S_{k,\alpha_k}\otimes I_\cH)\cH: \ \alpha_1\in \FF_{n_1}^+,\ldots, \alpha_{n_k}\in \FF_{n_k}^+\}
=\ell^2(\FF_{n_1}^+\times\cdots \times \FF_{n_k}^+)\otimes \cL,
$$
where $\cL=(P_\CC\otimes I_{\CC^m})\cH$. Taking into account that
$\ell^2(\FF_{n_1}^+\times\cdots \times \FF_{n_k}^+)\otimes \cL$ is jointly reducing for  all the isometries $S_{i,s}\otimes I_{\CC^m}$, one can see that relation \eqref{dil3} implies
\begin{equation*}
T_{1,\alpha_1}\cdots T_{k,\alpha_k}=P_\cH(S_{1,\alpha_1}\dots S_{k,\alpha_k}\otimes I_{\cL})|_\cH, \qquad \alpha_i\in \FF_{n_i}^+,
\end{equation*}
Due to the uniqueness of the minimal dilation of $T=(T_1,\ldots, T_k)$ (see Theorem \ref{dil-Berezin}), we conclude that  $\dim \cD_T=\dim \cL=\infty$. The proof is complete.
\end{proof}

The next result provides a classification  of the pure elements of rank one in  the regular $\Lambda$-polybal .

\begin{corollary} Let $T=(T_1,\ldots, T_k)\in B(\cH)^{n_1}\times \cdots \times B(\cH)^{n_k}$.
 Then $T$ is a pure element in the regular $\Lambda$-polyball such that $\rank \Delta_T(I)=1$  if and only if  there is  a jointly co-invariant subspace
$\cM\subset \ell^2(\FF_{n_1}^+\times\cdots \times \FF_{n_k}^+)$   under  the isometries $S_{i,s}$,  where $i\in\{1,\ldots, k\}$ and  $s\in \{1,\ldots, n_i\}$, such that $T$ is  jointly unitarily equivalent to  the compression
$P_\cM S|_\cM:=(P_\cM S_{1}|_\cM, \ldots, P_\cM S_{k}|_\cM)$, where
$$
P_\cM S_{i}|_\cM:=[P_\cM S_{i,1}|_\cM \cdots  P_\cM S_{i,n_i}|_\cM].
$$
If $\cM'$ is another jointly co-invariant subspace under $S_{i,s}$, then
$
P_\cM S|_\cM$
  and
$P_{\cM'} S|_{\cM'} $
are unitarily equivalent if and only if $\cM=\cM'$.
\end{corollary}
\begin{proof}
The direct implication  is due to Theorem \ref{classif1}. To prove the converse, assume that
$T=(T_1,\ldots, T_k)$ where $T_{i,s}:=P_\cM S_{i,s}|_\cM$ and $\cM$ is  a jointly co-invariant subspace
$\cM\subset \ell^2(\FF_{n_1}^+\times\cdots \times \FF_{n_k}^+)$   under  the isometries $S_{i,s}$.
Note that
\begin{equation*}
\Delta_T(I):=(id-\Phi_{T_1})\circ\cdots \circ (id-\Phi_{T_k})(I)=P_\cM(\Delta_S(I)\otimes I_{\CC})|_\cM
=P_\cM P_\CC |_\cM
\end{equation*}
and, consequently, $\dim \cD_T\leq 1$. Since $S=(S_1,\ldots, S_k)$ is a pure $k$-tuple, so is $T$. Thus $\Delta_T(I)\neq 0$, which shows that $\dim \cD_T= 1$.

To prove the last part of the corollary, note that, as in the proof of  Theorem \ref{dil-Berezin}, we can show that  the two $k$-tuples are unitarily equivalent if and only if there is a unitary operator
$$
U:\ell^2(\FF_{n_1}^+\times\cdots \times \FF_{n_k}^+) \to \ell^2(\FF_{n_1}^+\times\cdots \times \FF_{n_k}^+)
$$
such that
$
US_{i,s}=S_{i,s}U
$
 for any $i\in \{1,\dots, k\}$, $s\in \{1,\ldots, n_i\}$, and $U\cM=\cM$. Since $U$ is a unitary operator, we also deduce that
$
U(S_{i,s}^*\otimes I_{\cD_T})=(S_{i,s}^*\otimes I_\cD)U.
$
Since  $C^*$-algebra  $C^*(\{S_{i,s}\})$ is irreducible, we must have  $U=cI$ for some constant $c\in \TT$. Consequently, $\cM=U\cM=\cM'$. The proof is complete.
\end{proof}

\bigskip

\section{Dilation theory on regular $\Lambda$-polyballs}

In this section, we show that any $k$-tuple in the regular $\Lambda$-polyball admits  a minimal  dilation which is a $k$-tuple of doubly $\Lambda$-commuting row isometries, uniquely determined up to an isomorphism.
We show that $T$ is a pure element in ${\bf B}_\Lambda(\cH)$ if and only if    its minimal isometric dilation  is a pure element in
 ${\bf B}_\Lambda(\cK)$. The  universal algebra  generated by a $k$-tuple $V=(V_1,\ldots, V_k)$ of doubly $\Lambda$-commuting row isometries such
$\Delta_V(I)=0$ is identified.
In  the particular case when $n_1=\cdots =n_k=1$, we obtain  an  extension of  Brehmer's  result, showing that any $k$-tuple in the $\Lambda$-polyball  admits a unique minimal
  doubly $\Lambda$-commuting unitary
    dilation.

\begin{theorem}   \label{iso-dil} Let   $T=(T_1,\ldots, T_k)$ be   a  $k$-tuple in the regular $\Lambda$-polyball ${\bf B}_\Lambda(\cH)$. Then there is   a Hilbert $\cK\supset \cH$ and a $k$-tuple $V=(V_1,\ldots, V_k)$ of doubly $\Lambda$-commuting row isometries  on $\cK$ such that
$$
T_{i,s}^*=V_{i,s}^*|_\cH,\qquad i\in \{1,\ldots, \}, s\in \{1,\ldots, n_i\},
$$
and such that the dilation is minimal, i.e.
$$
\cK=\overline{\text{\rm span}}
\left\{  V_{1,\alpha_1}\cdots V_{k,\alpha_k}\cH
 :\ \alpha_1\in \FF_{n_1}^+,\ldots, \alpha_k\in \FF_{n_k}^+\right\}.
 $$
 Moreover, the minimal dilation is unique up to  an isomorphism.

\end{theorem}
\begin{proof}
Let $T=(T_1,\ldots, T_k)$ with $T_i=[T_{i,1}\cdots T_{i,n_i}]$ and $T_{i,s}\in B(\cH)$. According to Corollary \ref{cp},
$T\in {\bf B}_\Lambda(\cH)$ if and only if
there is a completely positive linear map  $\Psi: C^*(\{S_{i,s}\})\to B(\cH)$ such that
$$
\Psi(p(\{S_{i,s}\}, \{S_{i,s}^*\}))=p(\{T_{i,s}\}, \{T_{i,s}^*\})
$$
for any polynomial $p(\{S_{i,s}\}, \{S_{i,s}^*\})$ of the form \eqref{polynomial}.
Let $\pi:C^*(\{S_{i,s}\})\to B(\cK)$ be the minimal Stinespring dilation \cite{St} of $\Psi$. Then we have
$\Psi\left(p(\{S_{i,s}\}, \{S_{i,s}^*\})\right)=P_\cH \pi\left(p(\{S_{i,s}\}, \{S_{i,s}^*\})\right)|_\cH
$
  and
\begin{equation}
\label{minimal}
\cK=\overline{\text{\rm span}}\left\{ \pi\left(p(\{S_{i,s}\}, \{S_{i,s}^*\})\right)\cH: \ p(\{S_{i,s}\}, \{S_{i,s}^*\})\in C^*(\{S_{i,s}\})\right\}.
\end{equation}
In what follows, we prove that $P_\cH \pi(S_{1,\alpha_1}\cdots S_{k,\alpha_k})|_{\cH^\perp}=0$ for any $\alpha_i\in \FF_{n_i}^+$.
Indeed, we have
\begin{equation*}
\begin{split}
&T_{1,\alpha_1}\cdots T_{k,\alpha_k}T_{k,\alpha_k}^*\cdots T_{1,\alpha_1}^*\\
&\qquad=P_\cH\pi(S_{1,\alpha_1}\cdots S_{k,\alpha_k})\pi(S_{k,\alpha_k}^*\cdots S_{1,\alpha_1}^*)|_\cH\\
&\qquad=P_\cH\pi(S_{1,\alpha_1}\cdots S_{k,\alpha_k})(P_\cH+P_{\cH^\perp})\pi(S_{k,\alpha_k}^*\cdots S_{1,\alpha_1}^*)|_\cH\\
&\qquad= P_\cH\pi(S_{1,\alpha_1}\cdots S_{k,\alpha_k})P_\cH \pi(S_{k,\alpha_k}^*\cdots S_{1,\alpha_1}^*)|_\cH
+P_\cH\pi(S_{1,\alpha_1}\cdots S_{k,\alpha_k})P_{\cH^\perp}\pi(S_{k,\alpha_k}^*\cdots S_{1,\alpha_1}^*)|_\cH\\
&\qquad =T_{1,\alpha_1}\cdots T_{k,\alpha_k}T_{k,\alpha_k}^*\cdots T_{1,\alpha_1}^*
+\left(P_\cH\pi(S_{1,\alpha_1}\cdots S_{k,\alpha_k})P_{\cH^\perp}\right)\left(P_\cH\pi(S_{1,\alpha_1}\cdots S_{k,\alpha_k})P_{\cH^\perp}\right)^*.
\end{split}
\end{equation*}
Hence, we deduce that  $P_\cH \pi(S_{1,\alpha_1}\cdots S_{k,\alpha_k})|_{\cH^\perp}=0$, which shows that $ \pi(S_{1,\alpha_1}\cdots S_{k,\alpha_k})({\cH^\perp})\subset \cH^\perp$ and, consequently,  $ \pi(S_{1,\alpha_1}\cdots S_{k,\alpha_k})^*({\cH})\subset \cH$. Note that the later relation implies
$$
T_{k,\alpha_k}^*\cdots T_{1,\alpha_1}^*=P_\cH \pi(S_{k,\alpha_k})^*\cdots \pi(S_{1,\alpha_1})^*|_\cH=\pi(S_{k,\alpha_k})^*\cdots \pi(S_{1,\alpha_1})^*|_\cH
$$
for any $\alpha_i\in \FF_{n_i}^+$. Moreover, relation \eqref{minimal} implies
\begin{equation*}
\cK=\overline{\text{\rm span}}\left\{ \pi(S_{1,\alpha_1})\cdots \pi(S_{k,\alpha})\cH :  \ \alpha_i\in \FF_{n_i} \right\}.
\end{equation*}
Setting $V_{i,s}:=\pi(S_{i,s})$, we complete the proof of the existence of the minimal dilation of  $T$.

To prove the uniqueness, let $V'=(V_1',\ldots, V_k')$ be another minimal dilation of $T$  on a Hilbert space $\cK'\supset\cH$ such that
$$
T_{i,s}^*={V'}_{i,s}^*|_\cH,\qquad i\in \{1,\ldots, \}, s\in \{1,\ldots, n_i\},
$$
and  $$
\cK'=\overline{\text{\rm span}}
\left\{  V'_{1,\alpha_1}\cdots V'_{k,\alpha_k}\cH
 :\ \alpha_1\in \FF_{n_1}^+,\ldots, \alpha_k\in \FF_{n_k}^+\right\}.
 $$
Consider the representation $\rho:C^*(\{S_{i,s}\})\to B(\cK')$ defined by
$\rho(p(\{S_{i,s}\}, \{S_{i,s}^*\}):=p(\{V'_{i,s}\}, \{{V'}_{i,s}^*\})$ for any
$p(\{S_{i,s}\}, \{S_{i,s}^*\})\in C^*(\{S_{i,s}\})$ and note that
$$
\cK'=\overline{\text{\rm span}}\left\{ \rho\left(p(\{S_{i,s}\}, \{S_{i,s}^*\})\right)\cH: \ p(\{S_{i,s}\}, \{S_{i,s}^*\})\in C^*(\{S_{i,s}\})\right\}.
$$
On the other hand, we have
\begin{equation*}
\begin{split}
\Psi(p(\{S_{i,s}\}, \{S_{i,s}^*\}))&=p(\{T_{i,s}\}, \{T_{i,s}^*\}\\
&=P_\cH \pi(p(\{S_{i,s}\}, \{S_{i,s}^*\}))|_\cH\\
&=P_\cH \rho(p(\{S_{i,s}\}, \{S_{i,s}^*\}))|_\cH
\end{split}
\end{equation*}
for any polynomial  $p(\{S_{i,s}\}, \{S_{i,s}^*\})\in C^*(\{S_{i,s}\})$  of the form
\eqref{polynomial}. Consequently, $\pi$ and $\rho$ are minimal Stinespring dilations of the completely positive linear mao $\Psi$. Due to the uniqueness of the minimal Stinespring dilations, there is a unitary operator $U:\cK\to \cK'$ such that $U V_{i,s}=V'_{i,s}U$  for any
$i\in \{1,\ldots, \}$, $s\in \{1,\ldots, n_i\}$, and $U\cH=\cH$. The proof is complete.
\end{proof}

\begin{theorem} \label{Wold-partic2} Let  $V=(V_1,\ldots, V_k)$ be  a $k$-tuple of doubly $\Lambda$-commuting row isometries acting on a Hilbert space $\cK$. Then there is a unique
othogonal decomposition
$$
\cK=\cK^{(s)}\oplus \cK^{(\Delta)},
$$
where $\cK^{(s)}, \cK^{(\Delta)}$ are reducing  subspaces under all isometries $V_{i,s}$,  for $i\in \{1,\ldots, k\}$, $s\in \{1,\ldots, n_i\}$, with the following properties.
\begin{enumerate}
\item[(i)] $V|_{\cK^{(s)}}$ is  a  $k$-tuple of doubly $\Lambda$-commuting pure row isometries, which is isomorphic to the standard $k$-tuple ${\bf S}=({\bf S}_1,\ldots, {\bf S}_k)$ with wandering subspace of dimension  equal to $\dim \Delta_V(I)\cK$.

\item[(ii)]   $V|_{\cK^{(\Delta)}}$ is  a   $k$-tuple of doubly $\Lambda$-commuting   row isometries such that $\Delta_V(I_{\cK^{(\Delta)}})=0$.

\end{enumerate}
Moreover, we have
$$
 \cK^{(s)} =\bigoplus_{\alpha_{1}\in \FF_{n_{1}}^+,\ldots, \alpha_{k}\in \FF_{n_{k}}^+} V_{1,\alpha_{1}}\cdots V_{k,\alpha_{k}}\left(\Delta_V(I)\cK\right) \quad \text{and}\quad \cK^{(\Delta)}={\cK^{(s)}}^\perp.
 $$
\end{theorem}

\begin{proof} According to Theorem \ref{Wold2}, the Hilbert space $\cK$ admits a unique orthogonal decomposition
$$
\cK=\cK_{\{1,\ldots, k\}}\oplus \cK^{(\Delta)}
$$
with the following properties.
\begin{enumerate}
\item[(i)]   The subspaces $\cK_{\{1,\ldots, k\}}$  and  $\cK^{(\Delta)}:=\bigoplus_{{A\subset \{1,\ldots, k\}}\atop{A\neq \{1,\ldots, k\}}} \cK_A$ are reducing    for all the isometries $V_{i,s}$, where $i\in \{1,\ldots, k\}$ and $s\in \{1,\ldots, n_i\}$;

\item[(ii)]    $V_i|_{\cK_{\{1,\ldots, k\}}}:=[V_{i,1}|_{\cK_{\{1,\ldots, k\}}}\cdots V_{i,n_i}|_{\cK_{\{1,\ldots, k\}}}]$ is a pure row isometry for any $i\in \{1,\ldots,k\}$.

\item[(iii)]   $V_i|_{\cK^{(\Delta)}}:=[V_{i,1}|_{\cK^{(\Delta)}}\cdots V_{i,n_i}|_{\cK^{(\Delta)}}]$  is a $k$-tuple of doubly $\Lambda$-commuting   row isometries for any $i\in \{1,\ldots,k\}$.
\end{enumerate}
Due to  Theorem \ref{Wold3} and Remark \ref{partic}, we have
$
\cK_{\{1,\ldots, k\}}=\cK^{(s)}=\cap_{i=1}^k \cK_i^{(s)}$,
where $\cK_i^{(s)}$ is  defined by relation \eqref{alt}.
According to the results of Section 2, $V|_{\cK^{(s)}}$ is   isomorphic to the standard $k$-tuple ${\bf S}=({\bf S}_1,\ldots, {\bf S}_k)$ with wandering subspace of dimension  equal to $\dim \Delta_V(I)\cK$.
On the other hand, we saw in the proof of Theorem \ref{Wold2} that  $\cK_A:=P_A \cK$, where
$P_A:=\left(\prod_{i\in A}P_i^{(s)}\right)\left(\prod_{i\in A^c}P_i^{(c)}\right)$    and $P_i^{(s)}$, $P_i^{(c)}$  are    the orthogonal projections of $\cK$ onto  the subspaces
 $\cK_i^{(s)}$ and $\cK_i^{(c)}$, respectively. Note that, if  $A\subset \{1,\ldots, k\}$ and $A\neq \{1,\ldots, k\}$ then there is $i\in A$ such that
  $V_i|_{\cK_A}:=[V_{i,1}|_{\cK_A}\cdots V_{i,n_i}|_{\cK_A}]$  is a Cuntz row isometry. Consequently, we have $(I-\sum_{s=1}^{n_i} V_{i,s} V_{i,s}^*)|_{\cK_A}=0$. Now, taking into account that $\Delta_V(I_{\cK})=\prod_{i=1}^k (I-\sum_{s=1}^{n_i} V_{i,s} V_{i,s}^*)$ and the factors of this product are commuting, we deduce that
   $\Delta_V(I_{\cK^{(\Delta)}})=0$. The uniqueness of the decomposition  can be proved as in the proof of Theorem \ref{Wold-partic}.
The proof is complete.
\end{proof}

\begin{corollary} Let $T=(T_1,\ldots, T_k)$ be a  $k$-tuple in the regular $\Lambda$-polyball
 ${\bf B}_\Lambda(\cH)$ and let $V=(V_1,\ldots, V_k)$ be its minimal dilation on a Hilbert space $\cK$.
 Then the following  statements hold.
 \begin{enumerate}
 \item[(i)]  $V$ is a pure element in ${\bf B}_\Lambda(\cK)$ if and only if $T$ is a pure element in ${\bf B}_\Lambda(\cH)$.
 \item[(ii)]
 $\Delta_V(I_\cK)=0$ if and only if $\Delta_T (I_\cH)=0$.
 \end{enumerate}

\end{corollary}
\begin{proof}
For each $(p_1,\ldots, p_k)\in \NN^k$, denote
$$
\Delta_T^{(p_1,\ldots, p_k)}:=(id -\Phi_{T_1}^{p_1})\circ\cdots \circ (id-\Phi_{T_k}^{p_k})(I).
$$
Due to Lemma \ref{commutative} and Lemma \ref{tech}, we have
$$
\Delta_T^{(p_1,\ldots, p_k)}=\prod_{i=1}^k(I-\Phi_{T_i}^{p_i}(I))
$$
and the order of the factors in the product above is irrelevant.
Note that, since $T_i$ is a row contraction,  $\{(id-\Phi_{T_i}^{p_i}(I)\}_{p_i=1}^\infty$ is an increasing sequence of positive operators and  $\Delta_T^{(p_1,\ldots, p_k)}$ is an increasing sequence of positive operators with respect to $p_1,p_2\ldots, p_k$.
In what follows, we show that $T$ is a pure element in ${\bf B}_\Lambda(\cH)$ if and only if
$$
\text{\rm SOT-}\lim_{(p_1,\ldots, p_k)\in \NN^k} \Delta_T^{(p_1,\ldots, p_k)}=I.
$$
Let $A_i:=\text{\rm SOT-}\lim_{p_i\to \infty}\Phi_{T_i}^{p_i}(I)$ and note that
$$
\text{\rm SOT-}\lim_{(p_1,\ldots, p_k)\in \NN^k} \Delta_T^{(p_1,\ldots, p_k)}=(I-A_1)\cdots (I-A_k).
$$
To prove our assertion, it is enough to show that $A_1=\cdots =A_k=0$ if and only if
$(I-A_1)\cdots (I-A_k)=I$. The direct implication is clear. To prove the converse, assume that
$(I-A_1)\cdots (I-A_k)=I$ and that there is $i\in \{1,\ldots, k\}$ such that $A_i\neq 0$. Without loss of generality, we may assume that $A_1\neq 0$. Note that  $0\leq I-A_i\leq I$ and the operators $A_1,\ldots, A_k$ are pairwise commuting.
 Since $A_1\neq 0$, there exists  $x\neq 0$  such that
$\left<I-A_1)x,x\right> <1$.  Consequently,   $\|(I-A_1)^{1/2}x\|<\|x\|^2$ and
\begin{equation*}
\begin{split}
\left<\prod_{i=2}^k(I-A_i) (I-A_1)x, x\right>&\leq
\left< (I-A_1)x, \prod_{i=2}^k(I-A_i)x\right>\\
&\leq \|(I-A_1)x\|\left\|\prod_{i=2}^k(I-A_i)\right\|\\
&\leq \|(I-A_1)^{1/2}\|\|(I-A_1)^{1/2}x\|\|x\|<\|x\|
\end{split}
\end{equation*}
which contradicts that $(I-A_1)\cdots (I-A_k)=I$, and proves our assertion.

Now, to prove  item (i), we use  Theorem \ref{iso-dil} and the Wold decomposition of Theorem \ref{Wold-partic2} to obtain
$$
\Delta_T^{(p_1,\ldots, p_k)}=P_\cH^\cK\left[\begin{matrix} \Delta_S^{(p_1,\ldots, p_k)}\otimes I_{\cD_T}&0\\0&0
\end{matrix} \right]|_\cH
$$
with respect to the orthogonal decomposition $\cK=\cK^{(s)}\oplus \cK^{(\Delta)}$.
Using the results above, one can see that  $T$ is a pure $k$-tuple in ${\bf B}_\Lambda(\cH)$ if and only if  $$I_\cH=P_\cH^\cK\left[\begin{matrix} I_{\ell^2(\FF_{n_1}\times \cdots \times \FF_{n_k}^+)}\otimes I_{\cD_T}&0\\0&0
\end{matrix} \right]|_\cH.
$$
If $h\in \cH$ and $h=h_0+h_1$, where $h_0\in \cK^{(s)}$ and $h_1\in \cK^{(\Delta)}$, then
the relation above is equivalent to $h_0+h_1=P^\cK_\cH h_0$, which is equivalent to $\cH\subset \cK^{(s)}$. Indeed, we have
$$
\|h_0\|^2\geq \|P^\cK_\cH h_0\|^2=\|h_0\|^2+\|h_1\|^2\geq |h_0\|^2,
$$
which is equivalent to $h_1=0$ and, therefore,
$\cH\subset \cK^{(s)}=\ell^2(\FF_{n_1}\times \cdots \times \FF_{n_k}^+)\otimes I_{\cD_T}$.
Since the latter subspace is reducing for all $V_{i,s}$ and $\cK$ is the smallest reducing subspace for $V_{i,s}$ which contains $\cH$, we must have
$\cK=\ell^2(\FF_{n_1}\times \cdots \times \FF_{n_k}^+)\otimes I_{\cD_T}$, which proves part (i).

To prove part (ii), note that
$$
\Delta_V(I_\cK)=\left[\begin{matrix} \Delta_S (I)\otimes I_{\cD_T}&0\\0&0
\end{matrix} \right]
$$
which implies that  $\Delta_V(I_\cK)=0$ if and only if  $\Delta_S (I)\otimes I_{\cD_T}$. Since
$\Delta_S (I)=P_\CC$, it is clear that  $\Delta_V(I_\cK)=0$ if and only if $\cD_T=\{0\}$, i.e. $\Delta_T(I)=0$.
The proof is complete.
\end{proof}

We can apply Theorem 1.3.1 from \cite{Arv} to our setting to obtain the following commutant lifting theorem for the $C^*$-algebra $C^*(\{T_{i,s}\})$. If $\cS\subset B(\cH)$, we denote by $\cS'$ the commutant of $\cS$.

\begin{corollary} Let $T=(T_1,\ldots, T_k)\in B_\Lambda(\cH)$ and let $V=(V_1,\ldots, V_k)\in B_\Lambda(\cK)$ be its minimal isometric dilation on a Hilbert $\cK\supset \cH$. If $X\in C^*(\{T_{i,s}\})^\prime$, then there is a unique $\widetilde X\in C^*(\{V_{i,s}\})^\prime$ such that
$$
\widetilde X P_\cH=P_\cH X \quad \text{ and } \quad  X=P_\cH \widetilde X |_\cH,
$$
where $P_\cH$ is the orthogonal projection of $\cK$ onto $\cH$. Moreover the map $X\mapsto \widetilde X$ is a $*$-isomorphism.
\end{corollary}

In what follows, we need the following result.

 \begin{theorem}  \label{compact} Let $S=(S_1,\ldots, S_k)$ be the universal model associated with  the regular $\Lambda$-polyball.
    Then the closed
two-sided ideal generated by the projection
$\Delta_S(I):=\prod_{i=1}^k\left(I -\sum_{s=1}^{n_i} S_{i,s}S_{i,s}^*\right)$ in the $C^*$-algebra $C^*(\{S_{i,s}\})$ coincides with the ideal of all  compact operators in $B(\ell^2(\FF_{n_1}^+\times\cdots \times \FF_{n_k}^+))$.
\end{theorem}
\begin{proof}
Let  $q=\sum_{\alpha_i\in \FF_{n_i}^+} a_{(\alpha_1,\ldots, \alpha_k)} \chi_{(\alpha_1,\ldots, \alpha_k)}$  be  a vector in
 $ \ell^2(\FF_{n_1}^+\times\cdots \times \FF_{n_k}^+)$ and consider the polynomial
$q_m(\{S_{i,s}\}):= \sum_{\alpha_i\in \FF_{n_i}^+, |\alpha_1|+\cdots +|\alpha_k|\leq m} a_{(\alpha_1,\ldots, \alpha_k)} \chi_{(\alpha_1,\ldots, \alpha_k)}$. It is easy to see that
\begin{equation}
\label{qm}
q_m(\{S_{i,s}\})\chi_{(g_0^1,\ldots, g_0^k)}\to q, \quad \text{ as } m\to \infty.
\end{equation}
Similarly, let $p\in  \ell^2(\FF_{n_1}^+\times\cdots \times \FF_{n_k}^+)$ and, as above, let
$p_m(\{S_{i,s}\})$ be the associated operator such that
\begin{equation}
\label{pm}
p_m(\{S_{i,s}\})\chi_{(g_0^1,\ldots, g_0^k)}\to p, \quad \text{ as } m\to \infty.
\end{equation}
In what follows, we show that  the operator
$q_m(\{S_{i,s}\}) P_\CC p_m(\{S_{i,s}\})^*$  has rank one and it  is in the
 $C^*$-algebra  $C^*(\{S_{i,s}\})$.
Indeed, note that, if $f\in  \ell^2(\FF_{n_1}^+\times\cdots \times \FF_{n_k}^+)$, then
$$
P_\CC p_m(\{S_{i,s}\})^*f=\left<f, p_m(\{S_{i,s}\}) \chi_{(g_0^1,\ldots, g_0^k)}\right>,
$$
which implies
\begin{equation}
\label{qPp}
q_m(\{S_{i,s}\})P_\CC p_m(\{S_{i,s}\})^*f=\left<f, p_m(\{S_{i,s}\}) \chi_{(g_0^1,\ldots, g_0^k)}\right>q_m(\{S_{i,s}\})\chi_{(g_0^1,\ldots, g_0^k)}.
\end{equation}
Using the fact that
$
P_\CC=(id-\Phi_{S_1})\circ \cdots \circ (id-\Phi_{S_k})(I),
$
one can see that  $q_m(\{S_{i,s}\}) P_\CC p_m(\{S_{i,s}\})^*$  is a  rank one operator  in    $C^*(\{S_{i,s}\})$.

Now, we prove that
$$
q_m(\{S_{i,s}\}) P_\CC p_m(\{S_{i,s}\})^*\to A,\quad \text{ as } m\to\infty,
$$
in the operator norm topology, where $A$ is the rank one operator $f\mapsto \left<f,p\right>q$.
Indeed,  using relation \eqref{qPp} and setting
$$
\Omega_{p_n,q_n}(f):=\left<f, p_m(\{S_{i,s}\}) \chi_{(g_0^1,\ldots, g_0^k)}\right>q_m(\{S_{i,s}\})\chi_{(g_0^1,\ldots, g_0^k)},
$$
we have
\begin{equation*}
\begin{split}
&\left\|q_n(\{S_{i,s}\})P_\CC p_n(\{S_{i,s}\})^*f-q_m(\{S_{i,s}\})P_\CC p_m(\{S_{i,s}\})^*f\right\|\\
&\qquad= \|\Omega_{p_n,q_n}(f)-\Omega_{p_m,q_m}(f)\|\\
&\qquad=\|\Omega_{p_n,q_n}(f)-\Omega_{p_m,q_n}(f)+\Omega_{p_m,q_n}(f)-\Omega_{p_m,q_m}(f)\|
\\
&\qquad\leq \|\Omega_{p_n-p_m,q_n}(f)\|+ \|\Omega_{p_m,q_n-q_m}(f)\|\\
&\qquad\leq  \|f\|\|(p_n(\{S_{i,s}\})-p_m(\{S_{i,s}\})) \chi_{(g_0^1,\ldots, g_0^k)}\|
\| q_m(\{S_{i,s}\}) \chi_{(g_0^1,\ldots, g_0^k)}\|\\
&\qquad\qquad +
\|f\|\|(q_n(\{S_{i,s}\})-q_m(\{S_{i,s}\})) \chi_{(g_0^1,\ldots, g_0^k)}\|
\| p_m(\{S_{i,s}\}) \chi_{(g_0^1,\ldots, g_0^k)}\|\\
&\qquad \leq
\|f\|\|p\|\|(p_n(\{S_{i,s}\})-p_m(\{S_{i,s}\})) \chi_{(g_0^1,\ldots, g_0^k)}\|
+ \|f\|\|q\|\|(q_n(\{S_{i,s}\})-q_m(\{S_{i,s}\})) \chi_{(g_0^1,\ldots, g_0^k)}\|.
\end{split}
\end{equation*}
Consequently, due to  relations \eqref{qm} and \eqref{pm}, we deduce that the sequence
$\{q_n(\{S_{i,s}\})P_\CC p_n(\{S_{i,s}\})^*\}_{n=1}^\infty$ is  Cauchy in the operator norm and therefore convergent in norm. Moreover,  relation \eqref{qPp} implies
$$
q_m(\{S_{i,s}\})P_\CC p_m(\{S_{i,s}\})^*f\to \left<f,p\right>q
$$
for any $f\in  \ell^2(\FF_{n_1}^+\times\cdots \times \FF_{n_k}^+)$. Combining hese results, we conclude  that $q_m(\{S_{i,s}\})P_\CC p_m(\{S_{i,s}\})^*$ converges in the operator norm to the rank one operator $\left<\cdot,  p\right>q$, as $m\to \infty$.
Therefore, all rank one operators on the Hilbert space $\ell^2(\FF_{n_1}^+\times\cdots \times \FF_{n_k}^+)$ are in    $C^*(\{S_{i,s}\})$ and, consequently,  all compact operators in
$B( \ell^2(\FF_{n_1}^+\times\cdots \times \FF_{n_k}^+))$ are in  $C^*(\{S_{i,s}\})$. The proof is complete.
\end{proof}

Let $C_\Delta^*(\{V_{i,s}\})$ be the universal algebra  generated by a $k$-tuple $V=(V_1,\ldots, V_k)$ of doubly $\Lambda$-commuting row isometries such
$\Delta_V(I)=0$.

\begin{theorem}  Let $S=(S_1,\ldots, S_k)$ be  the standard $k$-tuple of doubly $\Lambda$-commuting row  isometries  $S_i=[S_{i,1}\cdots S_{i,n_i}]$ and let  $\boldsymbol\cK$ be the ideal of all  compact operators in $B(\ell^2(\FF_{n_1}^+\times\cdots \times \FF_{n_k}^+))$.  Then the sequence
of $C^*$-algebras
$$
0\to \boldsymbol\cK\to C^*(\{S_{i,s}\})\to C_\Delta^*(\{V_{i,s}\})\to 0
$$
is exact.
\end{theorem}
\begin{proof}  The proof is similar to that of Theorem \ref{exact1}. However, we outline  it for completeness.
Consider the collection of all  $k$-tuples $V^\omega=(V^\omega_1,\ldots, V^\omega_k)$
of doubly $\Lambda$-commuting row isometries $V^\omega_i=(V^\omega_{i,1}\cdots V^\omega_{i,n_i}]$ such that $\Delta_{V^\omega}(I)=0$ and  $C^*(\{V^\omega_{i,s}\})$ is irreducible. We define the $k$-tuple
$\widetilde V=(\widetilde V_1,\ldots, \widetilde V_k)$ with
 $\widetilde V_i=[\widetilde V_{i,1}\cdots \widetilde V_{i,n_i}]$ by setting

 $$
 \widetilde V_{i,s}:=\bigoplus_{\omega}  V^\omega_{i,s},\qquad i\in \{1,\ldots, k\}, s\in \{1,\ldots, n_i\}.
 $$
  One can prove that  $C^*(\{\widetilde V_{i,s}\})$ is isomorphic to $C_\Delta^*(\{V_{i,s}\})$,    the universal algebra  generated by a $k$-tuple $V=(V_1,\ldots, V_k)$ of doubly $\Lambda$-commuting row isometries such
$\Delta_V(I)=0$.   Let $V=(V_1,\ldots, V_k)$  be another $k$-tuple of doubly $\Lambda$-commuting  row isometries such that $\Delta_V(I)=0$. It is enough to show that there is a surjective $*$-homomorphism
$\varphi:C^*(\{\widetilde V_{i,s}\})\to C^*(\{ V_{i,s}\})$ such that
$\varphi(\widetilde V_{i,s})= V_{i,s}$ for any $ i\in \{1,\ldots, k\}$ and $ s\in \{1,\ldots, n_i\}$.
For this, it suffices to show that
\begin{equation}
\label{ine-vN3}
\|p(\{V_{i,s}\}, \{V_{i,s}^*\})\|\leq  \|p(\{\widetilde V_{i,s}\}, \{\widetilde V_{i,s}^*\})\|
\end{equation}
for any noncommutative polynomial  in $V_{i,s}$ and $V_{i,s}^*$. The proof is similar to that of inequality \eqref{ine-vN}.

According to Theorem \ref{compact}, the closed
two-sided ideal  $\cJ_\Delta$ generated by the projection
$\Delta_S(I):=\prod_{i=1}^k\left(I -\sum_{s=1}^{n_i} S_{i,s}S_{i,s}^*\right)$ in the $C^*$-algebra $C^*(\{S_{i,s}\})$ coincides with the ideal of all  compact operators in $B(\ell^2(\FF_{n_1}^+\times\cdots \times \FF_{n_k}^+))$.
The next step is to prove that  $C^*(\{\widetilde V_{i,s}\})$ is isomorphic to the $C^*$-algebra
$C^*(\{ S_{i,s}\})/_{{\cJ}_\Delta}$. According to Theorem \ref{C*},  there is a $*$-representation
$\pi: C^*(\{ S_{i,s}\})\to C^*(\{ \widetilde V_{i,s}\})$ such that
$\pi(S_{i,s})=\widetilde V_{i,s}$  for any $ i\in \{1,\ldots, k\}$ and $ s\in \{1,\ldots, n_i\}$.
Since $\pi(\cJ_\Delta)=0$, $\pi$ induces  a $*$-representation
$\psi:C^*(\{ S_{i,s}\})/_{{\cJ}_\Delta} \to C^*(\{\widetilde V_{i,s}\}) $ such that
 $$
\psi\left(p(\{S_{i,s}\}, \{S_{i,s}^*\})+\cJ_\Delta\right)=p(\{\widetilde V_{i,s}\}, \{\widetilde V_{i,s}^*\}).
$$
Therefore, $\psi$ is surjective and
\begin{equation}
\label{ine-vN4}
\|p(\{\widetilde V_{i,s}\}, \{\widetilde V_{i,s}^*\})\|\leq \|p(\{S_{i,s}\}, \{S_{i,s}^*\})+\cJ_\Delta\|.
\end{equation}
On the other hand, let $q:C^*(\{ S_{i,s}\})\to C^*(\{ S_{i,s}\})/_{{\cJ}_\Delta}$ be the canonical quotient map and note that
$\widehat S=(\widehat S_1,\ldots, \widehat S_k)$ with $\widehat S_{i,s}:=[q(S_{i,1}\cdots q(S_{i,n_i})]$ is a $k$-tuple of doubly $\Lambda$-commuting  row isometries in the $C^*$-algebra  $C^*(\{ S_{i,s}\})/_{{\cJ}_\Lambda}$ such that $\Delta_{\widehat S}(I)=0$. Due to the  inequality \eqref{ine-vN3}, we have
$$
\|p(\{q(S_{i,s})\}, \{q(S_{i,s})^*\})\|\leq  \|p(\{\widetilde V_{i,s}\}, \{\widetilde V_{i,s}^*\})\|
$$
which together with  inequality \eqref{ine-vN4} implies
$$
\|p(\{\widetilde V_{i,s}\}, \{\widetilde V_{i,s}^*\})\|= \|p(\{S_{i,s}\}, \{S_{i,s}^*\})+\cJ_\Delta\|.
$$
Consequently, $\psi$ is an isomorphism of $C^*$-algebras.  The proof is complete.
\end{proof}

Now, we consider the particular case when $n_1=\cdots =n_k=1$ and $\Lambda_{ij}=\lambda_{ij}\in \TT$ with
 $\lambda_{ji}=\bar \lambda_{ij}$. If $m\in \ZZ$, we set $m^+:=\max\{m,0\}$ and $m^-:=\max\{-m,0\}$.

\begin{theorem} \label{Brehmer-extension} If $T=(T_1,\ldots, T_k)$, $T_i\in B(\cH)$,  is a $k$-tuple in the regular $\Lambda$-polyball, then there is a
Hilbert space $\widetilde\cK\supset \cH$ and a $k$-tuple  $U=(U_1,\ldots, U_k)$, $U_i\in B(\cK)$, of doubly  $\Lambda$-commuting  unitary operators such that the following statements hold.
\begin{enumerate}
\item[(i)] The $k$-tuple $U=(U_1,\ldots, U_k)$ is a unitary dilation of $T=(T_1,\ldots, T_k)$, i.e.
$$
P_\cH U_{i_1}^{p_1}\cdots U_{i_m}^{p_m}|_\cH=
T_{i_1}^{p_1}\cdots T_{i_m}^{p_m}
$$
for any $i_1,\ldots, i_m\in \{1,\ldots, k\}$ and $p_1,\ldots p_m\in \NN$, $m\in \NN$.

\item[(ii)] For any $(m_1,\ldots, m_k)\in \ZZ^k$
\begin{equation*}
P_\cH\left(  U_{1}^{m_1^-}\cdots U_{k}^{m_k^-} {U_{1}^*}^{m_1^+}\cdots {U_{k}^*}^{m_k^+}\right)|_\cH\\
=  T_{1}^{m_1^-}\cdots T_{k}^{m_k^-} {T_{1}^*}^{m_1^+}\cdots {T_{k}^*}^{m_k^+}.
\end{equation*}

\item[(iii)] The dilation $U=(U_1,\ldots, U_k)$ is minimal, i.e.
$$\widetilde\cK=\overline{\rm span}\{U_{i_1}^{p_1}\cdots U_{i_m}^{p_m} \cH: \  p_1,\ldots p_m\in \ZZ, i_1,\ldots, i_m\in \{1,\ldots, k\}, m\in \NN\}.
$$

\item[(iv)]
The minimal  unitary dilation $U=(U_1,\ldots, U_k)$ is unique up to an isomorphism.

\item[(v)]  The $k$-tuple $V=(V_1,\ldots, V_k)$, where $V_i:=U_i|_\cK$ and
$$
\cK=\overline{\rm span}\{U_{i_1}^{p_1}\cdots U_{i_m}^{p_m} \cH: \  p_1,\ldots p_m\in \NN, i_1,\ldots, i_m\in \{1,\ldots, k\}, m\in \NN\}.
$$
is a minimal doubly $\Lambda$-commuting  isometric dilation of  $T=(T_1,\ldots, T_k)$ satisfying the relation
$$T_i^*=V_i^*|_\cH,\qquad i\in \{1,\ldots, k\}.
$$
\end{enumerate}
\end{theorem}
\begin{proof}

According to Theorem \ref{iso-dil}, there is a $k$-tuple   $V=(V_1,\ldots, V_k)$ of doubly
$\Lambda$-commuting isometries on a Hilbert space $\cK\supset \cH$
 such that
\begin{equation}
\label{T*}
T_{i}^*=V_{i}^*|_\cH
\end{equation}
for any $ i\in \{1,\ldots, k\}$. Moreover, the dilation $V$ is minimal, i.e.
 $$
  \cK=\overline{\rm span}\{V_{i_1}^{p_1}\cdots V_{i_m}^{p_m} \cH: \  p_1,\ldots p_m\in \NN, i_1,\ldots, i_m\in \{1,\ldots, k\}\}.
$$
On the other hand, according to Theorem 6.2 from \cite{DJP},
there is  a $k$-tuple   $U=(U_1,\ldots, U_k)$ of doubly
$\Lambda$-commuting unitaries on a Hilbert space $\widetilde\cK\supset  \cK$
 such that
$
V_{i,s}=U_{i,s}|_{ \cK}
$
for any $ i\in \{1,\ldots, k\}$ and $ s\in \{1,\ldots, n_i\}$. Moreover,   since
 $$\cK_0:=\overline{\rm span}\{U_{i_1}^{p_1}\cdots U_{i_m}^{p_m}  \cK: \  p_1,\ldots p_m\in \ZZ, i_1,\ldots, i_m\in \{1,\ldots, k\}, m\in \NN\}
$$
is a jointly reducing subspace under the unitary operators $U_{i,s}$, we can assume  that
$\widetilde\cK=\cK_0$.
Let $(m_1,\ldots, m_1)\in \ZZ^k$ . Since
$$
 U_{1}^{m_1^-}\cdots U_{k}^{m_k^-}|_{\cK}= V_{1}^{m_1^-}\cdots V_{k}^{m_k^-}\quad \text{and}\quad
U_{k}^{m_k^+}\cdots U_{1}^{m_1^+}|_{\cK}=V_{k}^{m_k^+}\cdots V_{1}^{m_1^+},
$$
 we deduce that
\begin{equation*}
P_{\cK}\left(  {U_{1}^*}^{m_1^+}\cdots {U_{k}^*}^{m_k^+}U_{1}^{m_1^-}\cdots U_{k}^{m_k^-}\right)|_{\cK}\\
=   {V_{1}^*}^{m_1^+}\cdots {V_{k}^*}^{m_k^+}V_{1}^{m_1^-}\cdots V_{k}^{m_k^-}.
\end{equation*}
Using the $\Lambda$-commutativity  of $U=(U_1,\ldots, U_k)$ and $V=(V_1,\ldots, V_k)$, we  also obtain
\begin{equation*}
P_{\cK}\left(  U_{1}^{m_1^-}\cdots U_{k}^{m_k^-} {U_{1}^*}^{m_1^+}\cdots {U_{k}^*}^{m_k^+}\right)|_{\cK}\\
=  V_{1}^{m_1^-}\cdots V_{k}^{m_k^-} {V_{1}^*}^{m_1^+}\cdots {V_{k}^*}^{m_k^+}.
\end{equation*}
Taking the compression to the Hilbert space $\cH$ and using relation \eqref{T*}, we obtain item (ii).

Now, we prove the uniqueness of the minimal unitary dilation $U=(U_1,\ldots, U_k)$.
Note that
$$
  \widetilde\cK=\overline{\rm span}\{ U_1^{m_1^-}\cdots U_k^{m_k^-} {U_1^*}^{m_1^+}\cdots {U_k^*}^{m_k^+}:\ (m_1,\ldots, m_k)\in \ZZ_k \}.
$$
Let $p=(p_1,\ldots, p_k)$ and $q=(q_1,\ldots, q_k)$ be in $\ZZ^k$ and set $m_i:=q_i-p_i$ for any $i\in \{1,\ldots, k\}$.
If $x,y\in \cH$,  we can use the $\Lambda$-commutativity  of $(U_1,\ldots, U_k)$ and item (ii) to deduce that
\begin{equation*}
\begin{split}
\left< U_1^{p_1}\cdots U_k^{p_k}x, U_1^{q_1}\cdots U_k^{q_k}y\right>
&=c_{p,q}\left< U_1^{p_1-q_1}\cdots U_k^{p_k-q_k}x,y\right>\\
&=c_{p,q}\left<U_{1}^{m_1^-}\cdots U_{k}^{m_k^-} {U_{1}^*}^{m_1^+}\cdots {U_{k}^*}^{m_k^+}x,y\right> \\
&=
c_{p,q}\left<T_{1}^{m_1^-}\cdots T_{k}^{m_k^-} {T_{1}^*}^{m_1^+}\cdots {T_{k}^*}^{m_k^+}x,y\right>
\end{split}
\end{equation*}
for some constant $c_{p,q}\in \TT$. Therefore, the inner product  $\left< U_1^{p_1}\cdots U_k^{p_k}x, U_1^{q_1}\cdots U_k^{q_k}y\right>$  does not depend on the particular choice of the dilation $U$. Consequently, the inner product of two finite sums
$$
\sum_{{(p_1,\ldots, p_k)\in \ZZ^k}\atop{|p_1|+\cdots +|p_k|\leq N}} U_1^{p_1}\cdots U_k^{p_k} h_{(p_1,\ldots,p_k)} \quad \text{ and }\quad  \sum_{{(p_1,\ldots, p_k)\in \ZZ^k}\atop{|p_1|+\cdots +|p_k|\leq N}} U_1^{p_1}\cdots U_k^{p_k} h'_{(p_1,\ldots,p_k)}
$$
depend only on $h_{(p_1,\ldots,p_k)}, h'_{(p_1,\ldots,p_k)}\in \cH$ and $(T_1,\ldots, T_k)$, and not on the particular choice of the dilation.
Let $U=(U_1,\ldots, U_k)$ and $U'=(U'_1,\ldots, U'_k)$ be two minimal  unitary dilations  of $T=(T_1,\ldots, T_k)$ on the spaces $\widetilde \cK$ and $\widetilde \cK'$, respectively.
Define  $\Gamma: \widetilde \cK\to \widetilde \cK'$ by setting
$$
\Gamma \left( \sum_{{(p_1,\ldots, p_k)\in \ZZ^k}\atop{|p_1|+\cdots +|p_k|\leq N}} U_1^{p_1}\cdots U_k^{p_k} h_{(p_1,\ldots,p_k)} \right)
:=\sum_{{(p_1,\ldots, p_k)\in \ZZ^k}\atop{|p_1|+\cdots +|p_k|\leq N}} {U'_1}^{p_1}\cdots {U'_k}^{p_k} h_{(p_1,\ldots,p_k)}
$$
for any $N\in \{0,1,\ldots\}$ and  $h_{(p_1,\ldots,p_k)}\in \cH$.  Note  that  $\Gamma$ is an isometry. Due to the minimality of $U$ and $U'$, $\Gamma$ can be extended by continuity to a unitary operator from $\widetilde \cK$ to $\widetilde \cK'$. Note that $\Gamma h=h$ for any $h\in \cH$ and
$$
\Gamma \left(U_i \sum_{{(p_1,\ldots, p_k)\in \ZZ^k}\atop{|p_1|+\cdots +|p_k|\leq N}} U_1^{p_1}\cdots U_k^{p_k} h_{(p_1,\ldots,p_k)} \right)
:= U'_i \Gamma \left(\sum_{{(p_1,\ldots, p_k)\in \ZZ^k}\atop{|p_1|+\cdots +|p_k|\leq N}} U_1^{p_1}\cdots U_k^{p_k} h_{(p_1,\ldots,p_k)} \right).
$$
Consequently, $\Gamma U_i=U_i'\Gamma$ for any $i\in \{1,\ldots, k\}$.
The proof is complete.
\end{proof}

We remark that if $\lambda_{ij}=1$ for any $i,j\in \{1,\ldots, k\}$ in Theorem \ref{Brehmer-extension},  we recover Brehmer's result \cite{Br} (see also \cite{SzFBK-book}).

      \bigskip

       %

      \end{document}